\RequirePackage[l2tabu, orthodox]{nag} 
\documentclass[a4paper,reqno,11pt]{article} 
\usepackage[stretch=10]{microtype} 
\usepackage[german,french,UKenglish]{babel}
\usepackage[T1]{fontenc} 
\usepackage[utf8]{inputenc} 
\usepackage{csquotes} 
\usepackage{letltxmacro} 
\usepackage{amsmath,amsfonts,amssymb,amsthm,mathrsfs}
\usepackage[margin=1.3in]{geometry}
\usepackage{setspace}
\usepackage{enumitem}
\usepackage{calc}
\usepackage{makebox}
\usepackage{xspace}
\usepackage{xifthen}
\usepackage{titlesec} 
\usepackage{tikz}
\usepackage{tikz-cd} 
\usetikzlibrary{calc}
\usetikzlibrary{babel}
\usetikzlibrary{matrix,arrows}
\usetikzlibrary{decorations.markings}
\usepackage[lf]{Baskervaldx} 
\usepackage[bigdelims,vvarbb]{newtxmath} 
\usepackage[cal=boondoxo]{mathalfa} 

\usepackage{verbatim} 
\usepackage{aliascnt} 
\usepackage[unicode,colorlinks=true,urlcolor=blue!70!black,citecolor=blue!60!black,linkcolor=blue!60!black,pdfauthor={Remy van Dobben de Bruyn},pdftitle={A condensed proof of the pro-étale and étale exodromy theorems}]{hyperref} 

\titlespacing{\section}{0pt}{12pt plus 6pt minus 4pt}{0pt plus 4pt minus 2pt}
\titlespacing{\subsection}{0pt}{12pt plus 6pt minus 4pt}{0pt plus 4pt minus 2pt}
\titlespacing{\subsubsection}{0pt}{0pt plus 4pt minus 2pt}{0pt plus 3pt minus 2pt}

\titleformat{\section}[block]{\Large\bfseries\scshape\filcenter}{\thesection.}{1ex}{}
\titleformat{\subsection}{\large\scshape\filcenter}{\thesubsection}{1ex}{}

\numberwithin{equation}{subsection}
\newtheoremstyle{thms}{1em}{0pt}{\itshape}{}{\itshape\bfseries}{. ----}{ }{\thmname{#1}\xspace\thmnumber{#2}\thmnote{ \normalfont(#3)}}
\theoremstyle{thms}
\newaliascnt{Thm}{equation}							
\newtheorem{Thm}[Thm]{Theorem}
\aliascntresetthe{Thm}

\newaliascnt{Prop}{equation}						
\newtheorem{Prop}[Prop]{Proposition}
\aliascntresetthe{Prop}

\newaliascnt{Lemma}{equation}						
\newtheorem{Lemma}[Lemma]{Lemma}
\aliascntresetthe{Lemma}

\newaliascnt{Cor}{equation}							
\newtheorem{Cor}[Cor]{Corollary}
\aliascntresetthe{Cor}

\newaliascnt{Conj}{equation}						

\aliascntresetthe{Conj}

\newaliascnt{Question}{equation}					

\aliascntresetthe{Question}

\newtheoremstyle{defs}{1em}{0pt}{}{}{\itshape\bfseries}{. ----}{ }{\thmname{#1} \thmnumber{#2}}
\theoremstyle{defs}
\newaliascnt{Rmk}{equation}							
\newtheorem{Rmk}[Rmk]{Remark}
\aliascntresetthe{Rmk}

\newaliascnt{Fact}{equation}						

\aliascntresetthe{Fact}

\newaliascnt{Def}{equation}							
\newtheorem{Def}[Def]{Definition}
\aliascntresetthe{Def}

\newaliascnt{Ex}{equation}							
\newtheorem{Ex}[Ex]{Example}
\aliascntresetthe{Ex}

\newaliascnt{Con}{equation}							

\aliascntresetthe{Con}

\newaliascnt{Not}{equation}							

\aliascntresetthe{Not}

\newaliascnt{Setup}{equation}						

\aliascntresetthe{Setup}

\newaliascnt{Picture}{equation}						

\aliascntresetthe{Picture}

\newtheoremstyle{par}{1em}{0pt}{}{}{\itshape\bfseries}{. ----}{ }{\thmnumber{#2}}
\theoremstyle{par}
\newaliascnt{Par}{equation}							
\newtheorem{Par}[Par]{Paragraph}
\aliascntresetthe{Par}

\theoremstyle{thms}
\newtheorem{thm}{Theorem}
\newtheorem*{thm*}{Theorem}

\newtheorem*{lemma*}{Lemma}

\numberwithin{equation}{section}						
\theoremstyle{thms}
\newaliascnt{BThm}{equation}							

\aliascntresetthe{BThm}

\newaliascnt{BProp}{equation}							
\newtheorem{BProp}[BProp]{Proposition}
\aliascntresetthe{BProp}

\newaliascnt{BLemma}{equation}							
\newtheorem{BLemma}[BLemma]{Lemma}
\aliascntresetthe{BLemma}

\newaliascnt{BCor}{equation}							
\newtheorem{BCor}[BCor]{Corollary}
\aliascntresetthe{BCor}

\theoremstyle{defs}
\newaliascnt{BRmk}{equation}							
\newtheorem{BRmk}[BRmk]{Remark}
\aliascntresetthe{BRmk}

\newaliascnt{BDef}{equation}							
\newtheorem{BDef}[BDef]{Definition}
\aliascntresetthe{BDef}

\theoremstyle{par}
\newaliascnt{BPar}{equation}							
\newtheorem{BPar}[BPar]{Paragraph}
\aliascntresetthe{BPar}

\numberwithin{equation}{subsection}						

\LetLtxMacro\oldproof\proof	
\renewcommand{\proof}[1][Proof]{\oldproof[#1]\unskip}
\LetLtxMacro\oldendproof\endproof
\renewcommand{\endproof}{\oldendproof\unskip}

\newenvironment{itemize*} 
  {\begin{itemize}
    \setlength{\itemsep}{1em}
    \setlength{\parskip}{-1em}
    \setlength{\topsep}{0pt}
    \setlength{\partopsep}{0pt}}
  {\end{itemize}}
  
\newenvironment{enumerate*}
  {\begin{enumerate}
    \setlength{\itemsep}{1em}
    \setlength{\parskip}{-1em}
    \setlength{\topsep}{0pt}
    \setlength{\partopsep}{0pt}}
  {\end{enumerate}}
  
\setlist{itemsep=0em,topsep=0cm,partopsep=0em,parsep=\lineskip}

\setlist[enumerate]{label=\normalfont(\alph*), align=left, leftmargin=3em, labelwidth=1em, itemindent=0pt, listparindent=0pt, labelindent=1em, labelsep=*}

\setlist[itemize]{leftmargin=1.3em}

\renewcommand{\Top}{\operatorname{Top}}
\DeclareMathOperator{\SpecTop}{SpecTop}
\DeclareMathOperator{\LTop}{LTop}
\DeclareMathOperator{\RTop}{RTop}
\DeclareMathOperator{\PreTop}{PreTop}
\DeclareMathOperator{\LFib}{LFib}
\DeclareMathOperator{\Set}{Set}
\DeclareMathOperator{\Fin}{Fin}
\newcommand{\fin}{{\operatorname{fin}}}
\DeclareMathOperator{\Ab}{Ab}
\DeclareMathOperator{\ProFin}{ProFin}
\DeclareMathOperator{\Stone}{Stone}
\DeclareMathOperator{\ExtrDisc}{ExtrDisc}
\DeclareMathOperator{\Et}{\acute Et}
\DeclareMathOperator{\AffEt}{Aff\acute Et}
\DeclareMathOperator{\ProEt}{Pro\acute Et}
\DeclareMathOperator{\AffProEt}{AffPro\acute Et}
\DeclareMathOperator{\Sch}{Sch}
\DeclareMathOperator{\Aff}{Aff}
\DeclareMathOperator{\Cat}{Cat}
\DeclareMathOperator{\ob}{ob}
\DeclareMathOperator{\Core}{Core}
\DeclareMathOperator{\Pts}{Pts}
\newcommand{\PTS}{\mathbf{Pts}}
\DeclareMathOperator{\BoolAlg}{BoolAlg}
\DeclareMathOperator{\CRing}{CRing}
\newcommand{\PrL}{\operatorname{Pr}^{\operatorname{L}}}
\newcommand{\PrR}{\operatorname{Pr}^{\operatorname{R}}}
\DeclareMathOperator{\size}{size}
\newcommand{\coh}{{\operatorname{coh}}}
\newcommand{\hyp}{{\operatorname{hyp}}}
\DeclareMathOperator{\HypCov}{HypCov}
\newcommand{\loccoh}{{\operatorname{loc\ \!coh}}}
\newcommand{\wl}{{\operatorname{wl}}}
\newcommand{\wsl}{{\operatorname{wsl}}}
\newcommand{\Zar}{{\operatorname{Zar}}}
\newcommand{\et}{{\operatorname{\acute et}}}
\newcommand{\proet}{{\operatorname{pro\acute et}}}
\newcommand{\affproet}{{\operatorname{affpro\acute et}}}
\newcommand{\cons}[1][]{_{\ifthenelse{\equal{#1}{}}{{\operatorname{cons}}}{#1{\operatorname{-cons}}}}}
\newcommand{\lcons}{_{\operatorname{lcons}}}
\newcommand{\lc}{_{\operatorname{lc}}}
\DeclareMathOperator{\Pro}{Pro}
\DeclareMathOperator{\Ind}{Ind}
\DeclareMathOperator{\Hom}{Hom}
\DeclareMathOperator{\Hens}{Hens}
\DeclareMathOperator{\Fun}{Fun}
\DeclareMathOperator{\Map}{Map}

\DeclareMathOperator{\Cont}{Cont}
\DeclareMathOperator{\Sh}{Sh}
\DeclareMathOperator{\PSh}{PSh}
\DeclareMathOperator{\CSh}{CSh}
\DeclareMathOperator{\Open}{Open}
\DeclareMathOperator{\Clopen}{Clopen}
\DeclareMathOperator*{\colim}{colim}
\DeclareMathOperator{\Spec}{Spec}
\DeclareMathOperator{\Gal}{Gal}
\newcommand{\GAL}{\mathbf{Gal}}
\DeclareMathOperator{\Cond}{Cond}
\newcommand{\COND}{\mathbf{Cond}}
\newcommand{\cl}{{\operatorname{cl}}}
\newcommand{\op}{^{\operatorname{op}}}
\newcommand{\cts}{^{\operatorname{cts}}}
\newcommand{\lex}{{\operatorname{lex}}}
\DeclareMathOperator{\ev}{ev}

\DeclareMathOperator{\post}{post}
\DeclareMathOperator{\acyc}{acyc}

\newcommand{\punct}[1]{\makebox[0pt][l]{\,#1}} 

\makeatletter
\newcommand{\hathatInternal}[2]{%
	\begingroup%
	\let\macc@kerna\z@%
	\let\macc@kernb\z@%
	\let\macc@nucleus\@empty%
	\widehat{\raisebox{#2}{\vphantom{\ensuremath{#1}}}\smash{\widehat{#1}}}%
	\endgroup%
}
\makeatother
\newcommand{\widehatt}[1]{\mathchoice
	{\hathatInternal{#1}{.4ex}}
	{\hathatInternal{#1}{.3ex}}
	{\hathatInternal{#1}{-1.5pt}}
	{\hathatInternal{#1}{1pt}}
}

\tikzcdset{arrow style=math font}

\let\OLDthebibliography\thebibliography
\renewcommand\thebibliography[1]{
  \OLDthebibliography{#1}
  \setlength{\parskip}{0pt}
  \setlength{\itemsep}{0pt plus 0.1em}
}

\begin{document}

\renewcommand{\sectionautorefname}{Section}
\renewcommand{\subsectionautorefname}{Subsection}		

\begin{center}
\vspace*{-3em}
\noindent\makebox[\linewidth]{\rule{15cm}{0.4pt}}
\vspace{-.5em}

{\LARGE{\textbf{A condensed proof of the pro-\'etale \\[.2em] and \'etale exodromy theorems}}}

\vspace{.5em}

{\textsc{Remy van Dobben de Bruyn}}

\vspace*{.25em}
\leavevmode
\noindent\makebox[\linewidth]{\rule{11cm}{0.4pt}}
\vspace{1em}
\end{center}

\renewcommand{\abstractname}{\small\bfseries\scshape Abstract}

\begin{abstract}\noindent
The exodromy correspondence of Barwick, Glasman, and Haine computes constructible sheaves of spaces on a scheme $X$ as an $\infty$-category of continuous functors from the profinite category $\operatorname{Gal}(X)$. Viewing $\operatorname{Gal}(X)$ instead as a condensed category, this was extended by Wolf to an exodromy correspondence for pro-\'etale sheaves. Using the condensed perspective from the outset, we give a quick and self-contained proof of the pro-\'etale exodromy theorem. This is used to extract an exodromy theorem for (Postnikov complete) \'etale sheaves that does not yet appear in the literature, which is closely related to Lurie's work on ultracategories. Finally, we use this to give a new proof of the constructible \'etale exodromy correspondence of Barwick, Glasman, and Haine. Without additional effort, our method removes the qcqs hypotheses on the schemes, and gives versions for sheaves with coefficients in more general $\infty$-categories. Finally, we refine the methods to obtain a $\kappa$-condensed statement for any uncountable cardinal $\kappa$ such that $\kappa > \lvert \mathcal O_X(U) \rvert$ for every affine open $U \subseteq X$.
\end{abstract}\vspace{.2em}


\phantomsection
\section*{Introduction}
\addcontentsline{toc}{section}{Introduction}

\phantomsection
\subsection*{Background}
\addcontentsline{toc}{subsection}{Background}
If $X$ is a reasonable topological space, the \emph{monodromy correspondence} gives an equivalence between locally constant sheaves on $X$ with values in a presentable $\infty$-category $\mathscr E$ and the functor category $\Fun(\Pi_\infty(X),\mathscr E)$. For a scheme $X$, a similar description is given for locally constant sheaves of \emph{finite sets} (or, more generally, \emph{$\pi$-finite spaces}) using the \'etale fundamental group $\pi_1^\et(X)$ \cite[Exp.~V]{SGA1}, \cite[Exp.~X, \S6]{SGA3II}, \cite[\S7]{BhattScholze} (resp.\ \'etale homotopy type $\Pi_\infty^\et(X)$ \cite{ArtinMazur,Friedlander,Hoyois}).

This categorical description of locally constant sheaves on a topological space was extended to constructible sheaves by MacPherson (unpublished, first presented in Treumann's thesis \cite{TreumannThesis,Treumann}), and Lurie obtained an $\infty$-categorical version \cite[Thm.~A.9.3]{LurieHA}. An \'etale version was proven by Barwick, Glasman, and Haine \cite[Thm.~12.1.6]{BGH}, who coined the term \emph{exodromy correspondence}:

\begin{thm*}[Barwick--Glasman--Haine]
Let $X$ be a qcqs scheme. Then there is an exodromy correspondence
\[
\Sh\cons(X_\et,\mathcal S) \xrightarrow\sim \Fun\cts(\Gal(X),\mathcal S_\pi)
\]
between the $\infty$-category of constructible sheaves of spaces on $X$ and the $\infty$-category of continuous functors from the Galois category of $X$ to $\pi$-finite spaces.
\end{thm*}

Here, the \emph{Galois category} $\Gal(X)$ is a pro-object in $\pi$-finite layered $\infty$-categories whose `materialisation' (limit in $\Cat_\infty$) is the $1$-category $\Pts^*(X_\et) = \Fun^*(\Sh(X_\et),\Set)$ of points of the \'etale topos $\Sh(X_\et)$, and \emph{continuous functors} means functors in $\Pro(\Cat_\infty)$, where~$\mathcal S_\pi$ is treated as the constant pro-system.

Any profinite $\infty$-category, such as $\Gal(X)$, can be viewed as a \emph{consensed} (or \emph{pyknotic}) $\infty$-category by taking the limit in $\Cond(\Cat_\infty)$ of the diagram of `discrete' \makebox{$\infty$-categories,} as in \cite[\S13.5]{BGH}. Writing $\GAL(X)$ for the condensed $\infty$-category defined by $\Gal(X)$, Wolf \cite[Cor.~1.2]{Wolf} extended the theorem above to a pro-\'etale exodromy theorem:

\begin{thm*}[Wolf]
Let $X$ be a qcqs scheme. Then there is an exodromy correspondence
\[
\Sh^\hyp(X_\proet,\mathcal S) \xrightarrow\sim \Fun\cts(\GAL(X),\COND(\mathcal S)).
\]
{\normalfont Here $\COND(\mathcal S)$ is the condensed $\infty$-category taking a profinite space $S$ to $\Cond(\mathcal S)_{/S}$ (equivalently, $\Sh(\ProFin_{/S},\mathcal S)$), and \emph{continuous functors} are functors of condensed $\infty$-categories.} 
\end{thm*}

\phantomsection
\subsection*{Main results}
\addcontentsline{toc}{subsection}{Main results}
In this paper, we take up a condensed perspective from the outset to give independent (and, in the author's opinion, more geometric, more direct, and more uniform) proofs of the theorems of Barwick--Glasman--Haine and Wolf, as well as a number of generalisations and variants.

Using a direct definition of the condensed Galois category $\GAL(X)$ (see \hyperref[Sec what is Galois]{\textsc{Construction of the Galois category}} further down in this introduction), we obtain a quick geometric proof of the following generalisation of Wolf's theorem:

\begin{thm}[\autoref{Thm proetale exodromy}]\label{thm 1}
Let $X$ be a scheme, and let $\mathscr E$ be an $\infty$-category that has small limits. Then there is an exodromy correspondence
\[
\mathcal{Ex} \colon \Sh^\hyp(X_\proet,\mathscr E) \xrightarrow\sim \Fun\cts(\GAL(X),\COND(\mathscr E)),
\]
where $\COND(\mathscr E)$ is the condensed $\infty$-category taking a profinite space $S$ to $\Sh(\ProFin_{/S},\mathscr E)$.
\end{thm}

To extract the constructible \'etale exodromy from the one for pro-\'etale hypersheaves, we first prove that the pro-\'etale exodromy correspondence restricts to a statement for Postnikov complete \'etale sheaves:

\begin{thm}[\autoref{Thm etale exodromy}]\label{thm 2}
Let $X$ be a scheme, and let $\mathscr E$ be a compactly assembled presentable $\infty$-category. Then the pro-\'etale exodromy correspondence restricts to an \'etale exodromy correspondence
\[
\mathcal{Ex}^\et \colon \Sh^{\post}(X_\et,\mathscr E) \xrightarrow\sim \Fun\cts(\GAL(X),\underline{\mathscr E}),
\]
where $\underline{\mathscr E}$ is the condensed $\infty$-category $S \mapsto \Sh(S,\mathscr E)$.
\end{thm}

Note that $\underline{\mathscr E}$ is a condensed $\infty$-category by \cite[Cor.~2.8]{HaineDescent} (this relies on the hypothesis that $\mathscr E$ is compactly assembled). Even for $\mathscr E = \mathcal S$, the result above does not appear in the literature, although Clark Barwick had already indicated awareness of a version of this result before the conception of the present paper. When $\mathscr E = \Set$ and $X$ is qcqs, this is almost the same statement as Lurie's version of Makkai's strong conceptual completeness theorem \cite[Thm.~0.0.6]{LurieUltra}. However, there is a small but important technical difference, which we comment on in \autoref{Rmk difference with Lurie}.

Finally, we obtain a new proof of the constructible \'etale exodromy theorem of Barwick--Glasman--Haine:

\begin{thm}[\autoref{Thm constructible exodromy}]\label{thm 3}
Let $X$ be a scheme. Then the \'etale exodromy correspondence restricts to a constructible \'etale exodromy correspondence
\[
\mathcal{Ex}^\et\cons \colon \Sh\cons(X_\et,\mathcal S) \xrightarrow\sim \Fun\cts(\GAL(X),\underline{\mathcal S}{}\cons),
\]
where $\underline{\mathcal S}{}\cons$ denotes the condensed $\infty$-category $S \mapsto \Sh\cons(S,\mathcal S)$.
\end{thm}

In particular, there are no qcqs hypotheses in \autoref{thm 1}, \autoref{thm 2}, and \autoref{thm 3}. This would be impossible to achieve via the profinite methods of \cite{BGH}, since the Galois category is only `locally profinite' in general. On the other hand, using the condensed point of view, this extension can easily be bootstrapped from the known results in the qcqs case by observing that $\GAL(X)$ is a Zariski cosheaf of condensed $\infty$-categories (our proofs do not require treating the qcqs case separately). In fact, using \autoref{thm 1}, we obtain the following Van Kampen type result:

\begin{thm}[\autoref{Thm van Kampen}]\label{thm 4}
The functor $\GAL \colon \Sch \to \Cond(\Cat_\infty)$ is a pro-\'etale hypercosheaf.
\end{thm}

This is similar to two Van Kampen statements for $\GAL(X)$ and $\Pi_\infty^{\proet}(X)$ proved in \cite{HHLMMW}; see \ref{Par cosheaf results}.

While our results are stronger than those of \cite{BGH} in some ways, they are weaker in others. We only obtain $\GAL(X)$ as a condensed category, and do not recover the profinite structure when~$X$ is a qcqs scheme. Moreover, we do not get an exodromy equivalence for constructible \'etale sheaves with respect to a \emph{fixed} stratification (however, a quick geometric proof of a 1-categorical constructible \'etale exodromy correspondence with respect to a fixed stratification was obtained in \cite{vDdBExodromy}). We also do not pursue the question of classifying the $\infty$-topoi satisfying such a profinite exodromy correspondence, which was carried out in \cite[Thm.~9.3.1]{BGH} in the form of an $\infty$-categorical Hochster duality.

\phantomsection
\subsection*{Construction of the Galois category}\label{Sec what is Galois}
\addcontentsline{toc}{subsection}{Construction of the Galois category}
Given a qcqs scheme $X$, Barwick--Glasman--Haine construct the Galois category $\Gal(X)$ as a pro-object in (layered) $\infty$-categories \cite[Def.~12.1.3]{BGH}. As a functor from $\pi$-finite layered $\infty$-categories to spaces, it is given by $\Pi \mapsto \Map_{\RTop_\infty^\coh}(\Sh(X_\et,\mathcal S),\Fun(\Pi,\mathcal S))$ \cite[10.1.1]{BGH}. When $\Pi = BG$ for a finite group $G$, coherent geometric morphisms $\Sh(X_\et,\mathcal S) \to \Fun(G,\mathcal S)$ classify $G\op$-torsors over $X$. In general, there is a similar description in terms of filtering functors $\Pi\op \to \Sh\cons(X_\et,\mathcal S)$, which are roughly `$\Pi\op$-torsors' (see \cite[Ch.~VII, Def.~8.1 and Cor.~9.2]{MacLaneMoerdijk} for the statement for 1-topoi).

Such objects have not been classically studied in algebraic geometry (although this could be interesting to work out). For a working algebraic geometer, to get a more concrete grip on the exodromy correspondence, it is desirable to have a more geometric description of the profinite category $\Gal(X)$. This is the starting point of this paper.

As a first step, \cite[12.1.5]{BGH} gives a concrete description of the underlying category of the profinite category $\Gal(X)$ (obtained by taking the limit in $\Cat_\infty$): this is given by the category of points $\Pts^*(X_\et)$ of the \'etale topos, which by \cite[Exp.~VIII, Thm.~7.9]{SGA4II} is equivalent to the opposite of the category of strictly Henselian local rings with a weakly \'etale map to $X$. It is classical that an \'etale sheaf $\mathscr F$ comes equipped with cospecialisation maps $\mathscr F_{\bar x} \to \mathscr F_{\bar y}$ whenever $x \in \overline{\{y\}}$ \cite[Exp.~VIII, 7.7]{SGA4II}. Thus, ignoring continuity issues, the \'etale exodromy functor of \autoref{thm 2} and \autoref{thm 3} should take $\mathscr F$ to the functor $\Pts^*(X_\et) \to \Set$ taking $\bar x$ to the stalk $\mathscr F_{\bar x}$, and taking a map $\Spec \mathcal O_{X,\bar y}^{\operatorname{sh}} \to \Spec \mathcal O_{X,\bar x}^{\operatorname{sh}}$ to the cospecialisation $\mathscr F_{\bar x} \to \mathscr F_{\bar y}$.

To prove an exodromy correspondence via this approach, what is missing is a description of the profinite topology on $\Gal(X)$. We note that $\Gal(X)$ becomes a category \emph{internal} to Stone spaces, so we need to describe a profinite set of objects and a profinite set of morphisms with a continuous composition operation. While \cite[0.0.2]{BGH} explains the profinite topology on $\Hom(\bar x,\bar y)$ for geometric points $\bar x, \bar y \to X$, it is not clear how to describe a Stone space of all objects or of all morphisms in $\Gal(X)$.

The reason this is difficult to do is that the categorically/homotopically meaningful \makebox{$(2,1)$-category} of Stone categories is a \emph{localisation} of category objects in Stone spaces along the collection of internal functors that are (internally) fully faithful and essentially surjective (see \cite[Thm.~1.2.13]{LurieGoodwillie} for a version of this statement in an $\infty$-topos). Since surjections of Stone spaces do not always split, this means that such representatives are not unique, and to do this for $\Gal(X)$ relies on making choices --- similar to how a choice of pro-\'etale hypercover by w-strictly local schemes leads to a model for the pro-\'etale homotopy type, as suggested in \cite[Rmk.~4.2.9]{BhattScholze} and worked out in \cite[Appendix~A]{HemoRicharzScholbach}, \cite[Prop.~7.3.1.11]{WolfThesis}, and \cite[Rmk.~3.20]{HHLMMW}.

Instead of addressing this question directly, we opt to work with condensed categories in favour of Stone categories, since internal category theory is better understood in an ($\infty$-)topos. We replace the profinite definition of $\Gal(X)$ by a condensed one, defined as
\begin{align*}
\GAL(X) \colon \ProFin\op &\to \Cat \\
S &\mapsto \Fun^*_{\loccoh}(\Sh(X_\et),\Sh(S)).
\end{align*}
Such a definition was already suggested by \cite[Ex.~13.5.4]{BGH} (where the locally coherent condition is missing). A detailed comparison between the profinite and condensed Galois categories appeared in \cite[Prop.~7.1.2.8]{WolfThesis} (but see \autoref{Rmk Gal(X) profinite}), which was partially motivated by the present paper (see \hyperref[Sec acknowledgements]{\textsc{Acknowledgements}}) as well as by \cite{HHLMMW}. While this paper was in preparation, the above definition was also used as the definition in \cite[Def.~3.24]{HHLMMW}, \cite[Rec.~2.2.30]{MairThesis}, and \cite[Def.~2.24]{Mair}. In particular, while the results of the present paper are separate from those in \cite{HHLMMW}, the underlying philosophy is very similar.

This condensed Galois category generalises the observation above that the $*$-points of $\GAL(X)$ are given by points of the \'etale topos $\Sh(X_\et)$. To prove the exodromy theorem, we also need a generalisation of the classification of points of the \'etale topos via strictly Henselian schemes. Extending \cite[Exp.~VIII, Thm.~7.9]{SGA4II}, we prove the following:

\begin{thm}\label{Thm technical}
Let $X$ be a scheme, and let $\Stone_{/X}^{\loccoh} \to \ProFin$ be the Grothendieck construction of the functor $\GAL(X) \colon \ProFin\op \to \Cat$. Then $\Stone_{/X}^{\loccoh}$ is canonically equivalent to the category $\AffProEt_{/X}^\wsl$ of w-strictly local schemes that are limits of affine schemes in $\Et_{/X}$.
\end{thm}

This is proved in \autoref{Thm equivalence} and \autoref{Cor loccoh}. An analogous result for arbitrary geometric morphisms from profinite sets to a coherent topos $\mathscr X$ is also the key ingredient for Lurie's proof of Makkai's strong conceptual completeness theorem \cite[Thm.~6.3.14]{LurieUltra}. The statement of \autoref{Thm equivalence} in itself can be deduced from Lurie's theorem, but we give a geometric proof in order to get the more refined versions of \autoref{Cor loccoh} and \autoref{Cor loccoh size}. In the special case where $X$ is affine and restricting to w-projective objects (see \autoref{Rmk w-projective}), a result similar to \autoref{Cor loccoh} was deduced from Lurie's version in \cite[Thm.~3.38]{Mair} around the same time the present paper appeared.

\phantomsection
\subsection*{On cardinality bounds}\label{Sec on cardinality}
\addcontentsline{toc}{subsection}{On cardinality bounds}
This paper heavily uses the language of condensed mathematics. Working with sheaves on the category of profinite sets inevitably leads to set-theoretic issues. We resolve these by restricting to profinite sets that are small with respect to some sufficiently large cardinal~$\kappa$. We prove a more precise version of \autoref{thm 1} taking these cardinals into account; see \autoref{Thm proetale exodromy} and \autoref{Def size scheme} for a precise formulation. We do a little extra work to avoid additional hypotheses on the cardinal $\kappa$: it need neither be regular nor a strong limit cardinal. In particular, we do not have access to enough w-projective objects in the pro-\'etale site. For example, this allows us to show that a version of \autoref{thm 1} holds in the \emph{light} condensed setting as long as $\mathscr O_X(U)$ is countable for every affine open $U \subseteq X$.

While it appears that the $\kappa$-condensed category $\GAL_\kappa(X)$ of a scheme $X$ depends on a choice of auxiliary cardinal $\kappa$, we prove an accessibility result in \hyperref[Sec change of cardinal]{\S4.2} for the functor $\GAL \colon \ProFin\op \to \Cat$, which implies that knowing $\GAL_\kappa(X)$ for a single (sufficiently large) value of $\kappa$ determines $\GAL_\lambda(X)$ for every (sufficiently large) cardinal $\lambda$.

\phantomsection
\subsection*{Outline of the proofs}
\addcontentsline{toc}{subsection}{Outline of the proofs}
We start by proving \autoref{Thm technical} by a similar strategy to \cite[Exp.~VIII, Thm.~7.9]{SGA4II}. Given a locally coherent geometric morphism $s_* \colon \Sh(S) \to \Sh(X_\et)$ from a profinite set~$S$, we construct a Henselisation $X_{(s)} \in \AffProEt_{/X}^\wsl$ of $X$ along $s$, which is a w-strictly local affine scheme whose profinite set of connected components is canonically identified with~$S$. Then the subscheme of closed points $Z = X_{(s)}^\cl$ is zero-dimensional with separably algebraically closed residue fields, so $\Sh(Z_\et) \simeq \Sh(S)$. We prove that $s_*$ is isomorphic to the pushforward $\Sh(Z_\et) \to \Sh(X_\et)$, and that every w-strictly local scheme in $\AffProEt_{/X}$ arises in this way.

With \autoref{Thm technical} in place, the proof of \autoref{thm 1} proceeds as follows. The functor category $\Fun\cts(\GAL(X),\COND(\mathscr E))$ may be viewed as a full subcategory of the category of $\mathscr E$-valued presheaves on the Grothendieck construction $\Stone_{/X}^{\loccoh} \simeq \AffProEt_{/X}^\wsl$. The same holds for $\Sh^\hyp(X_\proet,\mathscr E)$ since $\AffProEt_{/X}^\wsl$ forms a basis for the pro-\'etale topology, so it only remains to see that the descent conditions imposed on both subcategories agree. The key input for this verification is \autoref{Lem w-local topology}, which shows that the pro-\'etale topology on $\AffProEt_{/X}^\wsl$ agrees with the \emph{w-local topology} (\autoref{Def w-local topology}), whose covers are jointly surjective families of w-local maps in $\AffProEt_{/X}^\wsl$. (When working with $\kappa$-condensed sets for a strong limit cardinal $\kappa$, the proof simplifies even further by restricting to w-projective schemes and extremally disconnected sets: instead of the w-local topology, we need only consider functors that take finite coproducts to products. See \autoref{Rmk strong limit} for details.)

To obtain \autoref{thm 2} and \autoref{thm 3}, we identify both sides of the \'etale and constructible exodromy correspondences as full subcategories of the \makebox{$\infty$-categories} appearing in \autoref{thm 1}. To extract $\Sh^{\post}(X_\et,\mathcal S)$ inside $\Sh^\hyp(X_\proet,\mathcal S)$, we use the criterion that truncated \'etale sheaves are exactly the truncated pro-\'etale sheaves that take filtered limits of affine pro-\'etale schemes to colimits; see \autoref{Prop etale to proetale} (this is a variant of \cite[Lem.~5.1.2]{BhattScholze}). There is a similar description for $\Sh(S,\mathcal S)$ inside $\Sh(\ProFin_{/S},\mathcal S)$, which proves \autoref{thm 2}. To recover \autoref{thm 3}, we only need to observe that constructibility descends along surjective morphisms of schemes; see \autoref{Lem reflect coherence}.

\phantomsection
\subsection*{Structure of the paper}
\addcontentsline{toc}{subsection}{Structure of the paper}
We start in \hyperref[Sec prelim]{\S1} with some preliminary results. In \hyperref[Sec cardinality]{\S1.1}, we set up the cardinality bounds used throughout the paper. This is followed up in \hyperref[Sec acyclic]{\S1.2} by a quick review of acyclic and w-strictly local schemes, as defined by Artin \cite[\S3]{Artin} and Bhatt--Scholze \cite[Def.~2.2.1]{BhattScholze} respectively. In \hyperref[Sec kappa proet]{\S1.3}, we define a $\kappa$-small pro-\'etale site $X_{\kappa\text{-}\proet}$, as well as a new ($\kappa$-small) w-local site $X_{(\kappa\text{-})\wl}$. 

Next, \hyperref[Sec condensed]{\S2} is dedicated to the proof of \autoref{Thm technical}. In \hyperref[Sec condensed category]{\S2.1}, we give the definition of $\GAL(X)$ as a condensed category of points and relate it to the profinite definition of \cite{BGH}. As in the classical computation of the category of points of $\Sh(X_\et)$ \cite[Exp.~VIII, Thm.~7.9]{SGA4II}, we want to identify w-strictly local schemes $W \in \AffProEt_{/X}$ with the functors $\Sh(X_\et) \to \Set$ they corepresent. In fact, since every clopen in $W$ corepresents such a functor, these glue together to a functor $\Sh(X_\et) \to \Sh(S)$ where $S = \pi_0(W)$. We encode this extra data by exhibiting~$W$ itself as (the global sections of) an $S$-cosheaf in $\AffProEt_{/X}$; this theory is set up in \hyperref[Sec cosheaf]{\S2.2}. In \hyperref[Sec hensel]{\S2.3}, we define the Henselisation along a geometric morphism $s_* \colon \Sh(S) \to \Sh(X_\et)$, and show that they set up an equivalence between geometric morphisms $s_* \colon \Sh(S) \to \Sh(X_\et)$ and acyclic schemes $W \in \AffProEt_{/X}$ with $\pi_0(W) \cong S$ (see \autoref{Prop equivalence S}). This is used in \hyperref[Sec Galois]{\S2.4} to prove \autoref{Thm technical} (see \autoref{Thm equivalence} and \autoref{Cor loccoh}).

In \hyperref[Sec exodromy]{\S3}, we prove the pro-\'etale and \'etale exodromy theorems. The subsections \hyperref[Sec proetale exodromy]{\S3.1}, \hyperref[Sec etale exodromy]{\S3.2}, and \hyperref[Sec constructible exodromy]{\S3.3} are dedicated to the proofs of \autoref{thm 1}, \autoref{thm 2}, and \autoref{thm 3} respectively, along the lines outlined above.

We end in \hyperref[Sec formal]{\S4} with some formal properties of the condensed category of points: in \hyperref[Sec van Kampen]{\S4.1}, we prove the Van Kampen theorem (\autoref{thm 4}), and in \hyperref[Sec change of cardinal]{\S4.2}, we prove an accessibility result for the functor $\GAL(X) \colon \ProFin\op \to \Cat$.

In \hyperref[Sec appendix A]{\textsc{Appendix A}}, we include some general results on \makebox{$\infty$-topoi} that are used in \hyperref[Sec exodromy]{\S3} and \hyperref[Sec formal]{\S4}. The main results are the construction of a hypersheaf of large $\infty$-categories $U \mapsto \Sh^\hyp(\mathscr C_{/U},\mathscr E)$ on any $\infty$-site $\mathscr C$ with values in any $\infty$-category $\mathscr E$ with small limits (\autoref{Cor hypersheaf of categories}), and a computational tool for colimits of condensed $\infty$-categories (\autoref{Lem colim left fibrations}) that is used to prove \autoref{thm 4}. Finally, \hyperref[Sec appendix B]{\textsc{Appendix B}} gives an alternative, more algebro-geometric, way to think about the classification of w-strictly local schemes of \autoref{Thm equivalence}.

\phantomsection
\subsection*{Notational conventions}
\addcontentsline{toc}{subsection}{Notational conventions}
Our notations generally follow those of \cite{SGA4I,SGA4II,SGA4III,Stacks} on the one hand and \cite{LurieHTT,LurieHA,LurieSAG} on the other.

As in \cite[\S1.2.15 and \S5.4.1]{LurieHTT}, we assume the existence of sufficiently large strongly inaccessible cardinals. We will fix an uncountable strongly inaccessible cardinal once and for all (without name), and the notion of smallness is understood with respect to this cardinal. We will write $\Set$ (resp.\ $\mathcal S$, $\Cat_\infty$) for the category (resp.\ $\infty$-category) of small sets (resp.\ (essentially) small spaces, (essentially) small $\infty$-categories \cite[\S5.4.1]{LurieHTT}). We will write $\widehat{\mathcal S}$ (resp.\ $\widehat{\Cat}_\infty$) for the $\infty$-category of large spaces (resp.\ $\infty$-categories), which are assumed small with respect to some (fixed and unnamed) larger universe. In an even larger universe, we write $\widehatt{\Cat}_\infty$ for the $\infty$-category of very large $\infty$-categories. Unless otherwise noted, all spaces will be assumed small, and all $\infty$-categories will be assumed large (i.e., in~$\widehat{\Cat}_\infty$). As noted in \hyperref[Sec on cardinality]{\textsc{On cardinality bounds}}, we have made an effort to minimise the use of universes. The only place where it is really used is to be able to talk about large $\infty$-categories like $\Sh(\mathscr C,\mathscr E)$ for $\mathscr C$ a small $\infty$-site and $\mathscr E$ a (large) $\infty$-category with small limits. A change of universe argument is used in \autoref{Prop sheaf of categories} to apply a Yoneda embedding $\mathscr E \to \PSh(\mathscr E,\widehat{\mathcal S})$, but this can be avoided if $\mathscr E$ is assumed to be locally small. (It appears to be standard practice these days not to include this hypothesis.)

If $\mathscr C$ is an $\infty$-site and $\mathscr E$ an $\infty$-category, we write $\Sh(\mathscr C,\mathscr E)$ for the $\mathscr E$-valued sheaves on $\mathscr C$ (we recall this notion in \hyperref[Sec appendix A]{\textsc{Appendix A}}). Because a substantial part of the arguments are 1-categorical, we will reserve the notation $\Sh(\mathscr C)$ for $\Sh(\mathscr C,\Set)$ when $\mathscr C$ is a 1-site.

We denote the category of topological spaces by $\Top$, and use $\RTop$ and $\LTop$ for the \makebox{$(2,1)$-category} of topoi with right adjoints $f_*$ and left adjoints $f^*$ as morphisms, respectively. Likewise, $\RTop_\infty$ and $\LTop_\infty$ denote that $(\infty,1)$-category of $\infty$-topoi with right adjoints $f_*$ and left adjoints $f^*$ as morphisms, respectively.

\phantomsection
\subsection*{Lower versus higher categories}\label{Sec lower vs higher}
\addcontentsline{toc}{subsection}{Lower versus higher categories}
The object of study in this paper is $\mathscr E$-valued sheaves on the \'etale and pro-\'etale sites of a scheme~$X$. We have separated the geometric and (1-categorical) site-theoretic arguments on $X_\et$ and $X_\proet$ (\S1--2) from the higher categorical arguments about $\mathscr E$-valued sheaves (\S3--4 and \hyperref[Sec appendix A]{\textsc{Appendix A}}). In particular, the proof of \autoref{Thm technical} is entirely 1-categorical. 

Readers who are primarily interested in $\Set$- and $\Ab$-valued sheaves should be able to make the required modifications in \S3--4 to avoid higher category theory altogether. On the other hand, we chose to present the proof with the most general coefficients where it applies. For instance, applying this general version when $\mathscr E$ is the stable $\infty$-category $D(\mathbf Z)$ gives derived statements; see \autoref{Ex D(Z)} and \autoref{Ex D(Z) left completed}.

\phantomsection
\subsection*{Acknowledgements}\label{Sec acknowledgements}
\addcontentsline{toc}{subsection}{Acknowledgements}
I would like express great gratitude to Clark Barwick and Peter Haine. This work draws a lot of inspiration from their ideas, both from the paper \cite{BGH} and from many emails and conversations over the years. I thank Peter Scholze for helpful discussions and encouragement at the early stages, and Qingyuan Bai, Tobias Barthel, Max Blans, Thomas Blom, David Carchedi, Denis-Charles Cisinski, Johan Commelin, Gijs Heuts, Ieke Moerdijk, Luca Pol, and Mauro Porta for useful discussions on higher category theory, (higher) topos theory, and exodromy. Special thanks are due to Sebastian Wolf, who has been a collaborator through a large part of the project, though he has declined to be named as a coauthor. This work owes a lot to his ideas, and could not have been completed without his input.

This project was funded by the NWO grant VI.Veni.212.204, and finished at the Max Planck Institute for Mathematics under ERC Horizon grant no.~101042990. I am grateful to the MPIM for the excellent working conditions under which this project was completed.

\section{Preliminaries}\label{Sec prelim}

\subsection{Cardinality bounds}\label{Sec cardinality}
The main results of this paper are formulated in the language of condensed mathematics \cite{ClausenScholze}, which studies sheaves on the site $\ProFin$ of profinite sets. This leads to set-theoretic problems that can be resolved in a number of ways. We choose to work with $\kappa$-condensed sets, which for us means sheaves on the subcategory $\ProFin_\kappa$ of profinite sets $X$ such that $\lvert C(X,\mathbf F_2) \rvert < \kappa$ (see \ref{Par profinite kappa} below). For instance, for $\kappa = \aleph_1$, this gives \emph{light condensed sets}.

In this section, we set up size bounds on profinite sets and schemes that we will use throughout the paper.

\begin{Rmk}\label{Rmk colimit kappa small}
Let $\kappa$ be an infinite cardinal, and let $\mathcal I$ be a small category with $\lvert \ob \mathcal I \rvert \leq \kappa$. Let $\mathscr C$ be either $\Set$ or $\CRing$, let $A \in \mathscr C$, and let $D \colon \mathcal I \to \mathscr C_{A/}$ be a diagram such that $\lvert D(i) \rvert \leq \kappa$ for all $i \in \mathcal I$. Then $\lvert \colim D \rvert \leq \kappa$.
\end{Rmk}

\begin{Par}\label{Par Ind kappa}
If $\mathscr C$ is a small $\infty$-category and $\kappa$ is an infinite cardinal, we write $\Ind^\kappa(\mathscr C)$ for the full subcategory of $\Ind(\mathscr C)$ on the $\kappa$-small filtered colimits of objects of $\mathscr C$. Dually, write $\Pro^\kappa(\mathscr C)$ for the full subcategory of $\Pro(\mathscr C)$ on $\kappa$-small cofiltered limits of objects of~$\mathscr C$. We will say that an object of $\Ind(\mathscr C)$ (resp.~$\Pro(\mathscr C)$) is \emph{$\kappa$-small} if it is in $\Ind^\kappa(\mathscr C)$ (resp.~$\Pro^\kappa(\mathscr C)$).

When $\kappa$ is regular, these notations agree with other common notations. For instance, $\Ind^\kappa(\mathscr C)$ agrees with $\Ind_\omega^\kappa(\mathscr C)$ by \cite[Tags \href{https://kerodon.net/tag/0664}{0664} and \href{https://kerodon.net/tag/063R}{063R}(a)]{Kerodon}. If $\kappa$ is uncountable, this in turn coincides with the full subcategory $\Ind(\mathscr C)^\kappa \subseteq \Ind(\mathscr C)$ of $\kappa$-compact objects \cite[Tags \href{https://kerodon.net/tag/0692}{0692} and \href{https://kerodon.net/tag/061Q}{061Q}]{Kerodon}.

Some care should be taken when $\kappa = \omega$: every finite colimit of objects of $\mathscr C$ is compact, so we get $\Ind^\omega(\mathscr C) \subseteq \Ind(\mathscr C)^\omega$. Moreover, $\Ind(\mathscr C)^\omega$ is an idempotent completion of $\mathscr C$ \cite[Tag \href{https://kerodon.net/tag/0693}{0693}]{Kerodon}, so the inclusions $\mathscr C \hookrightarrow \Ind^\omega(\mathscr C) \hookrightarrow \Ind(\mathscr C)^\omega$ are equivalences when $\mathscr C$ is idempotent complete. However, when $\mathscr C \subseteq \mathcal S$ is the full subcategory of finite CW complexes, Wall's finiteness obstruction \cite[Thm.~F]{Wall} measures the difference between $\mathscr C = \Ind^\omega(\mathscr C)$ and $\Ind(\mathscr C)^\omega = \mathcal S^\omega$.

In the 1-categorical case, see also \cite[Rmk.~2.1.6 and Prop.~2.3.11]{MakkaiPare}.
\end{Par}

\begin{Par}\label{Par profinite kappa}
In case of the category of Stone topological spaces $\Pro(\Fin)$, we will simply write $\ProFin_\kappa$ for $\Pro^\kappa(\Fin)$. For instance, $\ProFin_{\aleph_1}$ is the category of light profinite sets. The following lemma is an easy consequence of the equivalence $\ProFin \simeq \BoolAlg\op$ given by $X \mapsto C(X,\mathbf F_2)$ and $\Spec A \mapsfrom A$.
\end{Par}

\begin{Lemma}\label{Lem kappa-small}
Let $\kappa$ be an infinite cardinal, and let $X \in \ProFin$. Then the following are equivalent:
\begin{enumerate}
\item $X \in \ProFin_\kappa$.
\item $\lvert C(X,\mathbf F_2) \rvert < \kappa$.
\item There exists a directed set $I$ with $\lvert I \rvert < \kappa$ and a diagram $I\op \to \Fin$ such that $X \cong \lim X_i$.
\item There exists a category $\mathscr C$ with $\lvert \operatorname{ob} \mathscr C \rvert < \kappa$ and $\lvert \operatorname{mor} \mathscr C \rvert < \kappa$ and a diagram $\mathscr C \to \Fin$ such that $X \cong \lim X_i$.\hfill\qed
\end{enumerate}
\end{Lemma}

\begin{Par}\label{Par schemes kappa why}
For schemes, we want a notion of size such that every $\kappa$-small scheme has a pro-\'etale cover by a $\kappa$-small w-strictly local scheme. Since the algebraic closure of a finite field is countably infinite, our notion of smallness should not distinguish between finite and countable infinite sets. This motivates the following definitions.
\end{Par}

\begin{Def}\label{Def size profinite}
If $X$ is a profinite set, write $\size(X)$ for the supremum of $\aleph_0$ and $\lvert C(X,\mathbf F_2)\rvert$.
\end{Def}

\begin{Def}\label{Def size scheme}
Let $X$ be a scheme. Write $\size(X)$ for the supremum of $\aleph_0$ and the cardinalities $\lvert \mathscr O_X(U)\rvert$ taken over all affine opens $U \subseteq X$. If $\kappa$ is a cardinal, we will say that~$X$ is \emph{$\kappa$-small} if $\size(X) < \kappa$, and we write $\Sch^\kappa$ for the full subcategory of $\Sch$ on $\kappa$-small schemes $X$. If $A \in \CRing$, we write $\size(A) = \max(\aleph_0,\lvert A \rvert)$, and we write $\CRing^\kappa$ for the full subcategory of $\CRing$ on $\kappa$-small commutative rings.
\end{Def}

If $X = \Spec A$ is affine, then the same argument as in \cite[Tag \href{https://stacks.math.columbia.edu/tag/000P}{000P}]{Stacks} shows that $\size(X) = \size(A)$. In particular, although our notion of size does not agree with the definition of \cite[Tag \href{https://stacks.math.columbia.edu/tag/000H}{000H}]{Stacks}, they do agree in the affine case.

For any scheme $X$, we will prove an exodromy theorem for the $\kappa$-small pro-\'etale site on~$X$, for any cardinal $\kappa > \size(X)$.

\begin{Par}\label{Par pi0 profinite}
If $X$ is a qcqs scheme, recall that the quotient topology on $\pi_0(X)$ for the map $X \to \pi_0(X)$ is profinite \cite[Tags \href{https://stacks.math.columbia.edu/tag/0906}{0906} and \href{https://stacks.math.columbia.edu/tag/094L}{094L}]{Stacks}. In this paper, we will always consider $\pi_0(X)$ as a profinite set.
\end{Par}

\begin{Lemma}\label{Lem pi0 kappa}
Let $X$ be a qcqs scheme, and $\pi_0(X) \in \ProFin$ its profinite set of connected components. Then $\size(\pi_0(X)) \leq \size(X)$.
\end{Lemma}

\begin{proof}
The set $C(X,\mathbf F_2)$ can be identified with the set of idempotents in $\mathscr O_X(X)$, so the result follows from \autoref{Lem kappa-small} since $\mathscr O_X(X)$ injects into a finite product $\prod_{i=1}^n \mathscr O_X(U_i)$ for $U_i \subseteq X$ affine.
\end{proof}

\begin{Rmk}
There is a small inconsistency in our terminology, in that a scheme $X$ is $\kappa$-small if and only if $\size(X) < \kappa$, but the same does not quite hold for profinite sets. Indeed, by \ref{Par profinite kappa} and \autoref{Lem kappa-small}, a profinite set $X$ is $\kappa$-small if and only if $\lvert C(X,\mathbf F_2) \rvert < \kappa$, but $\size(X)$ is equal to $\max(\aleph_0,\lvert C(X,\mathbf F_2) \rvert)$ and not $\lvert C(X,\mathbf F_2) \rvert$. This issue will not come up in practice, as the main theorems will start with a $\kappa$-small scheme $X$, which forces $\kappa$ to be uncountable since $\kappa > \size(X) \geq \aleph_0$. As explained in \ref{Par schemes kappa why}, our definition of $\size(X)$ is chosen so that every $\kappa$-small scheme has a pro-\'etale cover by a $\kappa$-small w-strictly local scheme, and this property would fail with other conventions. The convention for the $\size$ of a profinite set is then justified by \autoref{Lem pi0 kappa}.
\end{Rmk}

\subsection{Review of acyclic and w-strictly local schemes}\label{Sec acyclic}
Recall the following definition:

\begin{Def}\label{Def acyclic}
A scheme $W$ is \emph{acyclic} if it is quasi-compact and every \'etale surjection $U \twoheadrightarrow W$ has a section. This is the same definition as \cite[\S3]{Artin}, except that we always assume that $W$ is quasi-compact. We say that a scheme $W$ is \emph{w-strictly local} if it is acyclic and the set of closed points $W^\cl \subseteq W$ is closed. This coincides with the definition of \cite[Def.~2.1.1 and Def.~2.2.1]{BhattScholze} since $W$ is automatically affine by \autoref{Lem acyclic} below.
\end{Def}

\begin{Rmk}\label{Rmk w-projective}
We will say that a scheme $W$ is \emph{w-projective} if it is w-strictly local and $\pi_0(W)$ is projective in compact Hausdorff spaces. Such schemes are called \emph{w-contractible} in \cite[Def.~2.4.1]{BhattScholze}, but the condition is more similar to projectivity than contractibility. Note that this notion is stronger than the \emph{weakly projective} objects of \cite[Def.~6.2.2]{LurieUltra} in $\AffProEt_{/X} = \Pro(\AffEt_{/X})$ (see \autoref{Def AffProEt} for notation), which correspond to acyclic objects in our definition.
\end{Rmk}

In \autoref{Ex acyclic not wsl}, we will see examples of acyclic schemes that are not w-strictly local. (These can also be constructed directly using only scheme-theoretic methods.)

\begin{Lemma}\label{Lem acyclic}
Let $W$ be an acyclic scheme. Then
\begin{enumerate}
\item $W$ is affine.
\item Every intersection of clopen subschemes of $W$ is acyclic.
\item Every connected component is the spectrum of a strictly Henselian local ring. In particular, it has a unique closed point, whose residue field is separably algebraically closed.
\item The composite $W^\cl \to W \to \pi_0(W)$ is a homeomorphism.
\end{enumerate}
\end{Lemma}

\begin{proof}
Statements (a) and (c) are \cite[Prop.~3.1 and Prop.~3.2]{Artin}. For (b), let $U \subseteq W$ be clopen and let $f \colon V \twoheadrightarrow U$ be an \'etale surjection. Writing $U^c = W \setminus U$ for the complement, the surjection $V \amalg U^c \to U \amalg U^c = W$ has a section since $W$ is acyclic, hence so does $f$. This proves the result when $U$ is clopen, and the result for intersections of such follows from \cite[Prop.~3.3(i)]{Artin}. For (d), note that $W^\cl$ is quasi-compact as it is the maximum spectrum of a commutative ring. Thus, the continuous bijection $W^\cl \to \pi_0(W)$ is a closed map since $\pi_0(W)$ is compact Hausdorff (by \ref{Par pi0 profinite}), hence it is a homeomorphism.
\end{proof}

\begin{Def}\label{Def strictly profinite}
A scheme $S$ is called \emph{strictly profinite} if it is $0$-dimensional, reduced, qcqs, and all residue fields are separably algebraically closed.
\end{Def}

\begin{Rmk}\label{Rmk wsl}
If $S$ is strictly profinite, then it is automatically acyclic \cite[Prop.~3.2]{Artin}, hence affine by \autoref{Lem acyclic}, and w-strictly local since every point is closed. The condition that $S$ is reduced, affine, and $0$-dimensional means exactly that $S$ is an absolutely flat affine scheme \cite[Tag \href{https://stacks.math.columbia.edu/tag/092F}{092F}]{Stacks}, so a strictly profinite scheme is absolutely flat. Conversely, an absolutely flat and w-strictly local affine scheme is strictly profinite by \autoref{Lem acyclic}(c). In particular, if $W$ is a w-strictly local affine scheme, then the subscheme of closed points $W^\cl$ with the reduced induced scheme structure is strictly profinite \cite[Lem.~2.2.3 and Lem.~2.2.15]{BhattScholze}.
\end{Rmk}

\begin{Lemma}\label{Lem strictly profinite weakly etale}
Let $T \to S$ be a weakly \'etale map of affine schemes. If $S$ is strictly profinite, then so is $T$.
\end{Lemma}

\begin{proof}
By \autoref{Rmk wsl}, a scheme is strictly profinite if and only if it is affine, absolutely flat, and all residue fields are separably algebraically closed. Thus, the result follows from \cite[Tags \href{https://stacks.math.columbia.edu/tag/092I}{092I}(1) and \href{https://stacks.math.columbia.edu/tag/092R}{092R}]{Stacks}.
\end{proof}

\begin{Lemma}\label{Lem acyclic iso closed points}
Let $W \to V$ be a weakly \'etale morphism of acyclic schemes. Then $W \to V$ is an isomorphism if and only if $W^\cl$ maps bijectively to $V^\cl$.
\end{Lemma}

\begin{proof}
If $W \to V$ is an isomorphism, then it clearly induces a bijection $W^\cl \to V^\cl$. Conversely, suppose that $W^\cl \to V$ lands in $V^\cl$ and induces a bijection $W^\cl \to V^\cl$. To check that $W \to V$ is an isomorphism, it suffices to do this after base change along $\Spec \mathscr O_{V,v} \to V$ for every closed point $v \in V$ \cite[Tags \href{https://stacks.math.columbia.edu/tag/00HN}{00HN}(6) and \href{https://stacks.math.columbia.edu/tag/00DK}{00DK}]{Stacks}. Since $V^\cl \to \pi_0(V)$ is a bijection, the localisation at a closed point $v$ agrees with the irreducible component containing this point, so $\mathscr O_{V,v}$ is a strictly Henselian local ring by \autoref{Lem acyclic}. Moreover, $\Spec \mathscr O_{V,v} \times_V W \subseteq W$ is an intersection of clopen subschemes of $W$, hence acyclic by \autoref{Lem acyclic}(b). Thus, we may replace $V$ by $\Spec \mathscr O_{V,v}$ and $W$ by the fibre product. Since $W^\cl \to V^\cl$ is bijective, we see that $W$ is acyclic with a unique closed point, hence the spectrum of a strictly Henselian local ring. The result follows from Olivier's theorem \cite[Thm.~4.1]{Olivier} (see also \cite[Tag \href{https://stacks.math.columbia.edu/tag/092Z}{092Z}]{Stacks} and \cite[Thm.~2.3.5]{BhattScholze}).
\end{proof}

\begin{Par}
The category of schemes is canonically tensored over profinite sets: for a finite set $S$, we have $X \times S = \coprod_{s \in S} X$, and this is extended to profinite sets by taking limits in $\Sch$. If $V$ is an $X$-scheme, then continuous maps $\pi_0(V) \to S$ are in bijection with morphisms $V \to X \times S$ in $\Sch_{/X}$. Given a map $T \to S$ in $\ProFin$ and a map $\pi_0(V) \to S$, write $V \times_S T$ for the fibre product $V \times_{X \times S} (X \times T)$; note that this does not depend on the base scheme~$X$.
\end{Par}

\begin{Lemma}\label{Lem acyc profinite}
Let $W$ be an acyclic (resp.\ w-strictly local) scheme, let $S = \pi_0(W)$, and let $T \to S$ be a map of profinite sets. Then $W \times_S T$ is acyclic (resp.\ w-strictly local), and the map $W \times_S T \to W$ is pro-(affine \'etale) and takes closed points to closed points. Moreover, the natural map $\pi_0(W \times_S T) \to T$ is a homeomorphism.
\end{Lemma}

\begin{proof}
Any map $T \to S$ of profinite sets can be written as a cofiltered limit of maps $T_i \to S$ where each $T_i$ is a finite disjoint union of clopens in $S$. Thus, the map $W \times_S T \to W$ is a limit of maps $W_i \to W$ where $W_i$ is a finite disjoint union $\coprod_{j=1}^n U_j$ of clopens in~$W$. In particular, it is pro-(affine \'etale), and each $W_i \to W$ is a map between acyclic (resp.\ w-strictly local) schemes that takes closed points to closed points. The same holds in the limit by \cite[Prop.~3.3(i)]{Artin} and \cite[Lem.~2.1.9]{BhattScholze}. The final statement is clear for each $T_i \to S$, and therefore holds in the limit since the forgetful functor $\Aff \to \Top$ preserves cofiltered limits.
\end{proof}

The following lemma provides a converse:

\begin{Lemma}\label{Lem acyc profinite converse}
Let $V \to W$ be a weakly \'etale map of acyclic schemes that takes closed points to closed points. Then the natural map $f \colon V \to W \times_{\pi_0(W)} \pi_0(V)$ is an isomorphism.
\end{Lemma}

\begin{proof}
By \autoref{Lem acyc profinite}, the map $W \times_{\pi_0(W)} \pi_0(V) \to W$ is a weakly \'etale map of acyclic schemes that takes closed points to closed points. Since the same holds for $V \to W$, we see that $f$ is weakly \'etale \cite[Tag \href{https://stacks.math.columbia.edu/tag/0951}{0951}]{Stacks}. If $v \in V$ is closed, then so is its image $w \in W$ by assumption. Then $f(v)$ and the unique closed point of the component of $f(v)$ both map to $w$, so they agree since the fibres of $W \times_{\pi_0(W)} \pi_0(V) \to W$ are $0$-dimensional. This means that $f$ maps closed points to closed points. Since it induces a bijection on $\pi_0$, it induces a bijection on closed points, so $f$ is an isomorphism by \autoref{Lem acyclic iso closed points}.
\end{proof}

\begin{Cor}\label{Cor equiv wsl}
Let $W$ be an acyclic scheme, and let $S = \pi_0(W)$. Then the functors $W \times_S (-)$ and $\pi_0$ set up mutually inverse equivalences between $\ProFin_{/S}$ and the category of acyclic schemes with a weakly \'etale map to $W$ that takes closed points to closed points. \hfill\qedsymbol
\end{Cor}

\begin{Cor}\label{Cor w-local topology}
Consider a pullback diagram
\[
\begin{tikzcd}[row sep=1.1em,column sep=1.1em]
V' \ar{r}\ar{d} & W' \ar{d} \\
V \ar{r} & W
\end{tikzcd}
\]
of schemes such that $V$, $W$, and $W'$ are w-strictly local and $V \to W$ is w-local and weakly \'etale. Then $V'$ is w-strictly local and $V' \to W'$ is w-local and weakly \'etale.
\end{Cor}

\begin{proof}
By \autoref{Lem acyc profinite converse}, the map $V \to W \times_{\pi_0(W)} \pi_0(V)$ is an isomorphism. Then the base change is given by $V' \cong W' \times_{\pi_0(W)} \pi_0(V) \to W'$, so \autoref{Lem acyc profinite} gives the result.
\end{proof}

\begin{Lemma}\label{Lem strictly profinite}
Let $S$ be a scheme such that all local rings $\mathscr O_{S,s}$ are strictly Henselian; for example, a strictly profinite scheme. Then the canonical geometric morphism $\Sh(S_\et) \to \Sh(S_\Zar)$ is an equivalence.
\end{Lemma}

\begin{proof}
See \cite[Corollary~2.5]{Schroer}.
\end{proof}

All w-strictly local schemes constructed in \cite[\S2]{BhattScholze} are Henselian by definition, and we see that this is no coincidence:

\begin{Lemma}\label{Lem wsl Henselian}
Let $W$ be a w-strictly local scheme, and let $W^\cl \subseteq W$ be the reduced subscheme of closed points. Then $W$ is Henselian along $W^\cl$.
\end{Lemma}

\begin{proof}
By \cite[Tag \href{https://stacks.math.columbia.edu/tag/09XI}{09XI}(3)]{Stacks}, we have to prove that if $Y \to W$ is finite and $Z \hookrightarrow Y$ is the pullback of $W^\cl \hookrightarrow W$, then $\pi_0(Z) \to \pi_0(Y)$ is a bijection. For this, we may work in the fibres above the connected components of $W$, so we reduce to the case where $W$ is connected. Then $W$ is the spectrum of a strictly Henselian local ring (see \autoref{Lem acyclic}), so the result holds as Henselian local rings are Henselian along their closed point.
\end{proof}

\subsection{The \texorpdfstring{$\kappa$}{κ}-small pro-\texorpdfstring{\'etale}{étale} site}\label{Sec kappa proet}
In this section, we introduce a $\kappa$-small variant of the pro-\'etale site on a scheme $X$. In \autoref{Prop enough wsl kappa}, we prove that it has a basis of $\kappa$-small w-strictly local schemes, analogous to the unbounded situation \cite[Prop.~4.2.8]{BhattScholze}.

We also introduce a new w-local site $X_\wl$ on a scheme $X$ (\autoref{Def w-local topology}), and study its relationship to the pro-\'etale site (\autoref{Lem w-local topology}) and the site of profinite sets (\autoref{Lem hypercover}). Finally, we show that any weakly \'etale map $W \to X$ from a w-strictly local scheme $W$ is a limit of \'etale maps (\autoref{Cor wsl affproet}).

\begin{Def}\label{Def AffProEt}
If $\kappa$ is a cardinal and $X$ is a scheme with $\size(X) < \kappa$, we define the following full subcategories of $\Sch_{/X}$.
\begin{itemize}
\item Write $\Et_{/X}$ for the category of \'etale morphisms $Y \to X$, and write $\AffEt_{/X}$ for the subcategory of $\Et_{/X}$ on maps $Y \to X$ with $Y$ affine.
\item Write $\ProEt_{/X}$ for the category of weakly \'etale morphisms $Y \to X$.
\item Write $\AffProEt_{/X}$ for the essential image of $\Pro(\AffEt_{/X}) \to \ProEt_{/X}$ taking a pro-system to its limit in $\Sch_{/X}$.
\item Write $\ProEt^\kappa_{/X}$ (resp.~$\AffProEt^\kappa_{/X}$) for the $\kappa$-small objects in $\ProEt_{/X}$ (resp.~$\AffProEt_{/X}$), as in \autoref{Def size scheme}.
\item If $X = \Spec A$ is affine, we write $\Et_A \subseteq \CRing_{A/}$ for the category of \'etale $A$-algebras. Thus, we have $\Et_A\op \simeq \AffEt_{/X}$ and $\Ind^\kappa(\Et_A)\op \simeq \AffProEt^\kappa_{/X}$.
\end{itemize}
We write $\AffProEt_{/X}^{\acyc}$ (resp.\ $\AffProEt_{/X}^\wsl$) for the full subcategory of acyclic (resp.\ w-strictly local) objects in $\AffProEt_{/X}$, and likewise for $\AffProEt_{/X}^{\kappa,\acyc}$ and $\AffProEt_{/X}^{\kappa,\wsl}$.
\end{Def}

\begin{Lemma}\label{Lem affproet}
If $X$ is a scheme, then the functor $\lim\colon \Pro(\AffEt_{/X}) \to \AffProEt_{/X}$ is an equivalence.
\end{Lemma}

\begin{proof}
The functor is essentially surjective by definition, so we only need to prove it is fully faithful. Let $(U_i)_{i \in I}$ and $(V_j)_{j \in J}$ be a cofiltered systems in $\AffEt_{/X}$. Then
\[
\Hom_X\Big(\lim_{\substack{\longleftarrow \\ i \in I}} U_i, \lim_{\substack{\longleftarrow \\ j \in J}} V_j\Big) \cong \lim_{\substack{\longleftarrow \\ j \in J}} \Hom_X\Big(\lim_{\substack{\longleftarrow \\ i \in I}} U_i, V_j\Big) \cong \lim_{\substack{\longleftarrow \\ j \in J}} \colim_{\substack{\longrightarrow \\ i \in I\op}} \Hom_X(U_i, V_j), 
\]
where the second step follows from \cite[Tags \href{https://stacks.math.columbia.edu/tag/01ZC}{01ZC} and \href{https://stacks.math.columbia.edu/tag/02GR}{02GR}]{Stacks}.
\end{proof}

\begin{Rmk}
We warn the reader that $\AffProEt_{/X}$ is not the same thing as the affine objects in $\ProEt_{/X}$. It is clear that $\AffProEt_{/X}$ is contained in the affine objects of $\ProEt_{/X}$, but not every affine weakly \'etale $X$-scheme is a limit of affine \'etale $X$-schemes (even when $X$ is affine). For instance, the example of \cite[Ex.~4.1.12]{BhattScholze} is affine since it is obtained by glueing two affine schemes along closed subschemes \cite[Tag \href{https://stacks.math.columbia.edu/tag/0ET0}{0ET0}]{Stacks}, but it is not in $\AffProEt_{/X}$.
\end{Rmk}

Nonetheless, we have the following result:

\begin{Lemma}\label{Lem affproet over}
Let $X$ be a scheme and let $W \in \AffProEt_{/X}$. Then the natural functor $\AffProEt_{/W} \to (\AffProEt_{/X})_{/W}$ is an equivalence. The same holds for $\AffProEt^\kappa$ if $\kappa > \size(X)$ and $\kappa > \size(W)$.
\end{Lemma}

\begin{proof}
To define the functor, suppose $W = \lim_i W_i$ is a cofiltered limit of $W_i \in \AffEt_{/X}$. Then every $U \in \AffEt_{/W}$ is obtained by base change from $U_i \in \AffEt_{/W_i}$ for some~$i$ \cite[Tags \href{https://stacks.math.columbia.edu/tag/05N9}{05N9} and \href{https://stacks.math.columbia.edu/tag/07RI}{07RI}]{Stacks}, hence is also in $\AffProEt_{/X}$. This defines a functor \makebox{$\AffEt_{/X} \to (\AffProEt_{/X})_{/W}$,} and we can extend to $\AffProEt_{/X}$ by taking limits. It is fully faithful since both are full subcategories of $\Sch_{/W}$. If $V \to W$ is a morphism in $\AffProEt_{/X}$, there exists a cofiltered diagram $\{V_i \to W_i\}_{i \in \mathcal I}$ in $\Fun([1],\AffEt_{/X})$ whose limit is $V \to W$ \cite[Cor.~6.1.14]{KashiwaraSchapira}. The maps $V_i \to W_i$ are \'etale \cite[Tag \href{https://stacks.math.columbia.edu/tag/02GW}{02GW}]{Stacks} and the base change $V_i \times_{W_i} W$ is affine \cite[Tag \href{https://stacks.math.columbia.edu/tag/01JQ}{01JQ}]{Stacks}, so we may write $V$ as the limit of the diagram $\{V_i \times_{W_i} W\}_{i \in \mathcal I}$ in $\AffEt_{/W}$. The $\kappa$-small statement follows since both sides are the objects $V \to W \to X$ with $\size(V) < \kappa$.
\end{proof}

\begin{Def}\label{Def kappa pro etale}
Let $X$ be a scheme and let $\kappa$ be a cardinal with \makebox{$\size(X) < \kappa$.} Then the \emph{$\kappa$-pro-\'etale site} $X_{\kappa\text{-}\proet}$ (resp.\ \emph{$\kappa$-affine pro-\'etale site} $X_{\kappa\text{-}\affproet}$) is the category $\ProEt^\kappa_{/X}$ (resp.\ $\AffProEt^\kappa_{/X}$), with covering given by fpqc covers.
\end{Def}

The main local structure theory of the pro-\'etale site generalises to the $\kappa$-pro-\'etale site:

\begin{Prop}\label{Prop enough wsl kappa}
Let $X$ be a scheme and $\kappa > \size(X)$ a cardinal. If $V \in X_{\kappa\text{-}\proet}$, then there exists a covering $\{W_i \to V\}_{i \in I}$ in $X_{\kappa\text{-}\proet}$ where each $W_i$ is in $\AffProEt_{/X}^{\kappa,\wsl}$ as well as in $\AffProEt_{/V}^{\kappa,\wsl}$ such that every point $v \in V$ is the image of a closed point $w \in W_i$ for some $i \in I$.
\end{Prop}

\begin{proof}
The proof is the same as \cite[Lem.~4.2.4]{BhattScholze}. One has to check that all the constructions in \cite[\S2.2, 2.3]{BhattScholze} can be carried out within rings of cardinality $<\kappa$. We will show that each of the constructions involved is a colimit of rings of size $\leq \size(A)$ indexed by a poset of size $\leq \size(A)$, which gives the result by \autoref{Rmk colimit kappa small}

For the construction $A \mapsto A^Z$ of \cite[Lem.~2.2.4]{BhattScholze}, this follows since it is a filtered colimit of finite products of localisations of $A$ indexed by the set $\mathcal P_\fin(A)$ of finite subsets of $A$; see \cite[Tag \href{https://stacks.math.columbia.edu/tag/0973}{0973}]{Stacks}. For the construction $A \mapsto \overline A$ of \cite[Lem.~2.2.7]{BhattScholze}, this is \cite[Tag \href{https://stacks.math.columbia.edu/tag/097S}{097S}]{Stacks}. For the Henselisation of \cite[Def.~2.2.10]{BhattScholze}, we have
\[
\size(\Hens_A(B)) \leq \max(\size(A),\size(B)).
\]
Indeed, for any \'etale $A$-algebra $A'$, the set $\Hom_A(A',B)$ has cardinality $\leq \size(B)$, the set of \'etale $A$-algebras up to isomorphism has cardinality $\leq \size(A)$, and any \'etale $A$-algebra $A'$ has $\size(A') \leq \size(A)$.
\end{proof}

The usual argument \cite[Tag \href{https://stacks.math.columbia.edu/tag/03A0}{03A0}]{Stacks} therefore gives

\begin{Cor}
Let $X$ be a scheme, and let $\kappa > \size(X)$ be a cardinal. Then the inclusion $\AffProEt^\kappa_{/X} \hookrightarrow \ProEt^\kappa_{/X}$ induces an equivalence $\Sh(X_{\kappa\text{-}\affproet}) \simeq \Sh(X_{\kappa\text{-}\proet})$. \hfill\qedsymbol
\end{Cor}

In the proof of the pro-\'etale exodromy theorem, the following site will be useful.

\begin{Def}\label{Def w-local topology}
Let $X$ be a scheme. Then the \emph{w-local site} $X_\wl$ is the category $\AffProEt_{/X}^\wsl$ together with the pretopology whose covers are finite jointly surjective families of \makebox{w-local} maps. This defines a pretopology in the sense of \cite[Tag \href{https://stacks.math.columbia.edu/tag/00VH}{00VH}]{Stacks} by \autoref{Cor w-local topology}. Likewise, define the $\kappa$-small w-local site $X_{\kappa\text{-}\wl}$ as the full subcategory $\AffProEt_{/X}^{\kappa, \wsl}$ of $\kappa$-small w-strictly local schemes, endowed with the w-local topology.
\end{Def}

\begin{Lemma}\label{Lem w-local topology}
Let $X$ be a scheme. Then w-local topology and the pro-\'etale topology on $\AffProEt_{/X}^{\wsl}$ agree. The same holds for $\AffProEt_{/X}^{\kappa, \wsl}$ if $\kappa > \size(X)$.
\end{Lemma}

\begin{proof}
Clearly the pro-\'etale topology is finer than the w-local topology. For the converse, we have to show that any pro-\'etale cover in $\AffProEt_{/X}^\wsl$ can be refined by a w-local cover. First, assume $f \colon V \twoheadrightarrow W$ is a one-object pro-\'etale covering in $\AffProEt_{/X}^\wsl$. The closed subset $Z = f^{-1}(W^\cl) \subseteq V$ is contained in $V^\cl$: if $v \in V$ maps to a closed point $w$, then $\overline{\{v\}}$ maps to $w$, which forces $\overline{\{v\}} = \{v\}$ since the fibres of a weakly \'etale map are \makebox{$0$-dimensional} \cite[Tags \href{https://stacks.math.columbia.edu/tag/092I}{092I}(1) and \href{https://stacks.math.columbia.edu/tag/092F}{092F}]{Stacks}. Let $S = \pi_0(V) \simeq V^\cl$, and let $T \subseteq S$ be the closed subset corresponding to $Z \subseteq V^\cl$. Then $U = V \times_S T$ is w-strictly local and weakly \'etale over $V$ by \autoref{Lem acyc profinite}. The composition $U \to V \to W$ is still surjective: it hits all closed points by assumption, hence is surjective since flat maps satisfy `going up' \cite[Tag \href{https://stacks.math.columbia.edu/tag/03HV}{03HV}]{Stacks}. Thus, $U \twoheadrightarrow W$ is a w-local weakly \'etale cover refining the given cover $V \twoheadrightarrow W$.

In general, if $\mathscr V = \{V_i \to W\}_{i \in I}$ is a pro-\'etale cover in $\AffProEt_{/X}^\wsl$, there exists a finite subset $J \subseteq I$ such that $V = \coprod_{j \in J} V_j \twoheadrightarrow W$ is a pro-\'etale cover (and $V$ is w-strictly local). By the above, there exists a w-local cover $U \twoheadrightarrow W$ in $\AffProEt_{/X}^\wsl$ refining $V \twoheadrightarrow W$. Then $\{V_j \times_V U \to W\}_{j \in J}$ is a w-local cover in $\AffProEt_{/X}^\wsl$ refining $\mathscr V$. This proves the first statement, and the second follows by noting that $\size(V_j \times_V U) \leq \size(V_j)$ since $V_j \times_V U \hookrightarrow U$ is a closed immersion of affine schemes.
\end{proof}

Given a site $\mathscr C$ and an object $U \in \mathscr C$, write $\HypCov(\mathscr C,U)$ for the category of hypercoverings of $U$ in $\mathscr C$ in the sense of \cite[Tag \href{https://stacks.math.columbia.edu/tag/01G5}{01G5}]{Stacks} (this depends on the pretopology on $\mathscr C$, not just the topology).

\begin{Lemma}\label{Lem hypercover}
Let $X$ be a scheme, let $W \in \AffProEt_{/X}^\wsl$, and let $S = \pi_0(W)$. Then the functors
\[
\pi_0 \colon \HypCov(X_\wl,W) \leftrightarrows \HypCov(\ProFin,S) \colon W \times_S (-)
\]
are mutually inverse equivalences (where $\ProFin$ is endowed with the pretopology whose covers are jointly surjective finite families of continuous maps).
\end{Lemma}

\begin{proof}
Note that both constructions are indeed functorial and land in hypercovers. For $\pi_0 \simeq (-)^\cl$, this is because w-local covers define surjective maps on closed points, and for $W \times_S (-)$, this is \autoref{Lem acyc profinite}. They are mutual inverses by \autoref{Cor equiv wsl}.
\end{proof}

Finally, we will use the above to show that w-strictly local objects in $\ProEt_{/X}$ are always contained in $\AffProEt_{/X}$; see \autoref{Cor wsl affproet} below.

\begin{Lemma}\label{Lem affproet descent}
Let $X$ be a scheme, let $W \to X$ be a weakly \'etale morphism, let $\pi_0(W) \to S$ be a continuous map to a profinite set, and let $T \twoheadrightarrow S$ be a surjective continuous map of profinite sets. If $W \times_S T$ is in $\AffProEt_{/X}$, then so is $W$.
\end{Lemma}

\begin{proof}
Since $W \times_S T$ is the base change of $W \to \Spec \mathbf Z \times S$ along the fpqc surjection $\Spec \mathbf Z \times T \twoheadrightarrow \Spec \mathbf Z \times S$, we conclude from \cite[Tag \href{https://stacks.math.columbia.edu/tag/02L5}{02L5}]{Stacks} that $W$ is affine. Since $W \times_S T \in \AffProEt_{/X}$, there exists $U \in \AffEt_{/X}$ and a map $W \times_S T \to U$. We may write $T$ as a cofiltered limit of profinite sets $T_i$ that are finite coproducts of clopens in $S$. Then $\Hom_X(W \times_S T_i,U) \neq\varnothing$ for some $i$ by \cite[Tag \href{https://stacks.math.columbia.edu/tag/01ZC}{01ZC}]{Stacks}. Since $T_i \to S$ is surjective, it has a section, and we conclude that $\Hom_X(W,U) \neq \varnothing$. Choosing such a morphism $W \to U$ and replacing $X$ by $U$, we may assume $X$ is affine.

Then $\AffEt_{/X}$ has finite limits, so $\AffProEt_{/X} \simeq \Fun^\lex(\AffEt_{/X},\Set)\op$ \cite[Exp.~I, Thm.~8.3.3]{SGA4I}. The inclusion $\AffProEt_{/X} \to \Aff_{/X}$ has a left adjoint $F$ taking an affine \makebox{$X$-scheme} $Y$ to the left exact functor $\Hom_X(Y,-) \colon \AffEt_{/X} \to \Set$. Since $F$ is a localisation, the unit $Y \to F(Y)$ is an equivalence if and only if $Y \in \AffProEt_{/X}$. Since $F$ is a left adjoint, it preserves finite coproducts, so $F(W \times_S T_i) \cong F(W) \times_S T_i$. Moreover, $F$ preserves cofiltered limits: these are computed pointwise in $\Fun^\lex(\AffEt_{/X},\Set)\op$ since filtered colimits of sets commute with finite limits, so the claim follows by another application of \cite[Tag \href{https://stacks.math.columbia.edu/tag/01ZC}{01ZC}]{Stacks}. In particular, we get $F(W \times_S T) \cong F(W) \times_S T$. Since the natural map $W \times_S T \to F(W \times_S T) \cong F(W) \times_S T$ is an isomorphism and $T \to S$ is surjective, we conclude that $W \to F(W)$ was already an isomorphism \cite[Tag~\href{https://stacks.math.columbia.edu/tag/02L4}{02L4}]{Stacks}.
\end{proof}

\begin{Cor}\label{Cor wsl affproet}
Let $X$ be a scheme. Then every w-strictly local object of $\ProEt_{/X}$ is in $\AffProEt_{/X}$.
\end{Cor}

\begin{proof}
Suppose $W \to X$ is a weakly \'etale map from a w-strictly local scheme. By \autoref{Prop enough wsl kappa}, there exists a cover $\{W_i \to W\}_{i \in I}$ such that each $W_i$ is in $\AffProEt_{/X}^\wsl$ and $\AffProEt_{/W}^\wsl$. Applying \autoref{Lem w-local topology} to $W$, we may refine the cover by a w-local cover $\{W'_j \to W\}_{j \in J}$ over a finite index set $J$, and taking the disjoint union gives a w-local cover $W' \twoheadrightarrow W$ with $W' \in \AffProEt_{/X}^\wsl$. By \autoref{Lem acyc profinite converse}, we get $W' \cong W \times_{\pi_0(W)} \pi_0(W')$, so the result follows from \autoref{Lem affproet descent}.
\end{proof}

\section{The condensed Galois category}\label{Sec condensed}
The main result of this section is the classification of $\Fun^*(\Sh(X_\et),\Sh(S))$ for any profinite set $S$. Generalising the classification of points of the \'etale topos via spectra of strictly Henselian local rings in $\AffProEt_{/X}$ \cite[Exp.~VIII, Thm.~7.9]{SGA4II}, we prove in \autoref{Prop equivalence S} that geometric morphisms $s^* \colon \Sh(X_\et) \to \Sh(S)$ are classified by acyclic schemes $W \in \AffProEt_{/X}$ with $\pi_0(W) \simeq S$. This leads to an equivalence between the unstraightening of the functor $S \mapsto \Fun^*(\Sh(X_\et),\Sh(S))$ and the category $\AffProEt_{/X}^{\acyc}$; see \autoref{Thm equivalence}. This can also be deduced from \cite[Thm.~6.3.14]{LurieUltra}, but we give an independent and more geometric proof that also gives more refined statements for $\kappa$-bounded versions as well as for locally coherent geometric morphisms $\Sh(X_\et) \to \Sh(S)$.

In \hyperref[Sec condensed category]{\S2.1}, we give a direct definition of the condensed Galois category $\GAL(X)$ of a scheme $X$: it is given by the association taking $S \in \ProFin$ to the category $\Fun_\loccoh^*(\Sh(X_\et),\Sh(S))$ of locally coherent geometric morphisms $s^* \colon \Sh(X_\et) \to \Sh(S)$ (\autoref{Def Galois category}). When $X$ is qcqs, this is the condensed category defined by the profinite Galois category $\Gal(X)$ of \cite[Def.~12.1.3]{BGH}; see \autoref{Rmk Gal(X) profinite}.

The classical theory of \cite{SGA4II} says that geometric morphisms $s^* \colon \Sh(X_\et) \to \Set$ are always pro-represented by spectra of strictly Henselian local rings in $\AffProEt_{/X}$. To understand geometric morphisms $s^* \colon \Sh(X_\et) \to \Sh(S)$ for a profinite set $S$, we first study left exact functors $\mathscr C \to \Sh(S)$ for an arbitrary category $\mathscr C$. Instead of (pro-)representing such functors by a (pro-)object of $\mathscr C$, they are now represented by an \emph{$S$-cosheaf} in $\mathscr C$. This theory is set up in \hyperref[Sec cosheaf]{\S2.2}.

In \hyperref[Sec hensel]{\S2.3}, we define the Henselisation along a geometric morphism $s^* \colon \Sh(X_\et) \to \Sh(S)$, and show that it is always represented by an $S$-cosheaf in $\AffProEt_{/X}^{\acyc}$. This leads to the computation of geometric morphisms $s^* \colon \Sh(X_\et) \to \Sh(S)$ for a fixed profinite set $S$; see \autoref{Prop equivalence S}. In \hyperref[Sec Galois]{\S2.4}, we extract a statement about the unstraightening of the functor $S \mapsto \Fun^*(\Sh(X_\et),\Sh(S))$ from this; see \autoref{Thm equivalence}. We also obtain a version for locally coherent geometric morphisms; see \autoref{Cor loccoh}.

\subsection{The condensed category of points}\label{Sec condensed category}
\begin{Par}
Recall \cite[\S4.1]{Pyknotic}, \cite[\S13.3]{BGH}, \cite{ClausenScholze} that a ($\kappa$-)\emph{condensed $\infty$-category} is a hypersheaf $\ProFin\op \to \Cat_\infty$ (respectively $\ProFin_\kappa\op \to \Cat_\infty$) for the effective epimorphism topology. For ($\kappa$-)condensed $n$-categories for $n < \infty$ (in particular when $n=1$), it suffices to look at sheaves instead of hypersheaves, since $\Cat_n$ is an $(n+1)$-category.
\end{Par}

\begin{Par}\label{Par profinite scheme}
In this section, we need to prove a number of descent statements for surjective continuous maps $f \colon T \to S$ in $\ProFin$. Because the descent literature for schemes is very well developed, it will be beneficial to view $S \in \ProFin$ as the strictly profinite pro-\'etale $k$-scheme $S \times \Spec k$ for some separably algebraically closed field $k$. Then $\Sh((S \times \Spec k)_\et)$ is canonically identified with $\Sh(S)$ by \autoref{Lem strictly profinite}, and the structure map $S \times \Spec k \to \Spec k$ is affine pro-\'etale, hence integral \cite[Tag \href{https://stacks.math.columbia.edu/tag/0CKR}{0CKR}]{Stacks}. If $f \colon T \to S$ is a morphism in $\ProFin$, then $f \times \Spec k$ is integral \cite[Tag \href{https://stacks.math.columbia.edu/tag/035D}{035D}]{Stacks} and weakly \'etale \cite[Tag \href{https://stacks.math.columbia.edu/tag/092L}{092L}]{Stacks}, so in particular faithfully flat if $f$ is surjective.
\end{Par}

\begin{Lemma}\label{Lem Sh(S) sheaf}
The functor $\ProFin\op \to \LTop$ given by $S \mapsto \Sh(S)$ is a sheaf.
\end{Lemma}

This is a special case of \cite[Lem.~2.7(2)]{HaineDescent}. For completeness, we give a quick 1-categorical proof using descent results from the algebraic geometry literature.

\begin{proof}
Note that limits in $\LTop$ are computed as limits in $\Cat$ \cite[Exp.~VI, Prop.~7.4.7 and Exc.~7.4.14]{SGA4II} (for $\infty$-topoi, see \cite[Prop.~6.3.2.3]{LurieHTT}), so we may work in $\Cat$ instead. It suffices to prove descent for finite clopen decompositions $\{U_i \to S\}_{i \in I}$ and surjective maps $\{T \to S\}$, as these generate the topology on $\ProFin$. The result for a finite clopen decomposition is clear. If $f \colon T \to S$ is surjective map in $\ProFin$, apply \cite[Tag \href{https://stacks.math.columbia.edu/tag/0GEZ}{0GEZ}]{Stacks} to the surjective integral morphism of schemes $f \times \Spec k$ as in \ref{Par profinite scheme}.
\end{proof}

\begin{Def}\label{Def Pts}
Let $\mathscr X$ be a topos. The \emph{condensed category of points} $\PTS^*(\mathscr X)$ of $\mathscr X$ is the condensed category
\begin{align*}
\ProFin\op &\to \Cat\\
S &\mapsto \Fun^*(\mathscr X,\Sh(S)) \simeq \Fun_*(\Sh(S),\mathscr X)\op.
\end{align*}
This is a sheaf on $\ProFin$ by \autoref{Lem Sh(S) sheaf}. We denote the opposite condensed category by $\PTS_*(\mathscr X)$, as it is given by $S \mapsto \Fun_*(\Sh(S),\mathscr X)$. The underlying categories $\PTS^*(\mathscr X)(*)$ and $\PTS_*(\mathscr X)(*)$ are denoted by $\Pts^*(\mathscr X)$ and $\Pts_*(\mathscr X)$ respectively, which are the usual categories of points $\Fun^*(\mathscr X,\Set)$ and $\Fun_*(\Set,\mathscr X)$ (see also \autoref{Rmk op} for notational conventions appearing in the literature). We write $\PTS^*(X_\et)$ for $\PTS^*(\Sh(X_\et))$, and likewise for the other variants.
\end{Def}

\begin{Ex}\label{Ex Pts Zariski}
If $X$ is a scheme, then $\PTS_*(X_\Zar)$ is represented by the topological poset $(\lvert X\rvert,\geq) = (\lvert X \rvert,\leq)\op$, where $x \leq y$ if and only if $x \in \overline{\{y\}}$. Indeed, if $S \in \ProFin$, then \cite[Exp.~IV, 4.2.3]{SGA4I} shows that $\Fun_*(\Sh(S),\Sh(X_\Zar)) \simeq \Cont(S,\lvert X \rvert)$, viewed as a category by the partial order $\geq$ on $\lvert X\rvert$. Thus, we also see that $\PTS^*(X_\Zar)$ is represented by $(\lvert X\rvert,\leq)$.
\end{Ex}

\begin{Par}\label{Par coherent}
Recall that $\Sh(S)$ is a coherent topos for any profinite set $S$, and $\Sh(X_\et)$ is a locally coherent topos for any scheme $X$. Moreover, $\Sh(X_\et)$ is what \cite[Exp.~VI, Def.~2.3]{SGA4II} calls \emph{algebraic}: property \cite[Exp.~VI, Prop.~2.2(i ter)]{SGA4II} holds for the generating class $\AffEt_{/X} \subseteq \Sh(X_\et)$ since $U \to X$ is separated for all $U \in \AffEt_{/X}$ \cite[Tag \href{https://stacks.math.columbia.edu/tag/01KN}{01KN}]{Stacks}. Thus, all objects in $\Sh(S)$ or $\Sh(X_\et)$ are also algebraic \cite[Exp.~VI, 2.4.2]{SGA4II}.

Recall that a geometric morphism $f_* \colon \mathscr X \to \mathscr Y$ between algebraic topoi is \emph{coherent} \cite[Exp.~VI, Def.~3.1 and Prop.~3.2]{SGA4II} if $f^*$ takes coherent objects to coherent objects, and \emph{locally coherent} \cite[Exp.~VI, Def.~3.7]{SGA4II} if for every coherent object $X \in \mathscr X$, every coherent object $Y \in \mathscr Y$, and every map $X \to f^*Y$ in $\mathscr X$, the induced geometric morphism $\mathscr X_{/X} \to \mathscr Y_{/Y}$ is coherent. If $\mathscr X$ and $\mathscr Y$ are both coherent, then a geometric morphism $f_* \colon \mathscr X \to \mathscr Y$ is coherent if and only if it is locally coherent \cite[Exp.~VI, 3.7.1]{SGA4II}. We will write $\RTop^{\loccoh}$ and $\LTop^{\loccoh}$ for the subcategories of $\RTop$ and $\LTop$ of locally coherent topoi with locally coherent geometric morphisms.
\end{Par}

\begin{Lemma}
The functor $\ProFin\op \to \LTop^{\loccoh}$ given by $S \mapsto \Sh(S)$ is a sheaf.
\end{Lemma}

\begin{proof}
Note that $\Sh(S)$ is indeed a locally coherent (in fact, coherent) topos, and the geometric morphism $\Sh(T) \to \Sh(S)$ induced by any map $T \to S$ in $\ProFin$ is locally coherent (equivalently, coherent). Given a locally coherent topos $\mathscr X$ and a covering $\{g_i \colon T_i \to S\}_{i \in I}$ in $\ProFin$, we have to show that a geometric morphism $f \colon \Sh(S) \to \mathscr X$ is locally coherent if each composition $f \circ g_i$ is. We may assume $I$ is finite since $S$ is compact. Writing $g \colon T \to S$ for the disjoint union of the $T_i$, we see that $f \circ g$ is locally coherent since this property is local on $T$ and is assumed to hold for each $T_i$. Thus, we are reduced to the case of a one-object cover $g \colon T \twoheadrightarrow S$.

Given coherent objects $\mathscr F \in \Sh(S)$ and $X \in \mathscr X$ with a map $\mathscr F \to f^*X$, we have to show that the geometric morphism $\Sh(S)_{/\mathscr F} \to \mathscr X_{/X}$ is coherent. We know that $g^*\mathscr F$ is coherent, so by assumption, the composite geometric morphism $\Sh(T)_{/g^*\mathscr F} \to \Sh(S)_{/\mathscr F} \to \mathscr X_{/X}$ is coherent. Thus, for every coherent $Y \in \mathscr X_{/X}$, the pullback $g^*f^*Y$ is coherent. Coherent objects in $\Sh(S)$ are exactly those whose espace \'etal\'e is qcqs, and this property descends under $g$ by applying \ref{Par profinite scheme} and \cite[Tags \href{https://stacks.math.columbia.edu/tag/02KQ}{02KQ} and \href{https://stacks.math.columbia.edu/tag/02KR}{02KR}]{Stacks}. Thus, $f^*Y$ is coherent since $g^*f^*Y$ is, hence $\Sh(S)_{\mathscr F} \to \mathscr X_{/X}$ is coherent.
\end{proof}

\begin{Def}\label{Def Galois category}
The \emph{Galois category} of a scheme $X$ is the full condensed subcategory $\GAL(X) \subseteq \PTS^*(X_\et)$ on locally coherent geometric morphisms $s_* \colon \Sh(S) \to \Sh(X_\et)$, which is a sheaf on $\ProFin$ by the lemma above. Explicitly, it sends $S$ to the full subcategory of $\Fun_*(\Sh(S),\Sh(X_\et))\op$ on those $s_* \colon \Sh(S) \to \Sh(X_\et)$ for which there is a finite clopen decomposition $S = \coprod_{i \in I} S_i$ such that each $(s|_{S_i})_*$ factors through a qcqs open $U_i \subseteq X$ and $(s|_{S_i})_* \colon \Sh(S_i) \to \Sh(U_{i,\et})$ is coherent. We denote the opposite of $\GAL(X)$ by $\PTS_*^{\loccoh}(X_\et)$.
\end{Def}

\begin{Rmk}\label{Rmk 1-topos vs infinity}
In the definitions of $\PTS^*(X_\et)$ and $\GAL(X)$, we could also use $\infty$-topoi: by \cite[Lem.~6.4.5.6]{LurieHTT}, the restriction $\Fun_*(\Sh(S,\mathcal S),\Sh(X_\et,\mathcal S)) \to \Fun_*(\Sh(S),\Sh(X_\et))$ is an equivalence, and a geometric morphism $s_* \colon \Sh(S,\mathcal S) \to \Sh(X_\et,\mathcal S)$ is coherent if and only if its restriction $\Sh(S) \to \Sh(X_\et)$ is. Indeed, if $\Sh(S,\mathcal S) \to \Sh(X_\et,\mathcal S)$ is coherent, then so is the restriction $\Sh(S) \to \Sh(X_\et)$. The converse follows as in property $(*)$ of the proof of \cite[Thm.~A.7.5.3]{LurieSAG} (which works for any site $\mathscr C$ satisfying the hypotheses of \cite[Prop.~A.3.1.3(3)]{LurieSAG}, such as the site $\Sh(X_\et)^\coh$ of $0$-truncated objects of $\Sh(X_\et,\mathcal S)^\coh$).
\end{Rmk}

\begin{Rmk}\label{Rmk op}
Beware that there are different conventions in the literature for the (condensed) category of points of a topos. In \cite[Exp.~IV, \S6]{SGA4I}, the category $\Pts_*(\mathscr X)$ is called the category of points, whereas \cite[Def.~3.11.1]{BGH} uses this term for $\Pts^*(\mathscr X)$. We note that the Galois group of a field $k$ is definitionally equal to $\Pts^*(\Spec k_\et)$, and likewise the \'etale fundamental group is defined as the category of fibre functors $\Sh\lc(X_\et) \to \Fin$. Thus, our definition of $\GAL(X)$ is naturally an extension of this.
\end{Rmk}

\begin{Rmk}\label{Rmk Gal(X) profinite}
In the case where the scheme $X$ is qcqs, the above definition agrees with the definition of $\GAL(X)$ that appears in \cite[Def.~12.1.3]{BGH}. This is a direct consequence of higher Hochster duality \cite[Thm.~9.3.1]{BGH}; see \cite[Prop.~7.1.2.8]{WolfThesis}.

We warn the reader that the proof in \cite[Prop.~7.1.2.8]{WolfThesis} is missing an `op' (and a similar issue shows up in \cite[10.3.1]{BGH}). For a finite poset $P$ and a profinite set~$K$, the $\infty$-categorical Hochster duality of \cite[Thm.~9.3.1]{BGH} identifies the mapping space $\Map(K \times P,\widehat{\Pi}_{(\infty,1)}^S(\mathscr X))$ with $\Map_{\RTop_\infty^\coh}(\Sh(K \times P),\mathscr X) \simeq \Map_{\LTop_\infty^\coh}(\mathscr X,\Sh(K \times P))\op$. Analogously to \cite[Ex.~3.12.15]{BGH}, we have $\Sh(K \times P) \simeq \Fun(P,\Sh(K))$, so the right hand side is $\Map(P,\Fun_\coh^*(\mathscr X,\Sh(K)))\op \simeq \Map(P\op,\Fun_*^\coh(\Sh(K),\mathscr X))$, where an undesired `op' has appeared at $P$. (When applied to $P = \Delta^n$, it is easy to overlook this since $\Delta^n$ is equivalent to its opposite, but the functor $(-)\op \colon \Delta \to \Delta$ is not equivalent to the identity.)

However, every $\infty$-groupoid is canonically equivalent to its opposite \cite[\S57]{RezkNotes}, so we may replace $\Map_{\LTop_\infty^\coh}(\mathscr X,\Sh(K \times P))\op$ by $\Map_{\LTop_\infty^\coh}(\mathscr X,\Sh(K \times P))$, which is equivalent to $\Map(P,\Fun_\coh^*(\mathscr X,\Sh(K)))$. Applying this to all $P \in \Delta$ shows that the embedding $\Pro(\operatorname{Str}_\pi) \to \Cond(\Cat_\infty)$ takes the profinite layered $\infty$-category $\widehat{\Pi}_{(\infty,1)}^S(\mathscr X)$ to the condensed $\infty$-category $\PTS_\coh^*(\mathscr X)$ for any spectral $\infty$-topos $\mathscr X$. Thus, \autoref{Def Galois category} is an extension of this to non-qcqs schemes.
\end{Rmk}

\begin{Ex}\label{Ex lcons}
We saw in \autoref{Ex Pts Zariski} that $\PTS^*(X_\Zar)$ is represented by the topological poset $(\lvert X \rvert,\leq)$. If we restrict to the locally coherent geometric morphisms \makebox{$s^* \colon \Sh(X_\Zar) \to \Sh(S)$,} this is represented by $(X\lcons,\leq)$, where $X\lcons$ is the set $X$ endowed with the `local constructible topology', whose opens are disjoint unions of sets $F$ that are constructible in some affine open $U \subseteq X$ (this is also the underlying topological space of the scheme $X\cons$ constructed in \ref{Def Xcons}). When $X$ is qcqs, this has the effect of replacing the spectral topological space $\lvert X \rvert$ with the profinite space $X\cons$. On $\lvert X\rvert$, the partial order~$\leq$ is redundant, but on $X\cons$ it is exactly the information needed to recover the Zariski topology on $X$.
\end{Ex}

\subsection{Profinite cosheaves}\label{Sec cosheaf}
In the computation of the condensed category of points of \hyperref[Sec Galois]{\S2.4}, we will need to understand natural transformations between left exact functors $\Sh(X_\et) \to \Sh(S)$ for any profinite set~$S$. When $S$ is a point, the involved functors will be represented by objects of $\AffProEt_{/X}$, so a variant of the Yoneda lemma (see \autoref{Lem fully faithful} below) computes natural transformations as maps in $\AffProEt_{/X}$. The goal of this section is to explain a similar statement for an arbitrary profinite set $S$: if $\AffProEt_{/X} \vec\times_{\ProFin} \{S\}$ denotes the comma category of pairs $(Y,\pi_0(Y) \to S)$ with $Y \in \AffProEt_{/X}$, then \autoref{Cor cosheaf fully faithful} produces a fully faithful functor $\AffProEt_{/X} \vec\times_{\ProFin} \{S\} \to \Fun^\lex(\Sh(X_\et),\Sh(S))$.

\begin{Par}\label{Par cosheaf}
Recall that if $S$ is a profinite set and $\mathscr C$ is a category with finite limits, then the forgetful functor $\Fun(\Open(S)\op,\mathscr C) \to \Fun(\Clopen(S)\op,\mathscr C)$ induces an equivalence between $\mathscr C$-valued sheaves on $S$ and functors $\Clopen(S)\op \to \mathscr C$ that preserve finite products. Thus, if $\mathscr C$ is a category that is only assumed to have finite products, we will write $\Sh(S,\mathscr C)$ for the category of functors $\Clopen(S)\op \to \mathscr C$ that preserve finite products. Likewise, if $\mathscr C$ has finite coproducts, an $\mathscr C$-valued cosheaf on $S$ is a functor $\Clopen(S) \to \mathscr C$ that preserves finite coproducts. In this section, we will think of this as some sort of diagram in $\mathscr C$, and we therefore refer to it as an \emph{$S$-cosheaf in $\mathscr C$}. We write $\CSh(S,\mathscr C)$ for the category of $S$-cosheaves in $\mathscr C$.
\end{Par}

In the following result, we write $\lvert - \rvert \colon \Sch_{/X} \to \Top$ (resp.\ $\pi_0 \colon \Sch_{/X} \to \Top$) for the underlying topological space (resp.\ the set of connected components, endowed with the quotient topology).

\begin{Lemma}\label{Lem cosheaf in schemes}
If $X$ is a scheme and $S$ is a profinite set, then the following categories are canonically equivalent:
\begin{itemize}
\item $\CSh(S,\Sch_{/X})$;
\item The comma category $(\lvert-\rvert \downarrow S)$ of $X$-schemes $Y$ together with a continuous map $\lvert Y \rvert \to S$;
\item The comma category $(\pi_0 \downarrow S)$ of $X$-schemes $Y$ together with a continuous map $\pi_0(Y) \to S$;
\item $\Sch_{/X \times S}$.
\end{itemize}
\end{Lemma}

\begin{proof}
Given a coproduct-preserving functor $F \colon \Clopen(S) \to \Sch_{/X}$, we consider the \makebox{$X$-scheme} $Y = F(S)$, which has a map $\lvert Y \rvert \to S$ constructed as follows: for every continuous map $f \colon S \to T$ to a finite set $T$, we get a clopen decomposition $S = \coprod_{t \in T} f^{-1}(t)$, whence a clopen decomposition $Y = F(S) = \coprod_{t \in T} F(f^{-1}(t))$, giving a continuous map $\lvert Y \rvert \to T$. These are compatible for all maps $S \to T$ to finite sets $T$, giving a continuous map $\lvert Y \rvert \to S$ since $S$ is the limit. Conversely, a continuous map $g \colon \lvert Y \rvert \to S$ gives a functor $F \colon \Clopen(S) \to \Sch_{/X}$ by $U \mapsto g^{-1}(U)$ (as open subscheme of $Y$), which preserves finite coproducts. It is a straightforward verification that these two constructions are functorial and inverse to each other, which proves the equivalence $\CSh(S,\Sch_{/X}) \simeq (\lvert-\rvert \downarrow S)$. The equivalence $(\lvert-\rvert \downarrow S) \simeq (\pi_0\downarrow S)$ follows since every continuous map $\lvert Y \rvert \to S$ factors uniquely through $\pi_0(Y)$ \cite[Tag \href{https://stacks.math.columbia.edu/tag/08ZL}{08ZL}]{Stacks}. The equivalence $(\pi_0\downarrow S) \simeq \Sch_{/X \times S}$ follows since $\Hom_X(Y,X \times S) \simeq \Hom(\pi_0(Y),S)$ for $Y \in \Sch_{/X}$ and $S \in \ProFin$.
\end{proof}

\begin{Par}\label{Par oriented fibre product}
In the notations $(\lvert-\rvert \downarrow S)$ and $(\pi_0 \downarrow S)$, we would like to restrict to certain full subcategories $\mathscr C \subseteq \Sch_{/X}$. Since $\Hom(\lvert Y\rvert,S) \simeq \Hom(\pi_0(Y),S)$ for all $Y \in \Sch_{/X}$, we may unambiguously write the category of \autoref{Lem cosheaf in schemes} as the oriented fibre product\footnote{This is an alternative name and notation for comma categories; see for instance \cite[Tag \href{https://kerodon.net/tag/02AP}{02AP}]{Kerodon}. The notation $\Sch_{/X} \vec\times_{\Top} \{S\}$ does not specify the functor $\Sch_{/X} \to \Top$, which can be either $\lvert-\rvert$ or $\pi_0$.} $\Sch_{/X} \vec\times_{\Top} \{S\}$. Likewise, for any full subcategory $\mathscr C \subseteq \Sch_{/X}$, we define $\mathscr C \vec\times_{\Top} \{S\}$ as the full subcategory of $\Sch_{/X} \vec\times_{\Top} \{S\}$ of pairs $(Y,\pi_0(Y) \to S)$ with $Y \in \mathscr C$. The oriented fibre product $\mathscr C \vec\times_{\Top} \{S\}$ contains the fibre product $\mathscr C \times_{\Top} \{S\}$ of pairs $(Y,\pi_0(Y) \to S)$ where the map $\pi_0(Y) \to S$ is an isomorphism. The latter is the (homotopy) fibre of $\pi_0 \colon \mathscr C \to \Top$ over $S$. When $\mathscr C$ consists of qcqs schemes, the image of $\pi_0 \colon \mathscr C \to \Top$ lands in $\ProFin$ by \ref{Par pi0 profinite}, and we will write $\mathscr C \vec\times_{\ProFin} \{S\}$ instead of $\mathscr C \vec\times_{\Top} \{S\}$, and likewise for $\mathscr C \times_{\ProFin} \{S\}$.
\end{Par}

Recall from \autoref{Lem affproet} that $\AffProEt_{/X} \simeq \Pro(\AffEt_{/X})$. This implies the following.

\begin{Lemma}\label{Lem fully faithful}
If $X$ is a scheme, then the functor $\AffProEt_{/X} \to \Fun^\lex(\Sh(X_\et),\Set)\op$ taking $f \colon Y \to X$ to the left exact functor $\mathscr F \mapsto \Gamma(Y,f^*\mathscr F)$ is fully faithful.
\end{Lemma}

\begin{proof}
The Yoneda embedding $\AffEt_{/X} \to \Sh(X_\et)$ extends to a fully faithful functor $\Pro(\AffEt_{/X}) \to \Pro(\Sh(X_\et))$. If~$\mathscr C$ is a category with finite limits, there is a canonical equivalence $\Pro(\mathscr C) \simeq \Fun^\lex(\mathscr C,\Set)\op$, so we may think of objects of $\Pro(\Sh(X_\et))$ as left exact functors $\Sh(X_\et) \to \Set$. Then the functor given by a cofiltered diagram $(U_i)_{i \in I}$ in $\Pro(\AffEt_{/X})$ with limit $f \colon Y \to X$ in $\AffProEt_{/X}$ is the functor
\begin{align*}
\Sh(X_\et) &\to \Set \\
\mathscr F &\mapsto \colim_{\substack{\longrightarrow \\ i \in I\op}} \Gamma(U_i,\mathscr F) \cong \Gamma(Y,f^*\mathscr F),
\end{align*}
where the last isomorphism is \cite[Exp.~VII, Rmq.~5.14(a)]{SGA4II}.
\end{proof}

To ease notation, we will denote $\Gamma(Y,f^*\mathscr F)$ by $\Gamma(Y,\mathscr F)$ for $(f \colon Y \to X) \in \AffProEt_{/X}$.

\begin{Rmk}
The above functor extends to $\Sh(X_\proet) \to \Fun^\lex(\Sh(X_\et),\Set)\op$ taking a pro-\'etale sheaf $\mathscr G$ to $\mathscr F \mapsto \Hom(\mathscr G,\nu^*\mathscr F)$, where $\nu^* \colon \Sh(X_\et) \to \Sh(X_\proet)$ denotes the pullback. We note that this extension is \emph{not} fully faithful (or even conservative). For instance, if $X = \Spec k$ is the spectrum of an algebraically closed field $k$, then $\Sh(X_\proet) \simeq \Cond(\Set)$, and for every connected compact Hausdorff space $Y$, the functor $\Fin \to \Set$ given by $\Hom_{\Cond(\Set)}(Y,\nu^*-)$ is equivalent to the constant functor $*$. We do not know if the restricted functor $\ProEt_{/X} \to \Fun^\lex(\Sh(X_\et),\Set)\op$ is fully faithful. However, our main application is to w-strictly local objects $W \in \ProEt_{/X}$, which all lie in $\AffProEt_{/X}$ by \autoref{Cor wsl affproet}.
\end{Rmk}

\begin{Cor}\label{Cor cosheaf fully faithful}
Let $S$ be a profinite set and let $X$ be a scheme. Then the functor
\begin{align*}
\Phi_S \colon \big(\!\AffProEt_{/X} \underset{\ProFin}{\vec\times} \{S\}\big)\op &\to \Fun^\lex(\Sh(X_\et),\Sh(S)) \\
Y &\mapsto \Big( \mathscr F \mapsto \big(U \mapsto \Gamma(Y \times_S U,\mathscr F) \big) \Big)
\end{align*}
is fully faithful.
\end{Cor}

\begin{proof}
By \autoref{Lem cosheaf in schemes}, we have
\[
\big(\!\AffProEt_{/X} \underset{\ProFin}{\vec\times} \{S\}\big)\op \simeq \CSh(S,\AffProEt_{/X})\op \simeq \Sh(S,\AffProEt_{/X}\op).
\]
Thus, \autoref{Lem fully faithful} shows that the functor to $\Sh(S,\Fun^\lex(\Sh(X_\et),\Set))$ is fully faithful. We have a diagram of fully faithful functors
\[
\begin{tikzcd}
\Sh(S,\Fun^\lex(\Sh(X_\et),\Set)) \ar[hookrightarrow]{d} & \Fun^\lex(\Sh(X_\et),\Sh(S)) \ar[hookrightarrow]{d} \\
\PSh(S,\Fun(\Sh(X_\et),\Set)) \ar{r}{\sim} & \Fun(\Sh(X_\et),\PSh(S))\punct{.}
\end{tikzcd}
\]
Under the bottom equivalence, the full subcategories coming from the vertical arrows agree: if $F \colon \Clopen(S)\op \times \Sh(X_\et) \to \Set$ is a functor, it is an $S$-sheaf in $\Fun(\Sh(X_\et),\Set)$ if and only if $F(-,\mathscr F)$ is an $S$-sheaf for all $\mathscr F \in \Sh(X_\et)$, and then the induced functor $\Sh(X_\et) \to \Sh(S)$ is left exact if and only if $F(U,-)$ is left exact for all $U \in \Clopen(S)$.
\end{proof}

\begin{Def}\label{Def cosheaf representable}
Let $X$ be a scheme, and $S$ a profinite set. We say that a left exact functor $F \colon \Sh(X_\et) \to \Sh(S)$ is \emph{pro-representable} if it is in the essential image of the fully faithful functor $\Phi_S \colon (\AffProEt_{/X} \vec\times_{\ProFin} \{S\})\op \to \Fun^\lex(\Sh(X_\et),\Sh(S))$ of \autoref{Cor cosheaf fully faithful}.
\end{Def}

\begin{Par}\label{Par pushforward}
Let $X$ be a scheme and let $\alpha \colon S \to T$ be a continuous map of profinite sets. There is a functor $\alpha_* \colon \AffProEt_{/X} \vec\times_{\ProFin} \{S\} \to \AffProEt_{/X} \vec\times_{\ProFin} \{T\}$ taking the structure map $\pi_0(Y) \to S$ to the composition $\pi_0(Y) \to S \to T$. The diagram
\[
\begin{tikzcd}
\AffProEt_{/X} \underset{\ProFin}{\vec\times} \{S\} \ar{r}{\Phi_S}\ar{d}[swap]{\alpha_*} & \Fun^\lex(\Sh(X_\et),\Sh(S))\op \ar{d}{\alpha_* \circ -} \\
\AffProEt_{/X} \underset{\ProFin}{\vec\times} \{T\} \ar{r}[swap]{\Phi_T} & \Fun^\lex(\Sh(X_\et),\Sh(T))\op
\end{tikzcd}
\]
commutes (up to natural isomorphism). Indeed, both compositions take an object \makebox{$Y \in \AffProEt_{/X} \vec\times_{\ProFin} \{S\}$} to the functor
\begin{align*}
\Clopen(T)\op \times \Sh(X_\et) &\to \Set \\
(U,\mathscr F) &\mapsto \Gamma\big(Y \times_T U,\mathscr F\big) \simeq \Gamma\big(Y \times_S \alpha^{-1}(U),\mathscr F\big).
\end{align*}
\end{Par}

\subsection{Topos-theoretic Henselisation}\label{Sec hensel}
If $X$ is any scheme, then the points of the \'etale site correspond to strict Henselisations along geometric points \cite[Exp.~VIII, Thm.~7.9]{SGA4II}. In this section, we generalise this to a classification of geometric morphisms $s^* \colon \Sh(X_\et) \to \Sh(S)$ for a (fixed) profinite set~$S$; see \autoref{Prop equivalence S}. In the next section, this will be used to compute the unstraightening of the condensed category of points.

\begin{Def}\label{Def Henselisation}
Let $X$ be a scheme, let $S \in \ProFin$, and let $s_* \colon \Sh(S) \to \Sh(X_\et)$ be a geometric morphism. The \emph{Henselisation} of $X$ along $s$ is the pro-object $X_{(s)} \in \Pro(\Sh(X_\et))$ corepresenting the left exact functor $\Gamma_* s^* \colon \Sh(X_\et) \to \Set$.
\end{Def}

It would arguably be more natural to view $X_{(s)}$ as an $S$-cosheaf in $\Pro(\Sh(X_\et))$ right away. In \autoref{Lem Henselisation representable}, we show that $X_{(s)}$ is representable in $\AffProEt_{/X}$, so by \autoref{Lem cosheaf in schemes} and \ref{Par oriented fibre product}, an $S$-cosheaf structure is the same thing as a map $\pi_0(X_{(s)}) \to S$. We think it is notationally beneficial to write $X_{(s)}$ for the underlying scheme and later equip it with this extra structure. (In fact, in \autoref{Lem pi0 iso}, we will show that the map $\pi_0(X_{(s)}) \to S$ is an isomorphism, so that $X_{(s)}$ lives in $\AffProEt_{/X} \times_{\ProFin} \{S\}$.)

\begin{Par}\label{Par profinite section}
Recall that if $S$ is a profinite set then any epimorphism $\mathscr F \to \mathscr G$ in $\Sh(S)$ is surjective on global sections (for instance, this follows easily from \cite[Tag \href{https://stacks.math.columbia.edu/tag/08ZZ}{08ZZ}]{Stacks}). In particular, every effective epimorphism $\mathscr F \twoheadrightarrow \mathbf 1$ in $\Sh(S)$ admits a section.
\end{Par}

\begin{Lemma}\label{Lem Henselisation representable}
Let $X$ be a scheme, let $S \in \ProFin$, and let $s_* \colon \Sh(S) \to \Sh(X_\et)$ be a geometric morphism. Then $X_{(s)}$ is representable by a scheme in $\AffProEt_{/X}$.
\end{Lemma}

\begin{proof}
Write $\mathcal I$ for the comma category $(\mathbf 1 \downarrow s^*)$ consisting of pairs $(\mathscr F,\alpha)$ of a sheaf $\mathscr F \in \Sh(X_\et)$ and a section $\alpha \colon \mathbf 1 \to s^*\mathscr F$. Note that $\mathcal I$ has finite limits (in fact, the forgetful functor $\mathcal I \to \Sh(X_\et)$ creates finite limits), so in particular it is cofiltered. Moreover, we may write $X_{(s)}$ as the limit
\[
X_{(s)} \simeq \lim_{\substack{\longleftarrow \\(\mathscr F,\alpha) \in \mathcal I}}  \mathscr F.
\]
We claim that the full subcategory of $\mathcal I$ spanned by objects where $\mathscr F$ is representable by an affine scheme is coinitial. Since $\mathcal I$ has finite limits, it suffices to see that for any $(\mathscr F,\alpha) \in \mathcal I$, there exists a map from some object of the form $(\Spec(A),\beta)$. We may pick a surjection $\coprod_{i \in I} \Spec(A_i) \twoheadrightarrow \mathscr F$ for some set $I$ since $\Sh(X_\et)$ is generated by $\AffEt_{/X}$. Applying $s^*$ and pulling back along $\alpha$ gives an effective epimorphism
\[
\coprod_{i \in I} \Big( s^*\Spec(A_i) \underset{s^*\mathscr F}\times \mathbf 1\Big) \cong s^*\Big(\coprod_{i \in I}\Spec(A_i)\Big) \underset{s^*\mathscr F}\times \mathbf 1 \twoheadrightarrow \mathbf 1.
\]
Since $\mathbf 1 \in \Sh(S)$ is quasi-compact, the map $\coprod_{j \in J} s^*\Spec(A_j) \times_{s^*\mathscr F} \mathbf 1 \twoheadrightarrow \mathbf 1$ is still an epimorphism for some finite subset $J \subseteq I$. Setting $A = \prod_{j \in J} A_j$, it follows from \ref{Par profinite section} that there is a section $\mathbf 1 \to s^*\Spec(A) \times_{s^*\mathscr F} \mathbf 1$. This section induces a commutative triangle
\[
\begin{tikzcd}[row sep=1.2em,column sep=1.2em]
 & s^*\Spec(A) \ar{d} \\
\mathbf 1 \ar{r}[swap]{\alpha}\ar{ru}{\beta} & s^*\mathscr F\punct{,}
\end{tikzcd}
\]
as desired.
\end{proof}

\begin{Cor}\label{Cor cosheaf representable}
Let $X$ be a scheme, let $S \in \ProFin$, and let $s_* \colon \Sh(S) \to \Sh(X_\et)$ be a geometric morphism. Then the left exact functor $s^* \colon \Sh(X_\et) \to \Sh(S)$ is pro-represented by an object $\AffProEt_{/X} \vec\times_{\ProFin} \{S\}$.
\end{Cor}

\begin{proof}
For any clopen subset $U \subseteq S$, write $(s|_U)_* \colon \Sh(U) \to \Sh(X_\et)$ for the composite geometric morphism $\Sh(U) \hookrightarrow \Sh(S) \to \Sh(X_\et)$. Applying \autoref{Lem Henselisation representable} to $s|_U$ shows that $\Gamma(U,s^*(-)) \colon \Sh(X_\et) \to \Set$ is pro-representable for all $U \in \Clopen(S)$. Thus, $s^*$ is in the essential image of the functor $(\AffProEt_{/X} \vec\times_{\ProFin} \{S\})\op \to \Fun^\lex(\Sh(X_\et),\Sh(S))$ of \autoref{Cor cosheaf fully faithful}.
\end{proof}

Thus, $X_{(s)}$ comes equipped with a structure map $\pi_0(X_{(s)}) \to S$.

\begin{Lemma}\label{Lem pi0 iso}
Let $X$ be a scheme, let $S \in \ProFin$, and let $s_* \colon \Sh(S) \to \Sh(X_\et)$ be a geometric morphism. Then the structure map $\pi_0(X_{(s)}) \to S$ is an isomorphism.
\end{Lemma}

\begin{proof}
Note that $\pi_0(X_{(s)})$ is profinite since $X_{(s)}$ is affine (see \ref{Par pi0 profinite}). If~$T$ is a finite set, the universal property of $X_{(s)}$ gives
\begin{align*}
\Hom_{\ProFin}(\pi_0(X_{(s)}),T) \cong \Hom_X(X_{(s)},T \times X) \cong \Hom_S(\mathbf 1,T \times \mathbf 1) \cong \Hom_{\ProFin}(S,T).
\end{align*}
Thus, the same holds for profinite sets by taking limits, and the result follows from the Yoneda lemma.
\end{proof}

\begin{Lemma}\label{Lem Henselisation base change}
Let $X$ be a scheme, let $f \colon T \to S$ be a continuous map of profinite sets, and let $s_* \colon \Sh(S) \to \Sh(X_\et)$ be a geometric morphism. Then the natural map $X_{(s \circ f)} \to X_{(s)} \times_S T$ is an isomorphism.
\end{Lemma}

\begin{proof}
It suffices to show that $X_{(s)} \times_S T$ represents the functor $\Gamma_*f^*s^* \colon \Sh(X_\et) \to \Set$. When $f$ is a clopen immersion with complement $f' \colon T' \to S$, the $S$-cosheaf condition gives $X_{(s)} = X_{(s \circ f)} \amalg X_{(s \circ f')}$, so $X_{(s \circ f)} = X_{(s)} \times_S T$ by naturality of the map of \autoref{Lem pi0 iso}. When $T = \coprod_{i=1}^n U_i$ for clopens $f_i \colon U_i \hookrightarrow S$, then $X_{(s \circ f)} = \coprod_{i=1}^n X_{(s \circ f_i)} = X_{(s)} \times_S T$. Finally, an arbitrary continuous map $T \to S$ is a cofiltered limit of such, so the result follows since the functor $\AffProEt_{/X} \to \Fun^\lex(\Sh(X_\et),\Set)\op$ of \autoref{Lem fully faithful} preserves cofiltered limits.
\end{proof}

\begin{Lemma}\label{Lem Henselisation acyclic}
Let $X$ be a scheme, let $S \in \ProFin$, and let $s_* \colon \Sh(S) \to \Sh(X_\et)$ be a geometric morphism. Then $X_{(s)}$ is acyclic.
\end{Lemma}

\begin{proof}[First proof]
Let $f \colon U \twoheadrightarrow X_{(s)}$ be an \'etale surjection. By \autoref{Lem Henselisation representable}, the scheme $X_{(s)}$ is affine, so we may replace $U$ by a finite disjoint union of its affine opens to assume $U$ is affine, and hence $f$ is of finite presentation. Since $X_{(s)}$ is a cofiltered limit of affine \'etale $X$-schemes, we may find some $X_{(s)} \to T$ with $T \in \AffEt_{/X}$ and some \'etale surjection $f' \colon V \twoheadrightarrow T$ whose pullback to $X_{(s)}$ is $f$. Thus, finding a section of $f$ is the same as finding a factorisation of $X_{(s)} \to T$ into  $X_{(s)} \to V \to T$. By the universal property of $X_{(s)}$, this corresponds to finding a factorisation
\[
\begin{tikzcd}[row sep=1.2em,column sep=1em]
& s^*V \ar[two heads]{d}\\
\mathbf 1 \ar{r}\ar{ru} & s^*T\punct{,}
\end{tikzcd}
\]
which exists by \ref{Par profinite section}.
\end{proof}

Alternatively, we may use that we already know this result when $S$ is a point (the proof of which is essentially identical to the first proof above):

\begin{proof}[Second proof]
If $i \colon \{*\} \to S$ is a point, then \autoref{Lem Henselisation base change} implies that $X_{(s)} \times_S \{*\} = X_{(s \circ i)}$ is the Henselisation along a geometric morphism $\Set \to \Sh(X_\et)$, which is the spectrum of a strictly Henselian local ring by \cite[Exp.~VIII, Thm.~7.9]{SGA4II}. By \autoref{Lem pi0 iso}, this means that the connected components of $X_{(s)}$ are spectra of strictly Henselian local rings, which implies that $X_{(s)}$ is acyclic by \cite[Prop.~3.2]{Artin}.
\end{proof}

Recall from \autoref{Def AffProEt} that $\AffProEt_{/X}^{\acyc}$ denotes the category of acyclic objects in $\AffProEt_{/X}$. By \autoref{Lem pi0 iso} and \autoref{Lem Henselisation acyclic}, we see that $X_{(s)}$ gives an object of $\AffProEt_{/X}^{\acyc} \times_{\ProFin} \{S\}$. We provide a construction in the opposite direction as well.

\begin{Lemma}\label{Lem acyclic local}
Let $W$ be an acyclic (resp.~w-strictly local) scheme, and set $S = \pi_0(W)$. Then the geometric morphism $u_* \colon \Sh(W_\et) \to \Sh(S)$ (resp.~$u_* \colon \Sh(W_\proet) \to \Sh(\ProFin_{/S})$) induced by the continuous functor $u \colon \Clopen(S) \to \AffEt_{/W}$ (resp.~$u \colon \ProFin_{/S} \to \AffProEt_{/W}$) given by $T \mapsto W \times_S T$ is local. 
\end{Lemma}

In other words, the pushforward $u_*$ has a fully faithful right adjoint $u^!$ \cite[Prop.~1.4]{JohnstoneMoerdijk} (using that for $u^* \dashv u_* \dashv u^!$, the functor $u^*$ is fully faithful if and only if $u^!$ is). Thus, $u$ defines a map $u^! \colon \Sh(S) \to \Sh(W_\et)$ (resp.~$u^! \colon \Sh(\ProFin_{/S}) \to \Sh(W_\proet)$) in $\RTop$.

\begin{proof}
We claim that $u$ is cocontinuous \cite[Tag \href{https://stacks.math.columbia.edu/tag/00XJ}{00XJ}]{Stacks}. For $u \colon \Clopen(S) \to \AffEt_{/W}$, the image $u(V) = W \times_S V$ is acyclic for every $V \in \Clopen(S)$ by \autoref{Lem acyclic}. Thus, every \'etale cover $\{U_i \to u(V)\}_{i \in I}$ can be refined by a cover of $u(V)$ by clopens by choosing a section of $\coprod_{i \in I} U_i \twoheadrightarrow u(V)$. As $\Clopen(S) \simeq \Clopen(W)$, we see that all clopens are in the image of $u$, proving that $u$ is cocontinuous. In the w-strictly local case, the functor $u \colon \ProFin_{/S} \to \AffProEt_{/S}$ is cocontinuous by \autoref{Lem w-local topology} and \autoref{Cor equiv wsl}. Thus, in both cases, $u_*$ has a right adjoint $u^!$ by \cite[Tag \href{https://stacks.math.columbia.edu/tag/00XR}{00XR}]{Stacks}. Note that $u$ is fully faithful by \autoref{Cor equiv wsl} (in the case of clopens, this is also easily seen directly). We conclude that $u^!$ is fully faithful on presheaves \cite[Exp.~I, Prop.~5.6(iv)]{SGA4I}, hence also on sheaves since this is a full subcategory and $u^!$ sends sheaves to sheaves.
\end{proof}

\begin{Cor}\label{Cor geometric morphism}
Let $X$ be a scheme, let $W \in \AffProEt_{/X}^{\acyc}$, and let $S = \pi_0(W)$. Then the left exact functor\vspace{-1em}
\begin{align*}
s^* = \Phi_S(W) \colon \Sh(X_\et) &\to \Sh(S) \\
\mathscr F &\mapsto \big(U \mapsto \Gamma(W \times_S U,\mathscr F)\big)
\end{align*}
of \autoref{Cor cosheaf fully faithful} comes from a geometric morphism of topoi $s_* \colon \Sh(S) \to \Sh(X_\et)$. If $W$ is \makebox{w-strictly} local, it extends to a geometric morphism of topoi $s_* \colon \Sh(\ProFin_{/S}) \to \Sh(X_\proet)$ such that $s^*\mathscr F$ is given by $(T \to S) \mapsto \Gamma(W \times_S T,\mathscr F)$ for all $\mathscr F \in \Sh(X_\proet)$.
\end{Cor}

\begin{proof}
Take $s_*$ to be the composition of the geometric morphism $u^! \colon \Sh(S) \to \Sh(W_\et)$ (resp.~$u^! \colon \Sh(\ProFin_{/S}) \to \Sh(W_\proet)$) of \autoref{Lem acyclic local} with the geometric morphism $\Sh(W_\et) \to \Sh(X_\et)$ (resp.~$\Sh(W_\proet) \to \Sh(X_\proet)$) induced by the map $W \to X$.
\end{proof}

\begin{Rmk}\label{Rmk open}
The pullback $u_* \colon \Open(W) \to \Open(S)$ induced by the geometric morphism $u^! \colon \Sh(S) \to \Sh(W_\et)$ above takes an open $U \subseteq W$ to the open $U \cap W^\cl$ in $W^\cl \simeq \pi_0(W) \simeq S$ (under the identification of \autoref{Lem acyclic}(d)). Indeed, for $U \subseteq W$ open and $V \subseteq S$ clopen, we have $\Gamma(V,u_*\mathbf 1_U) \simeq \Gamma(u^{-1}(V),\mathbf 1_U)$, so $V \subseteq u_*U$ if and only if $u^{-1}(V) \subseteq U$, which is in turn equivalent to $V \subseteq U \cap W^\cl$ (i.e., $U$ contains the unique closed point of every connected component of $u^{-1}(V)$) since $U$ is stable under generisation.
\end{Rmk}

\begin{Rmk}\label{Rmk wsl coherent}
If $W$ is w-strictly local, there is a more geometric construction of the geometric morphisms $u^! \colon \Sh(S) \to \Sh(W_\et)$ and $u^! \colon \Sh(\ProFin_{/S}) \to \Sh(W_\proet)$ of \autoref{Lem acyclic local}. Indeed, the set of closed points $i \colon W^\cl \hookrightarrow W$ is closed, so we may endow it with the reduced induced scheme structure, which gives a strictly profinite scheme by \autoref{Rmk wsl}. By \autoref{Lem acyclic}(d), we may identify $\lvert W^\cl\rvert$ with $\pi_0(W) \simeq S$, and by \autoref{Lem strictly profinite}, we get an equivalence of topoi $\Sh(W^\cl_\et) \simeq \Sh(S)$. Likewise, there is a canonical identification $\Sh(W^\cl_\proet) \simeq \Sh(\ProFin_{/S})$. We claim that $u^!$ is the natural map of \'etale topoi $\Sh(W^\cl_\et) \to \Sh(W_\et)$ (resp.~of pro-\'etale topoi $\Sh(W^\cl_\et) \to \Sh(W_\proet)$). Indeed, for every $T \in \ProFin_{/S}$, the inclusion $W^\cl \times_S T \hookrightarrow W \times_S T$ is a Henselian pair by \autoref{Lem wsl Henselian} and \autoref{Lem acyc profinite}. Thus, for $\mathscr F \in \Sh(W_\et)$ (resp.~$\mathscr F \in \Sh(W_\proet)$), we have $\Gamma(W \times_S T,\mathscr F) \cong \Gamma(W^\cl \times_S T,i^*\mathscr F)$ by \cite[Tag \href{https://stacks.math.columbia.edu/tag/09ZH}{09ZH}]{Stacks}. In other words, the functor $u^! \colon \Sh(W_\et) \to \Sh(S)$ (resp.~$u^! \colon \Sh(W_\proet) \to \Sh(\ProFin_{/S})$) constructed in the proof of \autoref{Lem acyclic local} is isomorphic to $i^*$.
\end{Rmk}

\begin{Prop}\label{Prop equivalence S}
Let $X$ be a scheme and let $S$ be a profinite set. Then the fully faithful functor $\Phi_S \colon (\AffProEt_{/X} \vec\times_{\ProFin} \{S\})\op \to \Fun^\lex(\Sh(X_\et),\Sh(S))$ of \autoref{Cor cosheaf fully faithful} restricts to an equivalence
\[
\AffProEt_{/X}^{\acyc} \underset{\ProFin}\times \{S\} \xrightarrow\sim \Fun^*(\Sh(X_\et),\Sh(S))\op \simeq \Fun_*(\Sh(S),\Sh(X_\et)).
\]
\end{Prop}

\begin{proof}
By \autoref{Cor geometric morphism}, the image lands in $\Fun^*(\Sh(X_\et),\Sh(S))$. Full faithfulness is inherited from \autoref{Cor cosheaf fully faithful}, and essential surjectivity follows from \autoref{Cor cosheaf representable}, \autoref{Lem pi0 iso}, and \autoref{Lem Henselisation acyclic}.
\end{proof}

When $W$ is w-strictly local, we get a bound on $\size(W)$ in terms of $\size(S)$.

\begin{Lemma}\label{Lem size Hens}
Let $X$ be a scheme, let $W \in \AffProEt_{/X}^{\wsl}$, and set $S = \pi_0(W)$. Then
\[
\size(W) \leq \max(\size(X),\size(S)).
\]
In particular, if $\kappa > \size(X)$ is any cardinal, then $\size(W) < \kappa$ if and only if $\size(S) < \kappa$.
\end{Lemma}

We do not know if the same statement also holds for $W \in \AffProEt_{/X}^{\acyc}$.

\begin{proof}
The final statement follows from the first and \autoref{Lem pi0 kappa}. For the first statement, let $\kappa = \max(\size(X),\size(S))$, so that $\size(X) \leq \kappa$ and $\size(S) \leq \kappa$. If $X = \bigcup_{i \in I} X_i$ is an affine open cover, then $\size(X_i) \leq \kappa$ for all $i$. We may choose a finite clopen cover $W = \coprod_{j \in J} W_j$ refining the cover $W = \bigcup_{i \in I} W \times_X X_i$. If $S_j = \pi_0(W_j)$, we have $\size(S_j) \leq \kappa$ for all $j \in J$, and it suffices to show that $\size(W_j) \leq \kappa$ for all $j \in J$ since $\Gamma(W,\mathscr O_W) = \prod_{j \in J} \Gamma(W_j,\mathscr O_{W_j})$. In other words, replacing $X$, $W$, and $S$ by $X_i$, $W_j$, and $S_j$, we reduce to the case where $X$ is affine.

By \autoref{Prop equivalence S} and \autoref{Rmk wsl coherent}, we know that $W$ is obtained as the Henselisation $X_{(s)}$ along the geometric morphism $s_* \colon \Sh(S) \simeq \Sh(W^\cl_\et) \to \Sh(X_\et)$. Since $X_{(s)}$ is a pro-object in $\AffEt_{/X}$, it is canonically isomorphic to the limit
\[
X_{(s)} \simeq \lim_{\substack{\longleftarrow \\ (U,\alpha) \in (W \downarrow \AffEt_{/X})}} U
\]
taken over the comma category $(W \downarrow \AffEt_{/X})$ of pairs $(U,\alpha)$ with $U \in \AffEt_{/X}$ and $\alpha \in \Hom_X(W,U) \simeq \Gamma(S,s^*U)$. By \autoref{Rmk colimit kappa small}, it suffices to show that each $U \in \AffEt_{/X}$ has $\size(U) \leq \kappa$ and the category $(W \downarrow \AffEt_{/X})$ has $\leq \kappa$ objects (up to isomorphism).

Since $\size(X) \leq \kappa$, we conclude that the set of isomorphism classes in $\AffEt_{/X}$ has cardinality $\leq \kappa$, and any $U \in \AffEt_{/X}$ has $\size(U) \leq \kappa$ since it is finitely presented over~$X$. Finally, if $U \in \AffEt_{/X}$, we have to show that $\lvert \Gamma(S,s^*U) \rvert \leq \kappa$. Since $U \to X$ is an \'etale map of affine schemes, it is finitely presented. Thus, the map $V = U \times_X W^\cl \to W^\cl$ is a finitely presented \'etale morphism, hence there exists a finite constructible stratification of $W^\cl$ over which $V \to W^\cl$ is finite \'etale \cite[Tag \href{https://stacks.math.columbia.edu/tag/03S0}{03S0}]{Stacks}. Since $W^\cl$ is strictly profinite (\autoref{Rmk wsl}), any constructible subset of $W^\cl$ is clopen, and we conclude that $V \to W^\cl$ is finite \'etale. This forces $V$ to be isomorphic to a clopen inside $W^\cl \times \{1,\ldots,d\}$ for some $d \in \mathbf N$. In other words, $s^*U$ is a summand of a constant sheaf $S \times \{1,\ldots, d\}$, and
\[
\lvert \Gamma(S,s^*U) \rvert \leq \lvert \Hom_{\ProFin}(S,\{1,\ldots,d\}) \rvert \leq \lvert \Clopen(S) \rvert^d \leq \kappa,
\]
since $\Hom_{\ProFin}(S,\{1,\ldots,d\})$ is the set of $d$-tuples of pairwise distinct clopens in $S$ and $\lvert \Clopen(S) \rvert \leq \size(S)$ by \autoref{Def size profinite}.
\end{proof}

\subsection{The Galois category in terms of w-strictly local schemes}\label{Sec Galois}
In \autoref{Prop equivalence S}, we showed that geometric morphisms $s^* \colon \Sh(X_\et) \to \Sh(S)$ for a profinite set $S$ are classified by $S$-cosheaves in $\AffProEt_{/X}^{\acyc}$, generalising the result for $S = *$ from \cite[Exp.\ VIII, Thm.\ 7.9]{SGA4II}. In this section, we upgrade this to a statement about the unstraightening (= Grothendieck construction) of the condensed category of points; see \autoref{Thm equivalence} and \autoref{Cor loccoh}. As mentioned at the start of \hyperref[Sec condensed]{\S2}, the result \autoref{Thm equivalence} itself can be deduced from \cite[Thm.~6.3.14]{LurieUltra}, but our geometric proof gives finer control over the way the equivalence is constructed.

\begin{Def}\label{Def Stone over X}
If $\mathscr X$ is a topos, write $\Stone_{/\mathscr X}$ for the Grothendieck construction of $\PTS_*(\mathscr X) \simeq \PTS^*(\mathscr X)\op$. We may understand its objects and morphisms as follows:
\begin{itemize}
\item The objects of $\Stone_{/\mathscr X}$ are pairs $(S,s^*)$ of a profinite set $S$ and a geometric morphism $s^* \colon \mathscr X \to \Sh(S)$;
\item If $(S,s^*), (T,t^*) \in \Stone_{/\mathscr X}$, then a morphism $(\alpha,\nu) \colon (S,s^*) \to (T,t^*)$ consists of a continuous map $\alpha \colon S \to T$ and a natural transformation $\nu \colon t^* \to \alpha_*s^*$ (which is adjoint to a natural transformation $\alpha^*t^* \to s^*$).
\end{itemize}
In the special case $\mathscr X = \Sh(X_\et)$ for some scheme $X$, we will write $\Stone_{/X}$ for $\Stone_{/\Sh(X_\et)}$, and we will write $\Stone_{/X}^\loccoh$ for the full subcategory on locally coherent geometric morphisms $s^* \colon \Sh(X_\et) \to \Sh(S)$ (see \ref{Par coherent}).
\end{Def}

\begin{Def}
Let $X$ be a scheme. Define a functor $\Phi \colon \AffProEt_{/X}^{\acyc} \to \Stone_{/X}$ as follows.
\begin{itemize}
\item If $W \in \AffProEt_{/X}^{\acyc}$, set $S = \pi_0(W)$ and define $\Phi(W)$ to be the geometric morphism $s^* = \Phi_S(W) \colon \Sh(X_\et) \to \Sh(S)$ of \autoref{Cor geometric morphism}. Explicitly, $s^* \colon \Sh(X_\et) \to \Sh(S)$ takes a sheaf $\mathscr F \in \Sh(X_\et)$ to $U \mapsto \Gamma(W \times_S U,\mathscr F)$.
\item Let $f \colon V \to W$ be a map in $\AffProEt_{/X}^{\acyc}$ inducing the map $\alpha \colon S \to T$ after applying~$\pi_0$. Then $f$ is a map $\alpha_*V \to W$ in $\AffProEt_{/X} \vec\times_{\ProFin} \{T\}$ (see \ref{Par pushforward} for notation), and we define $\Phi(f) = \Phi_T(f)$ in the notation of \autoref{Cor geometric morphism}. Explicitly, if $\Phi(V) = (S,s^*)$ and $\Phi(W) = (T,t^*)$, then $\Phi(f)$ is the map $(\alpha,\nu) \colon (S,s^*) \to (T,t^*)$ given by the natural transformation $\nu \colon t^* \to \alpha_*s^*$ given on $U \in \Clopen(T)$ by the pullback map $\Gamma(W \times_T U,-) \to \Gamma(V \times_S \alpha^{-1}(U),-)$.
\end{itemize}
\end{Def}

\begin{Thm}\label{Thm equivalence}
If $X$ is a scheme, then the functor $\Phi \colon \AffProEt_{/X}^{\acyc} \to \Stone_{/X}$ is an equivalence.
\end{Thm}

\begin{proof}
By construction, the triangle
\[
\begin{tikzcd}[row sep=1em,column sep=-1em]
\AffProEt_{/X}^{\acyc} \ar{rr}{\Phi}\ar{rd}[swap]{\pi_0} & & \Stone_{/X} \ar{ld} \\
 & \ProFin &
\end{tikzcd}
\]
commutes: if $f \colon V \to W$ is a map in $\AffProEt_{/X}^{\acyc}$, the morphism of profinite sets underlying $\Phi(f)$ is $\pi_0(f)$. Thus, for fully faithfulness, it suffices to restrict to morphisms living above a fixed map $\alpha \colon S \to T$ in $\ProFin$. Let $V \in \AffProEt_{/X}^{\acyc} \times_{\ProFin} \{S\}$ and $W \in \AffProEt_{/X} \times_{\ProFin} \{T\}$. Then a morphism $f \colon V \to W$ lifts $\alpha$ if and only if the diagram
\[
\begin{tikzcd}
\pi_0(V) \ar{r}{\pi_0(f)}\ar{d}[rotate=90,anchor=north,xshift=.1em]{\sim} & \pi_0(W) \ar{d}[rotate=90,anchor=north,xshift=.1em]{\sim} \\
S \ar{r}[swap]{\alpha} & T
\end{tikzcd}
\]
commutes. These are given by $\Hom_{\AffProEt_{/X} \vec\times_{\ProFin} \{T\}}(\alpha_*V,W)$, and fully faithfulness follows from \autoref{Cor cosheaf fully faithful} and \ref{Par pushforward}. Likewise, for essential surjectivity, it suffices to restrict to objects living above a fixed profinite set $S$, where the result was proved in \autoref{Prop equivalence S}.
\end{proof}

\begin{Par}\label{Par cartesian}
In particular, the functor $\pi_0 \colon \AffProEt_{/X}^{\acyc} \to \ProFin$ is a cartesian fibration. A map $V \to W$ in $\AffProEt_{/X}^{\acyc}$ is $\pi_0$-cartesian if and only if the map $V \to W \times_{\pi_0(W)} \pi_0(V)$ is an isomorphism. This is more or less immediate from the definitions, but also follows from the theorem above and \autoref{Lem Henselisation base change}. By \autoref{Lem acyc profinite} and \autoref{Lem acyc profinite converse}, these are exactly the maps $V \to W$ that take closed points to closed points.
\end{Par}

For the pro-\'etale exodromy correspondence, we will want to use $\GAL(X)$ instead of $\PTS^*(X_\et)$; see \autoref{Def Galois category}. The importance of locally coherent geometric morphisms is explained by the following lemma.

\begin{Lemma}\label{Lem coherent}
Let $X$ be a scheme, let $W \in \AffProEt_{/X}^{\acyc}$, and let $(S,s^*) = \Phi(W)$. Then $W$ is w-strictly local if and only if the geometric morphism $s_* \colon \Sh(S) \to \Sh(X_\et)$ is locally coherent.
\end{Lemma}

\begin{proof}
If $W$ is w-strictly local, then $s_*$ is the geometric morphism $\Sh(W^\cl_\et) \to \Sh(X_\et)$ induced by the morphism of schemes $W^\cl \to X$ by \autoref{Rmk wsl coherent}, and is therefore locally coherent. Conversely, assume that $s_*$ is locally coherent. Writing $W$ as a cofiltered limit $\lim_{i \in I} U_i$ for $U_i \in \AffEt_{/X}$, we know that $s_*$ factors via \makebox{$\Sh(S) \to \Sh(W_\et) \to \Sh(U_{i,\et})$.} Since $\Sh(U_{i,\et})$ is coherent, the hypothesis implies that $\Sh(S) \to \Sh(U_{i,\et})$ is coherent. Since $\Sh(W_\et)^\coh = \colim_{i \in I\op} \Sh(U_{i,\et})^\coh$ \cite[Exp.~IX, Cor.~2.7.4]{SGA4III}, the induced geometric morphism $g_* = u^! \colon \Sh(S) \to \Sh(W_\et)$ of \autoref{Lem acyclic local} is also coherent. If $w \in W\setminus W^\cl$, we need to show that there exists an open neighbourhood of $w$ missing~$W^\cl$. By \autoref{Rmk open}, the pullback $g^* = u_* \colon \Open(W) \to \Open(S)$ is given by $U \mapsto U \cap W^\cl$. Let $x$ be the unique closed point in the connected component $u(w)$ of~$w$, and let $U \subseteq W$ be a quasi-compact open containing $w$ but not~$x$. Since $g_*$ is coherent, the open set $V = g^*U = U \cap W^\cl$ is clopen, hence $U' = U \cap u^{-1}(V^c)$ is an open such that $U' \cap W^\cl = \varnothing$. Since $x \not\in U$, we see that $u(w) = u(x) \not \in V$, so $w \in U'$. Thus, $U'$ is an open neighbourhood of $w$ missing $W^\cl$.
\end{proof}

\begin{Ex}\label{Ex acyclic not wsl}
Let $R$ be a Henselian $\mathbf Q$-valued valuation ring with separably algebraically closed residue field and fraction field (for instance $R = \mathscr O_{\mathbf C_p}$), and let $X = \Spec R$. Then $\Sh(X_\et) \simeq \Sh(X_\Zar)$ by \autoref{Lem strictly profinite}, and $X_\Zar$ is a Sierpi\'nski space. Thus, geometric morphisms $s_* \colon \Sh(S) \to \Sh(X_\Zar)$ from a profinite set $S$ are in bijection with open sets $U \subseteq S$, corresponding to the preimage of the generic point in $X_\Zar$. Since the generic point is quasi-compact, we see that $s_*$ is (locally) coherent if and only if $U$ is clopen. This gives many examples of geometric morphisms that are not locally coherent, so \autoref{Lem coherent} gives many examples of acyclic schemes that are not w-strictly local.

For instance, if $S = \mathbf N \cup \{\infty\}$, we can take the map $\Sh(S) \to \Sh(X_\Zar)$ given by the open set $\mathbf N$. The Henselisation $X_{(s)}$ is the pro-(open immersion) $W \subseteq X \times S$ given by $X \times \{\infty\} \cup \{\eta\} \times S$, where $\eta \in X$ denotes the generic point. The point $(\eta,\infty)$ is not closed, but is in the closure of $\{\eta\} \times S$, hence also in the closure of $W^\cl$. 
\end{Ex}

\begin{Cor}\label{Cor loccoh}
The equivalence $\Phi$ of \autoref{Thm equivalence} restricts to an equivalence
\[
\Phi^{\loccoh} \colon \AffProEt_{/X}^{\wsl} \xrightarrow\sim \Stone_{/X}^{\loccoh} \tag*{\qedsymbol} 
\]
\normalfont Note that by \autoref{Cor wsl affproet}, the left hand side agrees with $\ProEt_{/X}^\wsl$.
\end{Cor}

For later use, we also record a cardinal-sensitive version of \autoref{Cor loccoh}.

\begin{Def}
Let $X$ be a scheme, and let $\kappa > \size(X)$ be a cardinal. Write $\GAL_\kappa(X)$ for the restriction of $\GAL(X)$ to $\ProFin_\kappa$. Write $\Stone_{/X}^{\kappa,\loccoh} \subseteq \Stone_{/X}^{\loccoh}$ for the full subcategory spanned by those pairs $(S,s^*)$ with $S \in \ProFin_\kappa$ (i.e., $\size(S) < \kappa$), which is equivalently the Grothendieck construction of $\GAL_\kappa(X)$.
\end{Def}

By \autoref{Lem size Hens}, we therefore get:

\begin{Cor}\label{Cor loccoh size}
Let $X$ be a scheme, and let $\kappa > \size(X)$ be a cardinal. Then the equivalence of \autoref{Thm equivalence} restricts to an equivalence
\[
\Phi^{\kappa,\loccoh} \colon \AffProEt_{/X}^{\kappa,\wsl} \xrightarrow\sim \Stone_{/X}^{\kappa,\loccoh} \tag*{\qedsymbol} 
\]
\normalfont In \hyperref[Sec change of cardinal]{\S4.2}, we show that knowing $\GAL_\kappa(X)$ for a single value of $\kappa > \size(X)$ determines $\GAL_\lambda(X)$ for every cardinal $\lambda > \size(X)$.
\end{Cor}

\section{Exodromy theorems}\label{Sec exodromy}
In this section, we prove the three main exodromy theorems. We start in \hyperref[Sec proetale exodromy]{\S3.1} with the pro-\'etale exodromy theorem, using the classification of w-strictly local schemes of \autoref{Cor loccoh} and \autoref{Cor loccoh size}. In \hyperref[Sec etale exodromy]{\S3.2}, we use this to extract an \'etale exodromy theorem, by identifying the essential image of the pullback $\Sh(X_\et) \to \Sh(X_\proet)$. This is used in \hyperref[Sec constructible exodromy]{\S3.3} to give an alternative proof of the constructible \'etale exodromy theorem of \cite[Thm.~12.1.6]{BGH}.

As explained in \hyperref[Sec lower vs higher]{Lower versus higher categories}, we switch to an $\infty$-categorical perspective here, in order to obtain exodromy theorems allowing very general coefficients. For instance, this proves an exodromy theorem for the $\infty$-category $D(X_\proet,\mathbf Z)$ without any additional effort; see \autoref{Ex D(Z)}.

\subsection{The \texorpdfstring{pro-\'etale}{pro-étale} exodromy correspondence}\label{Sec proetale exodromy}
Using the computation of $\AffProEt_{/X}^{\kappa,\wsl}$ of \autoref{Cor loccoh size}, we give a direct proof of the pro-\'etale exodromy theorem; see \autoref{Thm proetale exodromy}. Recall the following definition:

\begin{Def}
Let $\kappa$ be an uncountable cardinal, and let $\mathscr C, \mathscr D$ be $\kappa$-condensed $\infty$-categories. Then the $\infty$-category of \emph{continuous functors} from $\mathscr C$ to $\mathscr D$ is the end
\[
\Fun\cts(\mathscr C,\mathscr D) = \int_{S \in \ProFin_\kappa\op} \Fun\!\big(\mathscr C(S),\mathscr D(S)\big).
\]
Concretely, this is defined as a limit over the twisted arrow category \cite[Def.~2.6]{GepnerHaugsengNikolaus}. The underlying $\infty$-groupoid of $\Fun\cts(\mathscr C,\mathscr D)$ is the space of natural transformations from $\mathscr C$ to $\mathscr D$, when viewed as functors $\ProFin_\kappa\op \to \Cat_\infty$. The formula above upgrades this to an $\infty$-category of natural transformations.
\end{Def}

\begin{Par}\label{Par lax colimit}
If $\mathscr C \colon \ProFin_\kappa\op \to \Cat_\infty$ is a $\kappa$-condensed $\infty$-category, then \cite[Thm.~1.1, Def.~2.6, and Def.~2.9]{GepnerHaugsengNikolaus} shows that the cartesian fibration classifying $\mathscr C$ can be computed as the coend of $\infty$-categories
\[
\int^{S \in \ProFin_\kappa\op} \mathscr C(S) \times (\ProFin_\kappa)_{/S}.
\]
When applied to the functor $\GAL_\kappa(X)\op \simeq \PTS_*^{\loccoh}(X_\et) \colon \ProFin_\kappa\op \to \Cat_\infty$ of \autoref{Def Galois category}, the equivalence of \autoref{Cor loccoh size} therefore yields an equivalence
\[
\int^S \GAL_\kappa(X)\op(S) \times (\ProFin_\kappa)_{/S} \to \AffProEt_{/X}^{\kappa,\wsl}.
\]
Explicitly, this functor sends a pair $(s_* \colon \Sh(S) \to \Sh(X_\et),f \colon T \to S)$ to the Henselisation $X_{(s \circ f)}$, which by \autoref{Lem Henselisation base change} agrees with $X_{(s)} \times_S T$.
\end{Par}

\begin{Par}
If $\mathscr E$ is an $\infty$-category with small limits, then \autoref{Cor hypersheaf of categories} defines a (large) condensed $\infty$-category $\COND_\kappa(\mathscr E)$ whose value on $S \in \ProFin_\kappa$ is $\Sh((\ProFin_\kappa)_{/S},\mathscr E)$. If $\mathscr E = \mathcal S$, then we may also describe $\COND_\kappa(\mathcal S)(S)$ as $\Cond_\kappa(\mathcal S)_{/S}$.
\end{Par}

\begin{Thm}\label{Thm proetale exodromy}
Let $X$ be a scheme, let $\kappa > \size(X)$ be a cardinal, and let $\mathscr E$ be an $\infty$-category with small limits. Then there is a canonical equivalence of $\infty$-categories
\[
\mathcal{Ex} \colon \Sh^\hyp(X_{\kappa\text{-}\proet},\mathscr E) \xrightarrow{\sim} \Fun\cts\! \big(\GAL_\kappa(X),\COND_\kappa(\mathscr E)\big).
\]
\end{Thm}

\begin{proof}
Let $X_{\kappa\text{-}\wl}$ be the w-local site of \autoref{Def w-local topology}. By \autoref{Prop enough wsl kappa}, it forms a basis for $X_{\kappa\text{-}\proet}$, and the topology on $X_{\kappa\text{-}\wl}$ agrees with the restriction of the $\kappa$-pro-\'etale topology by \autoref{Lem w-local topology}. Thus, the restriction functor
\[
\Sh^\hyp(X_{\kappa\text{-}\proet},\mathscr E) \to \Sh^\hyp(X_{\kappa\text{-}\wl},\mathscr E)
\]
is an equivalence by \autoref{Rmk hypersheaf basis}. By \ref{Par lax colimit}, we have
\[
\PSh(\AffProEt_{/X}^{\kappa,\wsl},\mathscr E) \simeq \PSh\!\Big(\int^S \GAL_\kappa(X)\op(S) \times (\ProFin_\kappa)_{/S},\mathscr E\Big).
\]
Since a coend is a type of colimit, the right hand side simplifies to
\[
\int_S \PSh\!\big(\GAL_\kappa(X)\op(S) \times (\ProFin_\kappa)_{/S},\mathscr E\big) \simeq \int_S \Fun\!\big(\GAL_\kappa(X)(S),\PSh((\ProFin_\kappa)_{/S},\mathscr E)\big),
\]
thus giving an equivalence
\[
\Psi \colon \PSh(\AffProEt_{/X}^{\kappa,\wsl},\mathscr E) \simeq \int_S \Fun\!\big(\GAL_\kappa(X)(S),\PSh((\ProFin_\kappa)_{/S},\mathscr E)\big).
\]
To finish the proof, we need to show that $\mathscr F \in \PSh(\AffProEt_{/X}^{\kappa,\wsl},\mathscr E)$ is a hypersheaf for the w-local topology if and only if $\Psi(\mathscr F)(S) \colon \GAL_\kappa(X)(S) \to \PSh((\ProFin_\kappa)_{/S},\mathscr E)$ lands in hypersheaves for all $S \in \ProFin_\kappa$. If $s_* \colon \Sh(S) \to \Sh(X_\et)$ is a locally coherent point, then $\Psi(\mathscr F)(S)(s_*) \colon (\ProFin_\kappa)_{/S}\op \to \mathscr E$ takes $(T \to S) \in \ProFin_{/S}$ to $\mathscr F(X_{(s)} \times_S T)$ (see \ref{Par lax colimit}). Thus, \autoref{Lem hypercover} shows that $\mathscr F$ is a hypersheaf for the w-local topology if and only if $\Psi(\mathscr F)(S)$ is a hypersheaf for every $S \in \ProFin_\kappa$.
\end{proof}

\begin{Rmk}\label{Rmk strong limit}
If $\kappa$ is a strong limit cardinal, the proof above can be simplified further. Indeed, since a w-strictly local scheme $W$ is w-projective (see \autoref{Rmk w-projective}) if and only if $\pi_0(W)$ is extremally disconnected, the equivalence of \autoref{Cor loccoh size} restricts to an equivalence
\[
\AffProEt_{/X}^{\kappa,\operatorname{wp}} \xrightarrow\sim \ExtrDisc_{/X}^{\kappa,\loccoh}.
\]
If $\kappa > \size(X)$ is a strong limit cardinal, then $\ProFin_\kappa$ is generated by extremally disconnected spaces and $\AffProEt_{/X}$ is generated by w-projective schemes, and we may think of a $\kappa$-condensed $\infty$-category as a finite product preserving functor $\ExtrDisc_\kappa\op \to \Cat_\infty$. The restriction of the pro-\'etale topology to $\AffProEt_{/X}^{\kappa,\operatorname{wp}}$ is given by the extensive topology, whose covers are given by finite disjoint unions (this is a special case of \autoref{Lem w-local topology}, but is also immediate from the definition of w-projective scheme). Thus, a presheaf $\mathscr F \colon \big(\AffProEt_{/X}^{\kappa,\operatorname{wp}}\big)\op \to \mathscr E$ is a sheaf if and only if it preserves finite products, and the corresponding condition on $\Psi(\mathscr F)$ is that $\Psi(\mathscr F)(S) \colon \GAL_\kappa(X)(S) \to \PSh((\ExtrDisc_\kappa)_{/S},\mathscr E)$ lands in finite product preserving functors for all $S \in \ExtrDisc_\kappa$. This proves the result since the topology on $(\ExtrDisc_\kappa)_{/S}$ is also the extensive topology. Thus, the use of the w-local topology is replaced by considering functors that preserve finite products.
\end{Rmk}

\begin{Rmk}\label{Rmk description Ex}
The proof above yields a concrete description of the functor $\mathcal{Ex}$: for an $\mathscr E$-valued $\kappa$-pro-\'etale sheaf $\mathscr F$ and any $S \in \ProFin_\kappa$ the associated functor
\[
\mathcal{Ex}(\mathscr F)(S) \colon \GAL_\kappa(X)(S) \to \COND_\kappa(\mathscr E)(S) = \Sh^{\hyp}((\ProFin_\kappa)_{/S},\mathscr E)
\]
is the transpose of the functor
\[
\GAL_\kappa(X)(S) \times ((\ProFin_\kappa)_{/S})\op \to (X^{\wsl})\op \xrightarrow{\mathscr F} \mathscr E.
\]
Explicitly a geometric morphism $s_* \colon \Sh(S) \to X$ is sent to the sheaf on $(\ProFin_\kappa)_{/S}$ given by the formula
\[
\label{equ:Description_of_Ex_Functor}
(f \colon T \to S) \mapsto \mathscr F(X_{(s \circ f)}) \simeq \mathscr F(X_{(s)} \times_S T).
\]
\end{Rmk}

\begin{Ex}\label{Ex D(Z)}
Taking $\mathscr E$ to be the derived $\infty$-category $D(\mathbf Z)$, we get an equivalence
\[
D(X_{\kappa\text{-}\proet},\mathbf Z) \simeq \Fun\cts\!\big(\GAL_\kappa(X),D(\COND_\kappa(\Ab))\big).
\]
Indeed, if $\mathscr C$ is a $1$-site, then $\Sh^\hyp(\mathscr C,D(\mathbf Z))$ is canonically identified with the derived category $D(\mathscr C,\mathbf Z)$ of abelian sheaves on $\mathscr C$ \cite[Thm.~2.1.2.2]{LurieSAG}. Applying this to $X_{\kappa\text{-}\proet}$ and $(\ProFin_\kappa)_{/S}$ for all $S \in \ProFin_\kappa$, the result follows from \autoref{Thm proetale exodromy}.
\end{Ex}

\subsection{The \texorpdfstring{\'etale}{étale} exodromy theorem}\label{Sec etale exodromy}
In this section, we show that the pro-\'etale exodromy correspondence of \autoref{Thm proetale exodromy} restricts to a statement for \'etale sheaves. The key ingredient is the following recognition principle for \'etale sheaves inside pro-\'etale sheaves (see also \cite[Lem.~5.1.2]{BhattScholze} and \cite[Prop.~2.38]{HHLMMW}).

\begin{Prop}\label{Prop etale to proetale}
Let $X$ be a scheme, let $\kappa > \size(X)$ be a cardinal, and let $n \in \mathbf N$. Then the canonical functor
\[
\nu^* \colon \Sh_{\leq n}(X_\et,\mathcal S) \to \Sh_{\leq n}(X_{\kappa\text{-}\proet},\mathcal S)
\]
is fully faithful. For $\mathscr F \in \Sh_{\leq n}(X_{\kappa\text{-}\proet},\mathcal S)$, the following are equivalent:
\begin{enumerate}
\item $\mathscr F$ is in the essential image of $\nu^*$.
\item For any cofiltered diagram $Y \colon \mathcal I \to \AffEt_{/X}$ whose limit is $\kappa$-small, the canonical map
\[
\colim_{\substack{\longrightarrow \\ i \in \mathcal I}} \mathscr F(Y_i) \to \mathscr F\Big( \lim_{\substack{\longleftarrow \\ i \in \mathcal I}} Y_i \Big)
\]
is an equivalence.
\item For any cofiltered diagram $Y \colon \mathcal I \to \AffProEt^\kappa_{/X}$ whose limit is $\kappa$-small, the canonical map
\[
\colim_{\substack{\longrightarrow \\ i \in \mathcal I}} \mathscr F(Y_i) \to \mathscr F\Big( \lim_{\substack{\longleftarrow \\ i \in \mathcal I}} Y_i \Big)
\]
is an equivalence.
\end{enumerate}
\end{Prop}

\begin{proof}
The fully faithfulness statement is given in \cite[Prop.~2.38]{HHLMMW} (see also \cite[Prop.~7.3.1.6]{WolfThesis} and \cite[Prop.~A.5.33]{MairThesis}), and the proof shows that $\nu^*$ is given by left Kan extension along $\AffEt_{/X} \to \AffProEt^\kappa_{/X}$ (without the need to sheafify). This proves the equivalence (a) $\Leftrightarrow$ (b). The implication (b) $\Rightarrow$ (c) holds since $\AffProEt_{/X} = \Pro(\AffEt_{/X})$ by \autoref{Lem affproet} and colimits commute with colimits. The converse is clear.
\end{proof}

\begin{Lemma}\label{Lem diagram of wsl cover}
Let $X$ be a scheme, let $\kappa > \size(X)$ be a cardinal, let $\mathcal I$ be a small category with a final object, and let $Y \colon \mathcal I \to \AffProEt^\kappa_{/X}$ be a diagram. Then there exists a diagram $W \colon \mathcal I \to \AffProEt^\kappa_{/X}$ of w-strictly local schemes with w-local maps and a map $W \to Y$ of diagrams that is levelwise surjective. Moreover, if $\mathcal I$ is cofiltered and $\lim_{i \in \mathcal I} Y_i$ is $\kappa$-small, then we may choose $W$ such that $\lim_{i \in \mathcal I} W_i$ is $\kappa$-small.
\end{Lemma}

\begin{proof}
Let $0 \in \mathcal I$ be a final object. All $Y_i$ have a canonical map to $Y_0$, and \autoref{Lem affproet over} gives $(\AffProEt^\kappa_{/X})_{/Y_0} \simeq \AffProEt^\kappa_{/Y_0}$, so we may assume without loss of generality that $X$ is affine and $Y_0 = X$. By \autoref{Prop enough wsl kappa}, we may choose $V \in \AffProEt_{/X}^{\kappa,\wsl}$ and a map $V \to X$ such that every $x \in X$ is the image of some closed point $v \in V$. Set $S = \pi_0(V)$, which defines a geometric morphism $s_* \colon \Sh(S) \simeq \Sh(V^\cl_\et) \to \Sh(X_\et)$ as well as a geometric morphism $s_* \colon \Cond_{/S} \simeq \Sh(V^\cl_\proet) \to \Sh(X_\proet)$. In particular, we get a pullback
\[
s^* \colon \AffProEt_{/X} \to \AffProEt_{/V^\cl} \simeq \ProFin_{/S}.
\]
This has a left adjoint $s_!$ taking $(f \colon T \to S)$ to the Henselisation $X_{(s \circ f)}$, which agrees with $V \times_S T$ by \autoref{Lem Henselisation base change}. By \autoref{Lem acyc profinite}, we see that $s_!$ sends arbitrary maps of profinite sets to w-local maps of strictly w-local schemes. Let $W \to Y$ be given by the levelwise counit $s_!s^*Y_i \to Y_i$. Since $V^\cl \to X$ is surjective, the same goes for $V^\cl \times_X Y_i \to Y_i$. But $V^\cl \times_X Y_i$ represents the sheaf $s^*Y_i$, so we see that the Henselisation $s_!s^*Y_i \to Y_i$ is surjective (in fact, the map $(s_!s^*Y_i)^\cl = s^*Y_i \to Y_i$ is already surjective). By \autoref{Rmk colimit kappa small}, the fibre product $V^\cl \times_X Y_i$ is a $\kappa$-small scheme, so $\pi_0(V^\cl \times_X Y_i) = s^*Y_i$ is in $\ProFin_\kappa$ by \autoref{Lem pi0 kappa}. Thus, the Henselisation $W_i$ is $\kappa$-small by \autoref{Lem size Hens}. This proves the first statement. For the second, note that $s^*$ preserves cofiltered limits, and $s_!$ does as well since it is given by $T \mapsto V \times_S T$. Thus, $\lim_{i \in \mathcal I} W_i \simeq s_!s^*(\lim_{i \in \mathcal I} Y_i)$ is $\kappa$-small if $\lim_{i \in \mathcal I} Y_i$ is, by the same argument.
\end{proof}

\begin{Cor}\label{Cor diagram of wsl hypercover}
Let $X$ be a scheme, let $\kappa > \size(X)$ be a cardinal, let $\mathcal I$ be a small category with a final object, and let $Y \colon \mathcal I \to \AffProEt^\kappa_{/X}$ be a diagram. Then there exists a simplicial diagram $W \colon \Delta\op \to \Fun(\mathcal I,\AffProEt^\kappa_{/X})$ with an augmentation $W \to Y$ such that the following hold:
\begin{enumerate}
\item For every $n \in \Delta\op$, the diagram $W_{n,\bullet} \in \Fun(\mathcal I,\AffProEt^\kappa_{/X})$ is a diagram of w-strictly local schemes with w-local maps;
\item For every $i \in \mathcal I$, the augmented simplicial scheme $W_{\bullet,i} \to Y_i$ is a pro-\'etale hypercovering.
\end{enumerate}
If $\mathcal I$ is cofiltered and $\lim_{i \in \mathcal I} Y_i$ is $\kappa$-small, we may choose $W$ such that each $\lim_{i \in \mathcal I} W_{n,i}$ is $\kappa$-small.
\end{Cor}

\begin{proof}
As in the proof of \autoref{Lem diagram of wsl cover}, replace $X$ by $Y_0$ to assume $X$ is affine. In particular, $\AffProEt_{/X}$ has finite limits in this case. The result follows from \autoref{Lem diagram of wsl cover} and the standard hypercover arguments (see e.g.\ \cite[Tag \href{https://stacks.math.columbia.edu/tag/094J}{094J}]{Stacks} or \cite[Exp.~V$^{\text{bis}}$, \S5.1]{SGA4II}), since a finite coproduct of w-strictly local schemes is w-strictly local.
\end{proof}

\begin{Def}\label{Def underline E}
Let $\mathscr E$ be a compactly assembled presentable $\infty$-category \cite[\S21.1.2]{LurieSAG}. For any uncountable cardinal $\kappa$, define the $\kappa$-condensed $\infty$-category $\underline{\mathscr E}$ via the assertion
\begin{align*}
\underline{\mathscr E} \colon \ProFin_\kappa\op &\to \Cat_{\infty} \\
S &\mapsto \Sh(S,\mathscr E).
\end{align*}
These are indeed $\kappa$-condensed $\infty$-categories by \cite[Cor.~2.8(2) and Ex.~1.28]{HaineDescent} (for $\mathscr E = \Set$, see also \autoref{Lem Sh(S) sheaf}). For simplicity, we denote $\underline{\mathcal S_{\leq n}}$ by $\underline{\mathcal S}_{\leq n}$, which takes $S \in \ProFin_\kappa$ to the $(n+1)$-category $\Sh(S,\mathcal S_{\leq n}) \simeq \Sh_{\leq n}(S,\mathcal S)$.
\end{Def}

\begin{Par}
For $S \in \ProFin_\kappa$, there is a fully faithful functor $\Sh(S,\mathcal S) \to \Cond_\kappa(\mathcal{S})_{/S}$. This functor is natural in $S$, giving a functor $\underline{\mathcal S} \to \COND_\kappa(\mathcal S)$ of condensed categories. It follows that for any $\kappa$-condensed $\infty$-category~$\mathscr C$, the $\infty$-category $\Fun\cts(\mathscr C,\underline{\mathcal S})$ is canonically a full subcategory of $\Fun\cts(\mathscr C,\COND_\kappa(\mathcal S))$.
\end{Par}

\begin{Thm}\label{Thm etale exodromy}
Let $X$ be a scheme, and let $\kappa > \size(X)$ be a cardinal. For any $n \in \mathbf N$, the equivalence $\mathcal{Ex}$ restricts to an equivalence
\[
\mathcal{Ex}^{\et}_{\leq n} \colon \Sh_{\leq n}(X_\et,\mathcal S) \xrightarrow \sim \Fun\cts(\GAL_\kappa(X),\underline{\mathcal{S}}_{\leq n}).
\]
If $\mathscr E$ is a compactly assembled presentable $\infty$-category, this induces an equivalence
\[
\mathcal{Ex}^{\et} \colon \Sh^{\post}(X_\et,\mathscr E) \xrightarrow \sim \Fun\cts(\GAL_\kappa(X),\underline{\mathcal{E}}).
\]
\normalfont Here, $\Sh^{\post}(X_\et)$ denotes the \emph{Postnikov completion} of the $\infty$-topos $\Sh(X_\et,\mathcal S)$, as defined in \cite[Def.~A.7.2.5]{LurieSAG}, and $\Sh^{\post}(X_\et,\mathscr E) \simeq \Sh^{\post}(X_\et) \otimes \mathscr E$; see \autoref{Def sheaf hypersheaf} and \autoref{Rmk presentable tensor}. 
\end{Thm}

\begin{proof}
If $\mathscr F \in \Sh_{\leq n}(X_\et,\mathcal S)$ and $s_* \colon \Sh(S) \to \Sh(X_\et)$ is a locally coherent geometric morphism, then $s^*_\proet \mathscr F = s^*_\et \mathscr F$ lands in $\Sh_{\leq n}(S,\mathcal S) \subseteq \Cond_\kappa(\mathcal S)_{/S}$. By \autoref{Rmk description Ex}, this means that $\mathcal{Ex}(\mathscr F)$ is in the full subcategory $\Fun\cts(\GAL_\kappa(X),\underline{\mathcal S}_{\leq n})$.

We first note that the converse is true if $X$ is strictly profinite: if $S = \lvert X \rvert$, then $\Sh(X_{\kappa\text{-}\proet},\mathcal S)$ is canonically identified with $\Sh((\ProFin_\kappa)_{/S},\mathcal S) \simeq \Cond_\kappa(\mathcal S)_{/S}$, and $\Sh(X_\et,\mathcal S)$ with $\Sh(S,\mathcal S)$, and likewise for $\Sh_{\leq n}$. Applying $\mathcal{Ex}(\mathscr F)(S) \colon \GAL_\kappa(X)(S) \to \Cond_\kappa(\mathcal S)_{/S}$ to this equivalence $s_* \colon \Sh(S) \xrightarrow \sim \Sh(X_\et)$ shows that $\mathcal{Ex}(\mathscr F) \in \Fun\cts(\GAL_\kappa(X),\underline{\mathcal S}_{\leq n})$ implies $\mathscr F \in \Sh_{\leq n}(X_\et,\mathcal S)$.

Now let $X$ be an arbitrary scheme, and let $\mathscr F \in \Sh(X_{\kappa\text{-}\proet},\mathcal S)$ be a sheaf such that $\mathcal{Ex}(\mathscr F)$ lies in $\Fun\cts(\GAL_\kappa(X),\underline{\mathcal S}_{\leq n})$. Note that this implies that $\mathscr F$ is $n$-truncated. We will check condition (c) of \autoref{Prop etale to proetale} to prove that $\mathscr F$ is in the essential image of $\Sh_{\leq n}(X_\et,\mathcal S) \to \Sh_{\leq n}(X_{\kappa\text{-}\proet},\mathcal S)$. Let $Y \colon \mathcal I \to \AffProEt^\kappa_{/X}$ be a cofiltered diagram whose limit is $\kappa$-small. For any $i_0 \in \mathcal I$, the forgetful functor $\mathcal I_{/i_0} \to \mathcal I$ is coinitial, so we may assume $\mathcal I$ has a terminal object $0$. Choose a hypercover $W \to Y$ as in \autoref{Cor diagram of wsl hypercover}, and write $W_\infty \to Y_\infty$ for the limit over $i \in \mathcal I$. Since truncated sheaves are hypercomplete \cite[Lem.~6.5.2.9]{LurieHTT}, the diagrams
\[
\mathscr F(Y_i) \to \mathscr F(W_{0,i})  \mathrel{\substack{\rightarrow \\[-0.6ex] \rightarrow}} \mathscr F(W_{1,i}) \mathrel{\substack{\rightarrow\\[-0.6ex] \rightarrow \\[-0.6ex] \rightarrow}} \cdots
\]
for $i \in \mathcal I \cup \{\infty\}$ are limit diagrams. Moreover, since $\mathscr F$ is $n$-truncated, we may truncate this diagram after the $(n+1)$-st step without affecting the limit; see e.g.\ \cite[Prop.~A.1]{HesselholtPstragowski}. In particular, they are finite limits, so they commutes with filtered colimits \cite[Ex.~7.3.4.7]{LurieHTT}. Thus, in the natural diagram
\begin{equation}
\begin{tikzcd}
\colim\limits_{\substack{\longrightarrow \\ i \in \mathcal I}} \mathscr F(Y_i) \ar{r}\ar[start anchor={[yshift=8pt]}]{d} & \colim\limits_{\substack{\longrightarrow \\ i \in \mathcal I}} \mathscr F(W_{0,i}) \ar[shift left=1]{r}\ar[shift right=1]{r}\ar[start anchor={[yshift=8pt]}]{d} & \colim\limits_{\substack{\longrightarrow \\ i \in \mathcal I}} \mathscr F(W_{1,i}) \ar[shift left=2]{r}\ar{r}\ar[shift right=2]{r}\ar[start anchor={[yshift=8pt]}]{d} & \cdots\\
\makebox*{$\colim \mathscr F(Y_i)$}{$\mathscr F(Y_\infty)$} \ar{r} & \makebox*{$\colim \mathscr F(W_{0,i})$}{$\mathscr F(W_{0,\infty})$} \ar[shift left=1]{r}\ar[shift right=1]{r} & \makebox*{$\colim \mathscr F(W_{1,i})$}{$\mathscr F(W_{1,\infty})$} \ar[shift left=2]{r}\ar{r}\ar[shift right=2]{r} & \cdots\punct{,}
\end{tikzcd}\label{Dia hypercover}
\end{equation}
both rows are limit diagrams. On the other hand, if $n \in \mathbf N$, then the schemes $W_{n,i}$ for $i \in \mathcal I \cup \{\infty\}$ are w-strictly local and the transition maps are w-local. Thus, the values $\mathscr F(W_{n,i})$ are given by the values on the strictly profinite schemes $W_{n,i}^\cl$, and we have $W_{n,\infty}^\cl = \lim_{i \in \mathcal I} W_{n,i}^\cl$. By the strictly profinite case treated above, the pullback of $\mathscr F$ to each $W_{n,i}^\cl$ is an \'etale sheaf, so \autoref{Prop etale to proetale} shows that the maps
\[
\colim_{\substack{\longrightarrow \\ i \in \mathcal I}} \mathscr F(W_{n,i}) \to \mathscr F(W_{n,\infty})
\]
are equivalences for all $n \in \mathbf N$. Thus, all vertical maps in \eqref{Dia hypercover} except the left one are equivalences, hence so is the one on the left since both rows are limit diagrams. This proves the first statement. When $\mathscr E = \mathcal S$, the left hand side in the second statement is by definition obtained by taking the limit of the left hand side of the first statement \cite[Def.~A.7.2.5]{LurieSAG}. The same goes for the right hand side since $\Sh(S,\mathcal S)$ is Postnikov complete for any $S \in \ProFin$ by \cite[Cor.~7.2.1.10 and Thm.~7.2.3.6]{LurieHTT}. This proves the second statement when $\mathscr E = \mathcal S$.

If $\mathscr E$ is arbitrary, we have $\Sh^{\post}(X_\et,\mathscr E) \simeq \Sh^{\post}(X_\et,\mathcal S)\otimes \mathscr E$ and $\Sh(S,\mathscr E) \simeq \Sh(S,\mathcal S) \otimes \mathscr E$. Thus, we get $\Fun(\GAL_\kappa(X)(S),\underline{\mathscr E}(S)) \simeq \Fun(\GAL_\kappa(X)(S),\underline{\mathcal S}(S)) \otimes \mathscr E$ for any $S \in \ProFin_\kappa$. The result follows from the $\mathcal S$-valued one since $- \otimes \mathscr E \colon \LTop \to \LTop$ preserves limits when $\mathscr E$ is compactly assembled \cite[Lem.~2.16]{HaineBaseChange}.
\end{proof}

\begin{Rmk}\label{Rmk difference with Lurie}
In the case where $\mathscr{E} = \Set$ and $X$ is qcqs, \autoref{Thm etale exodromy} is closely related to Lurie's strong conceptual completeness theorem \cite[Thm.~0.0.6]{LurieUltra} for the \'etale topos of $X$. However, there is a subtle difference, since Lurie considers the condensed category of all points $\PTS^*(X_\et)$ \cite[Def.~6.3.8]{LurieUltra}, while we only take (locally) coherent points into account (see \autoref{Def Galois category}).

Up to translation from ultracategories into condensed categories, Lurie's result \cite[Thm.~0.0.6]{LurieUltra} says that $\Sh(X_\et,\Set)$ can also be recovered as $\Fun\cts(\PTS^*_\kappa(\Sh(X_\et)),\underline{\Set})$. In other words, for the purpose of \autoref{Thm etale exodromy}, it is irrelevant whether one works with all points or only locally coherent points. However, the analogous claim fails for the pro-\'etale exodromy correspondence. More precisely, the functor
\begin{equation}
\Fun\cts\! \big(\GAL_\kappa(X),\COND_\kappa(\mathcal S)\big) \to  \Fun\cts\! \big(\PTS_\kappa^*(X_\et),\COND_\kappa(\mathcal S)\big) \label{eq Gal to Pts}
\end{equation}
induced by restricting along the inclusion $\GAL_\kappa(X) \to \PTS_\kappa^*(X_\et)$ is \emph{not} an equivalence in general. Indeed, one can recover an idempotent complete $\kappa$-condensed $\infty$-category $\mathscr C$ from the internal functor category $\mathcal Fun(\mathscr C,\COND_\kappa(\mathcal S))$, see \cite[Prop.~2.1.5.10]{MartiniWolf}. But the condensed $\infty$-category $\mathcal Fun(\mathscr C,\COND_\kappa(\mathcal S))$ is given by $S \mapsto \Fun\cts(\mathscr C \times h_S,\COND_\kappa(\mathcal S))$, and \autoref{Prop equivalence S} implies that $\GAL_\kappa(X) \times h_S \simeq \GAL_\kappa(X \times S)$, and likewise for $\PTS_\kappa^*$. The existence of geometric morphisms that are not locally coherent (see \autoref{Ex acyclic not wsl}) implies that the inclusion $\GAL_\kappa \hookrightarrow \PTS_\kappa^*$ is not an equality, so \eqref{eq Gal to Pts} cannot be an equivalence in general.

However, Haine has recently shown that $\GAL_\kappa(X)$ and $\PTS_\kappa^*(X_\et)$ have the same condensed classifying spaces \cite[Thm.~2.20]{HaineClassifying}, so the difference disappears again if one only cares about \emph{locally constant} pro-\'etale sheaves.

The place in the proof of \autoref{Thm proetale exodromy} where we crucially used that we are working with $\GAL_\kappa(X)$ and not $\PTS^*_\kappa(\Sh(X_\et))$ is the use of \autoref{Lem w-local topology} (and its corollary \autoref{Lem hypercover}). While $\AffProEt_{/X}^{\acyc}$ still forms a basis for the pro-\'etale topology, the topology cannot be described simply in terms of covers of the form $W \times_S T \to W$ for $W \in \AffProEt_{/X}^{\acyc}$ and $T \twoheadrightarrow S = \pi_0(W)$ a continuous surjection of profinite sets.
\end{Rmk}

\begin{Ex}\label{Ex D(Z) left completed}
The $\infty$-category $D(\mathbf Z)$ is compactly generated, so in particular compactly assembled \cite[Ex.~21.1.2.3]{LurieSAG}. Thus, \autoref{Thm etale exodromy} gives an equivalence
\[
\widehat D(X_\et,\mathbf Z) \simeq \Fun\cts\!\big(\GAL_\kappa(X),\underline{D(\mathbf Z)}\,\big),
\]
where $\widehat D(X_\et,\mathbf Z)$ denotes the left completion of $D(X_\et,\mathbf Z)$ \cite[Prop.~1.2.1.17]{LurieHA}. This is closely related to the identification of \cite[\S5.3]{BhattScholze} of $\widehat D(X_\et,\mathbf Z)$ inside $D(X_\proet,\mathbf Z)$.
\end{Ex}

\subsection{Exodromy for constructible \texorpdfstring{\'etale}{étale} sheaves}\label{Sec constructible exodromy}
Finally, we explain how to obtain the exodromy correspondence for constructible sheaves, as in \cite[Thm.~12.1.6]{BGH}, as a consequence of \autoref{Thm etale exodromy}.

\begin{Par}
Recall that an object $X$ of a locally coherent $\infty$-topos $\mathscr X$ is \emph{constructible} if $X \times U$ is a truncated and coherent object of $\mathscr X_{/U}$ for every coherent object $U \in \mathscr X$ \cite[Def.~2.3.4]{LurieDAGXIII}. We write $\mathscr X\cons \subseteq \mathscr X$ for the full $\infty$-subcategory on constructible objects.

When $\mathscr X = \Sh(X_\et,\mathcal S)$ for a scheme $X$, then a sheaf $\mathscr F$ is constructible if and only if the following condition holds: for every affine open $U \subseteq X$, there exists a finite constructible stratification $U = \bigcup_{i=1}^n U_i$ and finite \'etale covers $V_i \to U_i$ such that each $\mathscr F|_{V_i}$ is equivalent to the constant sheaf $\underline{Y_i}$ on some $\pi$-finite space $Y_i$ \cite[Thm.~2.3.24 and Prop.~2.3.5]{LurieDAGXIII} (see also \cite[Prop.~9.5.4]{BGH}). We will write $\Sh\cons(X_\et,\mathcal S) \subseteq \Sh(X_\et,\mathcal S)$ for the full \makebox{$\infty$-subcategory} on constructible sheaves.
\end{Par}

\begin{Lemma}
Let $\mathscr X$ be a locally coherent $\infty$-topos. Then the functor $e^* \colon \mathscr X \to \mathscr X^{\post}$ restricts to a fully faithful functor $\mathscr X\cons \to \mathscr X^{\post}$.
\end{Lemma}

\begin{proof}
By \cite[Rmk.~A.7.3.6]{LurieSAG}, the unit $X \to e_*e^*X$ is given by the map $X \to \lim_n \tau_{\leq n} X$.
We need to show that this is an equivalence when $X$ is constructible. Let $U = \coprod_{i \in I} U_i \to \mathbf 1$ be an effective epimorphism where each $U_i$ is coherent. For $i \in I$, denote by $f_i \colon \mathscr X_{/U_i} \to \mathscr X$ the natural geometric morphism, and define $f \colon \mathscr X_{/U} \to \mathscr X$ likewise. If $X$ is constructible, then $f_i^*X$ is truncated for all $i \in I$, so $f_i^*X \to \lim_n \tau_{\leq n} f_i^*X$ is an equivalence. Then the same goes for $f^*X \to \lim_n \tau_{\leq n} f^*X$. By \cite[Prop.~6.3.1.9]{LurieHTT} and since $f$ is \'etale, the right hand side is $f^*\lim_n \tau_{\leq n}X$. Then $X \to \lim_n \tau_{\leq n} X$ is an equivalence by \cite[Lem.~6.2.3.16]{LurieHTT}.
\end{proof}

The main tool needed to extract the constructible exodromy theorem from \autoref{Thm etale exodromy} is a descent principle for coherent objects:

\begin{Lemma}\label{Lem reflect coherence}
Let $f \colon \mathscr X \to \mathscr Y$ be coherent a geometric morphism between coherent $\infty$-topoi, and assume that $f^*$ is conservative. If $n \in \mathbf N$ and $Y \in \mathscr Y$, then $Y$ is $n$-coherent if and only if $f^*Y$ is $n$-coherent.
\end{Lemma}

\begin{proof}
If $Y$ is $n$-coherent, then $f^*Y$ is $n$-coherent by \cite[Cor.~3.4.6]{BGH}. For the converse, first note that since $f^*$ is conservative, the same goes for the pullback $\mathscr Y_{/U} \to \mathscr X_{/f^*U}$ for any $U \in \mathscr Y$. Since $f^*$ commutes with $\tau_{\leq -1}$ \cite[Prop.~6.3.1.9]{LurieHTT}, we see that $f^*$ reflects effective epimorphisms \cite[Cor.~6.2.3.5]{LurieHTT}. Note that $f^*$ always preserves effective epimorphisms \cite[Rmk.~6.2.3.6]{LurieHTT}.

We prove the statement by induction on $n$. For $n = 0$, let $\coprod_{i \in I} U_i \to Y$ be an effective epimorphism and assume that $f^*Y$ is quasi-compact. Then $\coprod_{i \in I} f^*U_i \simeq f^*(\coprod_{i \in I} U_i) \to f^*Y$ is an effective epimorphism, so there is a finite subset $J \subseteq I$ such that \makebox{$\coprod_{j \in J} f^*U_j \to f^*Y$} is an effective epimorphism. Then $\coprod_{j \in J} U_j \to Y$ is an effective epimorphism since $f^*$ reflects effective epimorphisms, so $Y$ is quasi-compact. If $n \geq 1$, then the induction hypothesis shows that $f^*$ reflects $(n-1)$-coherence of objects. The result for $n$-coherent objects now follows from \cite[Rmk.~A.2.0.15]{LurieSAG} since $f^*$ is left exact and preserves and reflects $(n-1)$-coherence.
\end{proof}

\begin{Def}
Define the $\kappa$-condensed category $\underline{\mathcal{S}}{}\cons$ of constructible spaces to be the full $\kappa$-condensed subcategory of $\underline{\mathcal S}$ that for every $S \in \ProFin_\kappa$ is given by the full subcategory
\[
\Sh\cons(S,\mathcal S) \subseteq \Sh(S,\mathcal S)
\]
of constructible sheaves on $S$. Note that a sheaf $\mathscr F$ is constructible if and only if there is a clopen decomposition $S = \coprod_{i=1}^n S_i$ and $\pi$-finite spaces $Y_i$ such that $\mathscr F|_{S_i} \simeq \underline{Y_i}$ for all $i$.

In particular, $\underline{\mathcal S}{}\cons$ is the discrete $\kappa$-condensed $\infty$-category on the $\infty$-category $\mathcal S_\pi$ of $\pi$-finite spaces. Thus, one may also denote it by $\mathcal S_\pi$ or $\underline{\mathcal S}{}_\pi$. We choose not to use $\underline{\mathcal S}{}_\pi$ because the similarly labelled $\kappa$-condensed $\infty$-category $\underline{\mathcal S}$ of \autoref{Def underline E} does \emph{not} arise as a discrete $\kappa$-condensed $\infty$-category.

However, the $\kappa$-condensed $\infty$-category $\underline{\mathcal S}$ of \autoref{Def underline E} is discrete as $\kappa$-condensed \emph{presentable} $\infty$-category. This fact, along with the awareness of a result like \autoref{Thm etale exodromy}, was pointed out by Clark Barwick before the work on the present paper started.
\end{Def}

\begin{Thm}\label{Thm constructible exodromy}
Let $X$ be a scheme. The \'etale exodromy correspondence restricts to an equivalence
\[
\mathcal{Ex}\cons^\et \colon \Sh\cons(X_\et,\mathcal S) \xrightarrow \sim \Fun\cts (\GAL(X),\underline{\mathcal S}{}\cons).
\]
\end{Thm}

\begin{proof}
If $S \in \ProFin$ and $s_* \colon \Sh(S) \to \Sh(X_\et)$ is a locally coherent geometric morphism, then we can view $\Sh(S)$ as the \'etale site on the strictly profinite scheme $X_{(s)}^\cl$. In particular, the geometric morphism $s$ extends to a geometric morphism $s_* \colon \Sh(S,\mathcal S) \to \Sh(X_\et,\mathcal S)$ on $\infty$-topoi. Thus, if $\mathscr F \in \Sh\cons(X_\et,\mathcal S)$, then $s^*\mathscr F$ is constructible by \cite[Prop.~2.3.6]{LurieDAGXIII}, showing that $\mathcal{Ex}^\et(\mathscr F)$ lands in $\Fun\cts(\GAL(X),\underline{\mathcal S}{}\cons)$.

Conversely, suppose $\mathscr F \in \Sh^{\post}(X_\et,\mathcal S)$ lands in $\Fun\cts(\GAL(X),\underline{\mathcal S}{}\cons)$ under $\mathcal{Ex}^\et$. To check that $\mathscr F$ is constructible, we may proceed locally, so we may assume $X$ is affine. By \autoref{Prop enough wsl kappa}, there exists $W \in \AffProEt_{/X}^\wsl$ such that $W^\cl \to X$ is surjective. Let $S = \lvert W^\cl \rvert$ and let $s_* \colon \Sh(S,\mathcal S) \to \Sh(X_\et,\mathcal S)$ be the induced geometric morphism of \makebox{$\infty$-topoi.} Then $s^*$ is conservative, hence reflects coherence of objects by \autoref{Lem reflect coherence}. It also reflects $n$-truncated objects for all $n \in \mathbf N$ since $s^*\tau_{\leq n} \simeq \tau_{\leq n}s^*$ \cite[Prop.~6.3.1.9]{LurieHTT} and since $s^*$ is conservative. Thus, $\mathscr F$ is truncated and coherent since $s^*\mathscr F$ is.
\end{proof}

In the setting of sheaves of sets, the descent statement of \autoref{Lem reflect coherence} may be replaced by \cite[Tag \href{https://stacks.math.columbia.edu/tag/095Q}{095Q}]{Stacks}.

\section{Formal properties of the Galois category}\label{Sec formal}
In this section, we prove two formal properties of the $\kappa$-condensed category $\GAL_\kappa(X)$. In \hyperref[Sec van Kampen]{\S4.1}, we prove a Van Kampen theorem, and in \hyperref[Sec change of cardinal]{\S4.2}, we prove an accessibility statement for the functor $\GAL(X) \colon \ProFin\op \to \Cat$.

\subsection{Van Kampen theorem}\label{Sec van Kampen}
In \autoref{Thm proetale exodromy}, we constructed an equivalence
\[
\Sh^\hyp(X_{\kappa\text{-}\proet},\mathcal S) \simeq \Fun\cts(\GAL_\kappa(X),\COND_\kappa(\mathcal S)).
\]
On the other hand, by continuous straightening/unstraightening (see \ref{Par internal straightening}), we get an equivalence
\[
\Fun\cts(\GAL_\kappa(X),\COND_\kappa(\mathcal S)) \simeq \LFib_{/\GAL_\kappa(X)}
\]
to the category of continuous left fibrations over $\GAL_\kappa(X)$. In this section, we explain how to think about the equivalence
\begin{equation}
\Sh^\hyp(X_{\kappa\text{-}\proet},\mathcal S) \simeq \LFib_{/\GAL_\kappa(X)}\label{Eq equivalence left fibration}
\end{equation}
directly. Since the forgetful functor $\LFib_{/\GAL_\kappa(X)} \to \Cond_\kappa(\Cat_\infty)$ preserves colimits, this implies a Van Kampen theorem for the functor $X \mapsto \GAL_\kappa(X)$; see \autoref{Thm van Kampen}.

\begin{Lemma}\label{Lem Gal left fibration}
Let $U \to X$ be a weakly \'etale morphism in $\Sch^\kappa$. Then the left fibration corresponding to the representable sheaf $h_U$ under the equivalence \eqref{Eq equivalence left fibration} is given by the natural functor $\GAL_\kappa(U) \to \GAL_\kappa(X)$.
\end{Lemma}

\begin{proof}
By \autoref{Rmk description Ex}, the continuous functor $\GAL_\kappa(X) \to \COND_\kappa(\mathcal S)$ corresponding to the representable sheaf $h_U$ defined by $U$ takes an $S$-points $s_* \colon \Sh(S) \to \Sh(X_\et)$ to the sheaf $T \mapsto \Hom_X(X_{(s)} \times_S T,U)$ on $(\ProFin_\kappa)_{/S}$. By the description of continuous unstraightening of \ref{Par Grothendieck construction}, this means that the $S$-points of the corresponding left fibration $\mathscr C \to \GAL_\kappa(X)$ are given by pairs of an $S$-point $s_* \colon \Sh(S) \to \Sh(X_\et)$ and an element of $\Hom_X(X_{(s)},U)$. By \autoref{Cor loccoh size}, the locally coherent $\kappa$-small $S$-points of $\Sh(X_\et)$ correspond to the fibre $\AffProEt_{/X}^{\kappa,\wsl} \times_{\ProFin} \{S\}$. Applying the same to $U$ instead of $X$ gives the result, since $(\AffProEt_{/X}^{\kappa,\wsl})_{/U} \simeq \AffProEt_{/U}^{\kappa,\wsl}$ (by \autoref{Cor wsl affproet} and \cite[Tag \href{https://stacks.math.columbia.edu/tag/0951}{0951}]{Stacks}). (Note that we cannot use \autoref{Lem affproet over} unless we assume $U \in \AffProEt_{/X}$.)
\end{proof}

This allows us to prove a Van Kampen theorem for $\GAL(X)$:

\begin{Thm}\label{Thm van Kampen}
The functor $\GAL_\kappa \colon \Sch^\kappa \to \Cond_\kappa(\Cat_\infty)$ is a pro-\'etale hypercosheaf.
\end{Thm}

\begin{proof}
Let $U_\bullet \to U$ be a pro-\'etale hypercover in $\Sch^\kappa$. By \autoref{Rmk hypersheaf hypercover}, we must show that the diagram
\[
\cdots \mathrel{\substack{\rightarrow\\[-0.6ex] \rightarrow\\[-0.6ex] \rightarrow}}\GAL_\kappa(U_1) \mathrel{\substack{\rightarrow\\[-0.6ex]  \rightarrow}} \GAL_\kappa(U_0) \to \GAL_\kappa(U)
\]
is a colimit diagram in $\Cond_\kappa(\Cat_\infty)$. It suffices to do this in $\Cond_\kappa(\Cat_\infty)_{/\GAL_\kappa(U)}$ by \cite[Prop.~1.2.13.8]{LurieHTT}. Each $\GAL_\kappa(U_i) \to \GAL_\kappa(U)$ is a left fibration by \autoref{Lem Gal left fibration}, so it suffices to check the hypercosheaf condition in $\LFib_{/\GAL_\kappa(U)}$ by \autoref{Lem colim left fibrations}. Under the equivalence \eqref{Eq equivalence left fibration}, this is equivalent to $\Sh^\hyp(U_{\kappa\text{-}\proet},\mathcal S)$, and the diagram corresponds to the diagram
\[
\cdots \mathrel{\substack{\rightarrow\\[-0.6ex] \rightarrow\\[-0.6ex] \rightarrow}} h_{U_1} \mathrel{\substack{\rightarrow\\[-0.6ex] \rightarrow}} h_{U_0} \to h_U
\]
by \autoref{Lem Gal left fibration}. The condition that this is a colimit diagram in $\Sh^\hyp(U_{\kappa\text{-}\proet},\mathcal S)$ is exactly the condition that $U_\bullet \to U$ is a hypercover.
\end{proof}

\begin{Cor}\label{Cor hypersheaf}
Let $\mathscr C$ be a $\kappa$-condensed $\infty$-category. Then the functor
\[
\Fun\cts(\GAL_\kappa(-),\mathscr C) \colon \Sch^\kappa \to \Cat_\infty
\]
is a pro-\'etale hypersheaf.
\end{Cor}

\begin{proof}
This follows from \autoref{Thm van Kampen} by \autoref{Rmk internal hom colim}.
\end{proof}

\begin{Ex}
The exodromy results of \autoref{Thm proetale exodromy}, \autoref{Thm etale exodromy}, and \autoref{Thm constructible exodromy} compute $\Fun\cts(\GAL_\kappa(-),\mathscr C)$ when $\mathscr C$ is $\COND_\kappa(\mathscr E)$, $\underline{\mathscr E}$, and $\underline{\mathcal S}{}\cons$ respectively. Thus, \autoref{Cor hypersheaf} shows that $\Sh^\hyp((-)_\proet,\mathscr E)$, $\Sh^{\post}((-)_\et,\mathscr E)$, and $\Sh\cons((-)_\et,\mathcal S)$ are pro-\'etale hypersheaves. These statements can also be proved directly. For the first, this is \autoref{Cor hypersheaf of categories}.
\end{Ex}

\begin{Par}\label{Par cosheaf results}
In the qcqs case and when $\kappa$ is strongly inaccessible, \cite{HHLMMW} proves two related descent statements for $\GAL_\kappa$:
\begin{itemize}
\item it is a hypercosheaf for the topology whose covers are given by jointly surjective integral maps \cite[Cor.~6.16]{HHLMMW};
\item it becomes a pro-\'etale hypercosheaf after applying the `invert everything' functor \cite[Prop.~3.17 and Prop.~3.38]{HHLMMW}.
\end{itemize}
\end{Par}

\begin{Rmk}
Since $\GAL_\kappa(X)$ is a $\kappa$-condensed $1$-category, the continuous unstraightening of a sheaf $\mathscr F \in \Sh^\hyp(X_{\kappa\text{-}\proet},\mathcal S)$ is an $n$-category if and only if $\mathscr F$ is $n$-truncated (if $n \geq 1$). Thus, \eqref{Eq equivalence left fibration} restricts to an equivalence
\[
\Sh(X_{\kappa\text{-}\proet},\mathcal S_{\leq n}) \simeq \LFib_{/\GAL_\kappa(X)} \underset{\COND_\kappa(\Cat_\infty)}\times \COND_\kappa(\Cat_n).
\]
By \autoref{Prop etale to proetale}, the pullback $\Sh(X_\et,\mathcal S_{\leq n}) \to \Sh(X_{\kappa\text{-}\proet},\mathcal S_{\leq n})$ is fully faithful. Under the equivalence above, the essential image of $\Sh(X_\et,\mathcal S_{\leq n})$ corresponds to the class of \makebox{$n$-categories} in $\LFib_{/\GAL_\kappa(X)}$ for which the structure map $\mathscr C \to \GAL_\kappa(X)$ has `discrete' fibres in a certain sense. This suggests thinking of the left fibration $\mathscr C \to \GAL_\kappa(X)$ defined by an \'etale sheaf $\mathscr F \in \Sh(X_\et,\mathcal S_{\leq n})$ as some sort of `espace \'etal\'e' of $\mathscr F$, analogous to the equivalence $\operatorname{LocalHomeo}_{/X} \simeq \Sh(X)$ for a topological space $X$.
\end{Rmk}

\subsection{Change of cardinal}\label{Sec change of cardinal}
The definition $\GAL_\kappa(X)$ appears to depend on a choice of auxiliary cardinal $\kappa > \size(X)$. In this section, we show that knowing a single value of $\GAL_\kappa(X)$ for $\kappa > \size(X)$ determines all values of $\kappa > \size(X)$.

If $\lambda > \kappa > \size(X)$, write $\nu \colon \ProFin_\kappa \hookrightarrow \ProFin_\lambda$ for the inclusion, inducing an adjunction $\nu_* \colon \Cat_\infty(\Cond_\lambda) \leftrightarrows \Cat_\infty(\Cond_\kappa):\!\nu^*$ as in \ref{Par internal cat}. It is clear that $\nu_*\GAL_\lambda(X) \simeq \GAL_\kappa(X)$. Conversely, we will prove in \autoref{Cor Kan extended} that $\GAL_\lambda(X) \simeq \nu^*\GAL_\kappa(X)$, so that knowing $\GAL_\kappa(X)$ for a single value of $\kappa > \size(X)$ determines $\GAL_\lambda(X)$ for any other $\lambda > \size(X)$. In particular, this implies that the functor $\GAL(X) \colon \ProFin\op \to \Cat$ is accessible, so $\GAL(X)$ is a condensed category in the sense of \cite[Def.~2.11]{ClausenScholze}.

In the lemma below, we write $\SpecTop$ for the category of spectral topological spaces and spectral continuous functions \cite[Tag \href{https://stacks.math.columbia.edu/tag/08YG}{08YG}]{Stacks}.

\begin{Lemma}\label{Lem kappa-compact?}
Let $X$ be a qcqs scheme, and let $\kappa = \size(X)^+$ be the successor cardinal. Let $P$ be a $\kappa$-filtered poset, let $p \mapsto S_p$ be a functor $P\op \to \SpecTop$ such that $S_q \to S_p$ is surjective for all $p \to q$ in $P$, and let $S = \lim_{p \in P\op} S_p$. Then the natural map
\[
\colim_{\substack{\longrightarrow \\ p \in P}} \Map_{\LTop^\coh}\!\big(\Sh(X_\et),\Sh(S_p)\big) \to \Map_{\LTop^\coh}\!\big(\Sh(X_\et),\Sh(S)\big)
\]
is an equivalence.
\end{Lemma}

Note that the qcqs hypothesis is needed to guarantee that $\Sh(X_\et)$ is coherent. We don't know if $\Sh(X_\et)$ is in fact $\kappa$-compact in $\LTop^\coh$, but we won't need this more general statement.

\begin{proof}
By \cite[Exp.~VI, Exc.~3.11]{SGA4II}, the functor $\mathscr X \mapsto \mathscr X^{\coh}$ is an equivalence between $\LTop^\coh$ and the category $\PreTop$ of small pretopoi. Applying this to $\Sh(X_\et)$ gives the category $\Sh(X_\et)^\coh = \Sh\cons(X_\et)$ of constructible \'etale sheaves (of sets). By \cite[Exp.~IX, Prop.~2.7]{SGA4III}, every constructible sheaf is a finite colimit of sheaves representable by finitely presented \'etale $X$-schemes. In particular, the set of constructible sheaves up to isomorphism has cardinality $\leq \size(X)$, and for constructible sheaves $\mathscr F$ and $\mathscr G$, the set $\Hom(\mathscr F,\mathscr G)$ has cardinality $\leq \size(X)$ (compare with the proof of \autoref{Prop enough wsl kappa}). Since ($\kappa$-)filtered colimits in $\Cat$ are computed as pointwise colimits in $\Cat_\infty(\mathcal S) \subseteq \Fun(\Delta\op,\mathcal S)$, this implies that $\Sh(X_\et)^\coh$ is $\kappa$-compact in $\Cat$. Since $\Sh(S)^\coh = \colim_{p \in P} \Sh(S_p)^\coh$ (see also \cite[8.1.4]{BGH}) and $P$ is $\kappa$-filtered, we conclude that the map
\[
\colim_{\substack{\longrightarrow \\ p \in P}} \Map_{\Cat}\!\big(\Sh(X_\et)^\coh,\Sh(S_p)^\coh\big) \to \Map_{\Cat}\!\big(\Sh(X_\et)^\coh,\Sh(S)^\coh\big)
\]
is an equivalence. It remains to prove that a functor $s^* \colon \Sh(X_\et)^\coh \to \Sh(S_p)^\coh$ is a morphism of pretopoi (i.e., preserves finite limits, finite coproducts, and quotients by equivalence relations) if and only if the composition $\pi_p^* \circ s^*$ is, where $\pi_p \colon S \to S_p$ is the natural map. Since $\pi_p^*$ is a morphism of pretopoi, it is clear that $\pi_p^* \circ s^*$ is a morphism of pretopoi if $s^*$ is. Conversely, suppose $\pi_p^* \circ s^*$ is a morphism of pretopoi. The surjectivity of the transition maps $S_q \to S_p$ implies the same for $\pi_p$ \cite[Tag \href{https://stacks.math.columbia.edu/tag/0A2W}{0A2W}]{Stacks}. Thus $\pi_p^*$ is conservative, so if $\mathscr D$ is a small category, then $s^*$ preserves all $\mathscr D$-indexed (co)limits if $\pi_p^*$ and $\pi_p^* \circ s^*$ do so.
\end{proof}

\begin{Cor}\label{Cor Kan extended}
Let $X$ be a scheme, and let $\lambda > \kappa > \size(X)$ be cardinals. Then $\nu^*\GAL_\kappa(X) \simeq \GAL_\lambda(X)$, where $\nu^* \colon \Cat_\infty(\Cond_\kappa) \to \Cat_\infty(\Cond_\lambda)$ denotes the pullback.
\end{Cor}

\begin{proof}
First assume that $X$ is qcqs. Write $\nu_{\text{pre}}^* \colon\! 
 \PSh(\ProFin_\kappa,\Cat_\infty) \to \PSh(\ProFin_\lambda,\Cat_\infty)$ for the left Kan extension; then $\nu^*$ is obtained by applying $\nu_{\text{pre}}^*$ followed by sheafification (if $\kappa$ and $\lambda$ are regular, the last step is not needed \cite[Prop.~2.1.8]{Mann}, but this will not play a role in our proof). We know that $\nu_*\GAL_\lambda(X) \simeq \GAL_\kappa(X)$, and it suffices to show that the counit $\nu_{\text{pre}}^*\GAL_\kappa(X) \to \GAL_\lambda(X)$ is an equivalence. Let $S \in \ProFin_\lambda$, and let $\mathscr C = (S \downarrow \ProFin_\kappa)\op$. Then the map $(\nu_{\text{pre}}^*\GAL_\kappa(X))(S) \to \GAL_\lambda(X)(S)$ is given by
\begin{equation}
\colim_{\substack{\longrightarrow \\ T \in \mathscr C}} \Fun_\coh^*(\Sh(X_\et),\Sh(T)) \to \Fun_\coh^*(\Sh(X_\et),\Sh(S)).\label{Eq left Kan}
\end{equation}
Note that the category $\mathscr C$ is $\size(X)^+$-filtered since it has $\kappa$-small colimits and $\kappa \geq \size(X)^+$. Let $\mathscr C' \subseteq \mathscr C$ be the full subcategory on morphisms $S \to T$ that are surjective. Since $(S \downarrow \ProFin_\kappa)$ has finite limits and every map $S \to T$ factors as $S \twoheadrightarrow T' \hookrightarrow T$ with $T' \in \ProFin_\kappa$, we see that $\mathscr C'$ is $\size(X)^+$-filtered and the inclusion $\mathscr C' \subseteq \mathscr C$ is cofinal (see \cite[Prop.~3.2.4]{KashiwaraSchapira} for the $\omega$-filtered case). We may choose a $\size(X)^+$-filtered poset $P$ and a cofinal functor $P \to \mathscr C'$ (see e.g.\ \cite[Prop.~5.3.1.16]{LurieHTT}). Writing this as $p \mapsto (S \to S_p)$, we know that each $S \to S_p$ is surjective and $S = \lim_{p \in P} S_p$. Then each $S_q \to S_p$ for $p \to q$ in $P$ is also surjective, so the diagram $p \mapsto S_p \times [n]$ satisfies the hypothesis of \autoref{Lem kappa-compact?} for all $n \in \mathbf N$, where $[n]$ denotes the finite T$_0$-space given by the Alexandrov topology on~$[n]$. Since $\Map([n],\Fun_\coh^*(\Sh(X_\et),\Sh(T))) \simeq \Map_{\LTop^\coh}(\Sh(X_\et),\Sh(T \times [n]))$ for all $T \in \SpecTop$ and filtered colimits in complete Segal spaces are computed in $\Fun(\Delta\op,\mathcal S)$, we conclude from \autoref{Lem kappa-compact?} that \eqref{Eq left Kan} is an equivalence. This proves the result when $X$ is qcqs.

If $X$ is an arbitrary $\kappa$-small scheme, choose an affine open cover $X = \bigcup_{i \in I} U_i$ and affine covers $U_i \cap U_j = \bigcup_{k \in K_{i,j}} V_k$. Then the $U_i$ and $V_k$ are $\kappa$-small affine schemes, so the counit $\nu^*\GAL_\kappa(W) \to \GAL_\lambda(W)$ is an equivalence for $W = U_i$ and $W = V_k$. The result for $X$ now follows from \autoref{Thm van Kampen} since $\nu^*$ preserves colimits.
\end{proof}

\appendix

\titleformat{\section}[block]{\Large\bfseries\scshape\filcenter}{Appendix \thesection.}{1ex}{}
\section{Generalities on \texorpdfstring{$\infty$}{∞}-topoi}\label{Sec appendix A}
In this appendix, we review some general results on $\infty$-topoi that are used in the paper.

\subsection{Sheaves with values}
This section is a review of (hyper)sheaves with values in an $\infty$-category $\mathscr E$. We give a general definition, which we then compare with other standard definitions in special cases.

\begin{Def}
Write $\widehat{\Cat}{}_\infty^R$ for the $\infty$-category of large $\infty$-categories with small limits, where morphisms are limit-preserving functors. Likewise, write $\widehat{\Cat}{}_\infty^L$ for the $\infty$-category of large $\infty$-categories with small colimits, where morphisms are colimit-preserving functors. Note that $(-)\op$ gives an equivalence $\widehat{\Cat}{}_\infty^R \simeq \widehat{\Cat}{}_\infty^L$. Given $\mathscr C \in \widehat{\Cat}{}_\infty^R$ and $\mathscr D \in \widehat{\Cat}{}_\infty$, write $\Fun^R(\mathscr C,\mathscr D) \subseteq \Fun(\mathscr C,\mathscr D)$ for the full (large) $\infty$-subcategory of limit-preserving functors, and likewise for $\Fun^L$. (As in \cite[Notation 5.1.5.1]{LurieHTT}, we use this notation even outside the presentable setting, noting that limit preserving functors need not necessarily be right adjoints in this level of generality.)
\end{Def}

\begin{Def}\label{Def sheaf with values}
Let $\mathscr X$ be an $\infty$-topos, and let $\mathscr E$ be an $\infty$-category. Then we write $\Sh(\mathscr X,\mathscr E)$ for the $\infty$-category $\Fun^R(\mathscr X\op,\mathscr E)$ of limit preserving functors $\mathscr X\op \to \mathscr E$. (See also \cite[Notation 6.3.5.16]{LurieHTT} and \cite[Def.~1.3.1.4]{LurieSAG}.)
\end{Def}

\begin{Rmk}\label{Rmk presentable tensor}
If $\mathscr E$ is presentable, then $\Sh(\mathscr X,\mathscr E)$ agrees with the Lurie tensor product $\mathscr X \otimes \mathscr E$ \cite[Prop.~4.8.1.17]{LurieHA}.
\end{Rmk}

However, sheaves with values are also of interest in cases where $\mathscr E$ is not presentable. For instance, in \hyperref[Sec van Kampen]{\S4.1}, we study hypercosheaves of $\kappa$-condensed $\infty$-categories, meaning that we take $\mathscr E = \Cond_\kappa(\Cat_\infty)\op$.

\begin{Rmk}\label{Rmk sheaf limit preserving functor}
If $\mathscr X$ is a presheaf $\infty$-topos $\PSh(\mathscr C)$ for a small $\infty$-category $\mathscr C$ and $\mathscr E$ has small limits, then the universal property of $\PSh(\mathscr C)$ \cite[Thm.~5.1.5.6]{LurieHTT} shows that $\Sh(\mathscr X,\mathscr E) \simeq \PSh(\mathscr C,\mathscr E)$. If~$\mathscr X$ is an accessible left exact localisation $\iota \colon \mathscr X \leftrightarrows \PSh(\mathscr C) :\! L$ at some class of morphisms $S$ inside a presheaf category $\PSh(\mathscr C)$ for some small $\infty$-category $\mathscr C$, then composition with the localisation $L \colon \PSh(\mathscr C) \to \mathscr X$ gives a functor
\[
F \colon \Sh(\mathscr X,\mathscr E) \to \Sh(\PSh(\mathscr C),\mathscr E) \simeq \PSh(\mathscr C,\mathscr E).
\]
Since $L$ is an accessible localisation, the functor $F$ above is fully faithful and identifies $\Fun^R(\mathscr X\op,\mathscr E)$ with the full $\infty$-subcategory of $\Fun^R(\PSh(\mathscr C)\op,\mathscr E)$ of functors that take morphisms in $S$ to equivalences \cite[Prop.~5.5.4.20]{LurieHTT}. In other words, $\Sh(\mathscr X,\mathscr E)$ consists of those functors $\mathscr C\op \to \mathscr E$ such that the unique small limit preserving extension $\PSh(\mathscr C)\op \to \mathscr E$ factors through $\mathscr X\op$.
\end{Rmk}

\begin{Rmk}\label{Rmk sheaf local}
Assume $\mathscr X$ is a left exact localisation \makebox{$\iota \colon \mathscr X \leftrightarrows \PSh(\mathscr C) :\! L$} for some small $\infty$-category $\mathscr C$ and that $\mathscr E$ has small limits. A more common reformulation of the final statement of \autoref{Rmk sheaf limit preserving functor} is that $\Sh(\mathscr X,\mathscr E) \subseteq \PSh(\mathscr C,\mathscr E)$ consists of those functors $\mathscr F \colon \mathscr C\op \to \mathscr E$ such that the presheaf of spaces $\Map_{\mathscr E}(E,\mathscr F(-))$ is in $\mathscr X$ for all $E \in \mathscr E$. Indeed, for every $E \in \mathscr E$, the functor $\Map_{\mathscr E}(E,-)$ preserves limits, so the diagram
\[
\begin{tikzcd}
\PSh(\mathscr C,\mathscr E) \ar{r}{\sim}[swap]{\phi}\ar{d}[swap]{\Map_{\mathscr E}(E,-)} & \Fun^R(\PSh(\mathscr C)\op,\mathscr E) \ar{d}{\Map_{\mathscr E}(E,-)} \\
\PSh(\mathscr C) \ar{r}{\sim}[swap]{\psi} & \Fun^R(\PSh(\mathscr C)\op,\mathcal S)
\end{tikzcd}
\]
commutes (up to natural isomorphism). Since the functors $\Map_{\mathscr E}(E,-)$ for $E \in \mathscr E$ are jointly conservative, we see that $\phi(\mathscr F)$ sends morphisms in the class $S$ of \autoref{Rmk sheaf limit preserving functor} to equivalences if and only if the same holds for $\psi(\Map_{\mathscr E}(E,\mathscr F))$ for all $E \in \mathscr E$. Since $\psi$ is given by $\mathscr G \mapsto \Map_{\PSh(\mathscr C)}(-,\mathscr G)$, this means that $\Map_{\mathscr E}(E,\mathscr F)$ is $S$-local for all $E \in \mathscr E$, i.e.\ lies in $\mathscr X$.
\end{Rmk}

\begin{Def}\label{Def sheaf hypersheaf}
If $\mathscr X = \Sh(\mathscr C)$ for a small $\infty$-category $\mathscr C$ with a Grothendieck topology, and $\mathscr E$ is an $\infty$-category, we write $\Sh(\mathscr C,\mathscr E)$ for $\Sh(\Sh(\mathscr C),\mathscr E)$. Likewise, if $\bullet \in \{\hyp,\post\}$, we write $\Sh^\bullet(\mathscr C,\mathscr E)$ for $\Sh(\Sh^\bullet(\mathscr C),\mathscr E)$.
\end{Def}

\begin{Rmk}\label{Rmk sheaf condition}
Let $\mathscr X = \Sh(\mathscr C)$ for a small $\infty$-category $\mathscr C$ with a Grothendieck topology, and assume that $\mathscr E$ has small limits. Then the class $\overline S$ of morphisms in $\PSh(\mathscr C)$ that become equivalences in $\Sh(\mathscr C)$ is generated by the set $S$ of monomorphisms $\mathcal I \hookrightarrow h_C$ that correspond to covering sieves on $C \in \mathscr C$ \cite[Def.~6.2.2.6]{LurieHTT}. By \autoref{Rmk sheaf local}, we see that $\Sh(\mathscr X,\mathscr E)$ coincides with those functors $\mathscr F \colon \mathscr C\op \to \mathscr E$ such that for every covering sieve $\mathscr C_{/C}^{(0)} \subseteq \mathscr C_{/C}$, the natural map
\begin{equation}
\mathscr F(C) \to \lim_{U \in \mathscr C_{/C}^{(0)}} \mathscr F(U)\label{Eq sheaf}
\end{equation}
is an equivalence. This is the obvious extension of the usual sheaf condition (see for instance \cite[Def.~6.2.2.6]{LurieHTT} or \cite[Exp.~I, 4.3 and Exp.~II, Def.~2.1]{SGA4I}) to $\mathscr E$-valued sheaves. When the Grothendieck topology on $\mathscr C$ comes from a pretopology \cite[Def.~2.1]{PortaYu}, this is equivalent to the more familiar condition that the diagram
\[
\mathscr F(C) \to \prod_{i \in I} \mathscr F(U_i) \mathrel{\substack{\rightarrow\\[-0.6ex] \rightarrow}} \prod_{i,j \in I} \mathscr F(U_i \times_C U_j) \mathrel{\substack{\rightarrow\\[-0.6ex] \rightarrow\\[-0.6ex] \rightarrow}} \cdots
\]
is a limit diagram in $\mathscr E$ for every covering $\{U_i \to C\}_{i \in I}$. Indeed, if $\mathscr C_{/C}^{(0)} \subseteq \mathscr C_{/C}$ is the sieve generated by $\{U_i \to C\}_{i \in I}$, then the corresponding monomorphism $\mathcal I \hookrightarrow h_C$ is the $(-1)$-truncation of $\coprod_{i \in I} h_{U_i} \to h_C$ \cite[Lem.~6.2.3.18]{LurieHTT}. In other words, $\mathcal I$ is the geometric realisation of the \v Cech nerve of $\coprod_{i \in I} h_{U_i} \to h_C$ \cite[Prop.~6.2.3.4]{LurieHTT}. The unique limit-preserving extension of $\mathscr F$ to $\PSh(\mathscr C)\op$ takes $\coprod_{i \in I} h_{U_i}$ to $\prod_{i \in I} \mathscr F(U_i)$, and analogously for the iterated fibre products since they are representable in $\mathscr C$ by \cite[Def.~2.1(iii)]{PortaYu}. The geometric realisation is then taken to the limit above.
\end{Rmk}

\begin{Rmk}\label{Rmk hypersheaf hypercover}
Likewise, let $\mathscr X = \Sh^\hyp(\mathscr C)$ for a small $\infty$-category $\mathscr C$ with a Grothendieck topology, and assume that $\mathscr E$ has small limits. Let $\overline S \subseteq \Fun([1],\PSh(\mathscr C))$ be the class of morphisms that become equivalences after hypersheafification. We say that an augmented simplicial presheaf $\mathscr G_\bullet \to h_C$ is a \emph{hypercovering} of $C \in \mathscr C$ if each $\mathscr G_n$ is a coproduct $\coprod_{i \in I_n} h_{U_i}$ of representable presheaves and the map $\lvert \mathscr G_\bullet \rvert \to h_C$ is in $\overline S$. If $S \subseteq \Fun([1],\PSh(\mathscr C))$ denotes the class of morphisms $\lvert \mathscr G_\bullet \rvert \to h_C$ arising in this way, then $\overline S$ is generated by $S$ under colimits in $\Fun([1],\PSh(\mathscr C))$. We conclude that $\mathscr F \colon \mathscr C\op \to \mathscr E$ is in $\Sh^\hyp(\mathscr C,\mathscr E)$ if and only if, for every hypercovering $\mathscr G_\bullet \to h_C$ with $\mathscr G_n \simeq \coprod_{i \in I_n} h_{U_i}$, the diagram
\[
\mathscr F(C) \to \prod_{i \in I_0} \mathscr F(U_i) \mathrel{\substack{\rightarrow\\[-0.6ex] \rightarrow}} \prod_{i \in I_1} \mathscr F(U_1) \mathrel{\substack{\rightarrow\\[-0.6ex] \rightarrow \\[-0.6ex] \rightarrow}} \cdots
\]
is a limit diagram in $\mathscr E$.
\end{Rmk}

\begin{Rmk}
If $\mathscr E$ does not have small limits, \autoref{Def sheaf with values} does not always give the expected answer. For instance, in the presheaf case of \autoref{Rmk sheaf limit preserving functor}, one would like to consider $\PSh(\mathscr C,\mathscr E)$. For instance, if $\mathscr C = *$ is a point, we expect to get~$\mathscr E$. But instead, \autoref{Def sheaf with values} gives the objects $x \in \mathscr E$ for which all small limits over the constant diagram with value $x$ exist. For instance, when $\mathscr E$ is the category of finite dimensional $k$-vector spaces, $\Sh(\PSh(*),\mathscr E)$ only contains the $0$ object. The situation is similar in the (hyper)sheaf cases of \autoref{Rmk sheaf condition} and \autoref{Rmk hypersheaf hypercover}.

However, the advantage of \autoref{Def sheaf with values} is that it only depends on the $\infty$-topos $\mathscr X$, and not on a presentation as a presheaf/sheaf/hypersheaf $\infty$-category. As we showed above, in cases where $\mathscr E$ has small limits, it coincides with the hands-on definition of presheaves/sheaves/hypersheaves.
\end{Rmk}

\begin{Rmk}\label{Rmk hypersheaf basis}
Let $\mathscr C$ be a small $\infty$-category with a Grothendieck topology, and let $\mathscr B \subseteq \mathscr C$ be a basis, i.e.\ a full subcategory such that for every $C \in \mathscr C$, the sieve generated by $\mathscr B \times_{\mathscr C} \mathscr C_{/C} \subseteq \mathscr C_{/C}$ is covering. Let $\mathscr E$ be an $\infty$-category with small limits. Then the restriction $\Sh^\hyp(\mathscr C,\mathscr E) \to \Sh^\hyp(\mathscr B,\mathscr E)$ is an equivalence. Indeed, for $\mathscr E = \mathcal S$, this is proved in \cite[Cor.~A.7]{Aoki}, and the general result follows by the definition of $\Sh^\hyp(-,\mathscr E)$ as $\Fun^R(\Sh^\hyp(-)\op,\mathscr E)$. (When the topologies on $\mathscr C$ and $\mathscr B$ come from pretopologies, it can also be proved using the methods of \cite[\S2.4]{PortaYu}. When the representable presheaves on $\mathscr C$ are hypersheaves, this is \cite[Prop.~2.22]{PortaYu}, but it is not hard to remove this hypothesis.)
\end{Rmk}

In the remainder of this section, we will show that $\mathscr E$-valued sheaves on an $\infty$-topos $\mathscr X$ form a sheaf of large $\infty$-categories on $\mathscr X$; see \autoref{Prop sheaf of categories} for a precise statement. We first need to construct it as a presheaf.

\begin{Par}\label{Par large codomain fibration}
Let $\mathscr X$ be an $\infty$-topos. Then the codomain functor $\Fun([1],\mathscr X) \to \mathscr X$ is a large cocartesian fibration by \cite[Cor.~2.4.7.11]{LurieHTT} or \cite[Tag \href{https://kerodon.net/tag/01UU}{01UU}]{Kerodon}. Therefore, its straightening is a functor $\mathscr X \to \widehat{\Cat}_\infty$, informally given by $U \mapsto \mathscr X_{/U}$ and taking $f \colon V \to U$ to the natural functor $f_! \colon \mathscr X_{/V} \to \mathscr X_{/U}$ given by $(g \colon W \to V) \mapsto (f \circ g \colon W \to U)$. The latter has a right adjoint given by base change, so the functor $\mathscr X \to \widehat{\Cat}_\infty$ lands in $\PrL$. (Note that this is not the same thing as the functor $\mathscr X\op \to \PrL$ of \cite[Prop.~6.1.1.4]{LurieHTT}, which exists because the base change map $f^* \colon \mathscr X_{/U} \to \mathscr X_{/V}$ preserves small colimits, hence admits a further right adjoint $f_*$.)

If $\mathscr E$ is an $\infty$-category with small limits, there is a functor $\Fun^R(-,\mathscr E) \colon (\widehat{\Cat}{}_\infty^R)\op \to \widehat{\Cat}_\infty$ taking $\mathscr C \in \widehat{\Cat}{}_\infty^R$ to $\Fun^R(\mathscr C,\mathscr E)$ and taking a limit-preserving functor $f \colon \mathscr C \to \mathscr D$ to the precomposition $\Fun^R(\mathscr D,\mathscr E) \to \Fun^R(\mathscr C,\mathscr E)$. Using the equivalence $(-)\op \colon \widehat{\Cat}{}_\infty^R \xrightarrow\sim \widehat{\Cat}{}_\infty^L$, we get the functor
\[
\Fun^R((-)\op,\mathscr E) \colon (\widehat{\Cat}{}_\infty^L)\op \to \widehat{\Cat}_\infty.
\]
Then the composition
\[
\mathscr X\op \to (\PrL)\op \subseteq (\widehat{\Cat}{}_\infty^L)\op \to \widehat{\Cat}_\infty
\]
is a functor given by $U \mapsto \Fun^R(\mathscr X_{/U}\op,\mathscr E) = \Sh(\mathscr X_{/U},\mathscr E)$ and taking a morphism $f \colon V \to U$ in $\mathscr X$ to the functor $\Sh(\mathscr X_{/U},\mathscr E) \to \Sh(\mathscr X_{/V},\mathscr E)$ given by precomposition with the limit-preserving functor $\mathscr X_{/V}\op \to \mathscr X_{/U}\op$. We denote this functor by $\Sh(\mathscr X_{/-},\mathscr E) \colon \mathscr X\op \to \widehat{\Cat}_\infty$.
\end{Par}

\begin{Ex}\label{Ex large codomain fibration spaces}
If $\mathscr E = \mathcal S$, then the functor $\Sh(\mathscr X_{/-},\mathcal S)$ is canonically identified with $\Fun^L(\mathcal S,\mathscr X_{/-}\op)\op \simeq \mathscr X_{/-}$, which is classified by the Cartesian fibration $\Fun([1],\mathscr X) \to \mathscr X$ of \cite[Lem.~6.1.1.1 and Prop.~6.1.1.4]{LurieHTT}. Thus, the functor $\Sh(\mathscr X_{/-},\mathcal S) \colon \mathscr X\op \to \widehat{\Cat}_\infty$ preserves small limits by \cite[Thm.~6.1.3.9, Prop.~6.1.3.10, and Prop.~5.5.3.13]{LurieHTT}. The following result shows that this holds more generally:
\end{Ex}

\begin{Prop}\label{Prop sheaf of categories}
Let $\mathscr X$ be an $\infty$-topos, and let $\mathscr E$ be an $\infty$-category with small limits. Then the functor $\Sh(\mathscr X_{/-},\mathscr E) \colon \mathscr X\op \to \widehat{\Cat}_\infty$ preserves small limits.
\end{Prop}

When $\mathscr E$ is presentable, this follows quickly from the $\mathcal S$-valued case of \autoref{Ex large codomain fibration spaces}: applying \cite[Prop.~4.8.1.17]{LurieHA} twice shows that $\Sh(\mathscr X_{/-},\mathscr E)$ is equivalent to $\Fun^R(\mathscr E\op,\mathscr X_{/-})$, and $\Fun^R(\mathscr E\op,-) \colon \PrR \to \PrR$ preserves limits (see also \cite[Rmk.~2.1.0.5]{LurieSAG}).

\begin{proof}
The result for $\mathscr E = \mathcal S$ was explained in \autoref{Ex large codomain fibration spaces}. When $\mathscr E = \widehat{\mathcal S}$ is an \makebox{$\infty$-category} of large spaces (in a larger universe), the functor $\Sh(\mathscr X_{/-},\widehat{\mathcal S}\,) \colon \mathscr X\op \to \widehatt{\Cat}_\infty$ is canonically equivalent to the composition of the inclusion $\mathscr X\op \to \widehat{\mathscr X}\op = \Sh(\mathscr X,\widehat{\mathcal S}\,)\op$ with the functor $\Sh(\widehat{\mathscr X}_{/-},\widehat{\mathcal S}\,) \colon \widehat{\mathscr X}\op \to \widehatt{\Cat}_\infty$ (defined as in \ref{Par large codomain fibration}) by \cite[Prop.~2.4.3]{Martini}. The first functor preserves small limits by \cite[Rmk.~6.3.5.17]{LurieHTT}, and the second by \autoref{Ex large codomain fibration spaces} applied within the larger universe.

If $\mathscr E \in \widehat{\Cat}_\infty$ is an arbitrary $\infty$-category with small limits, we consider the Yoneda embedding $\mathscr E \to \PSh(\mathscr E,\widehat{\mathcal S}\,)$. Let $D \colon \mathcal I \to \mathscr X$ be a small diagram given by $i \mapsto U_i$, and let $U$ be its colimit in $\mathscr X$. We get a commutative square
\begin{equation}
\begin{tikzcd}
\Sh(\mathscr X_{/U},\mathscr E) \ar{r}\ar{d} & \lim\limits_{i \in \mathcal I} \Sh(\mathscr X_{/U_i},\mathscr E) \ar{d} \\
\Sh(\mathscr X_{/U},\PSh(\mathscr E,\widehat{\mathcal S}\,)) \ar{r} & \lim\limits_{i \in \mathcal I} \Sh(\mathscr X_{/U_i},\PSh(\mathscr E,\widehat{\mathcal S}\,))
\end{tikzcd}\label{Dia large Yoneda}
\end{equation}
whose vertical arrows are fully faithful. Since limits in presheaves are pointwise, the bottom arrow is identified with
\[
\PSh\!\big(\mathscr E,\Sh(\mathscr X_{/U},\widehat{\mathcal S}\,)\big) \to \PSh\!\Big(\mathscr E,\lim_{i \in \mathcal I} \Sh(\mathscr X_{/U_i},\widehat{\mathcal S}\,)\Big),
\]
which is an equivalence of (very large) $\infty$-categories by the case $\widehat{\mathcal S}$ treated above. Thus, the top map in \eqref{Dia large Yoneda} is fully faithful. To show it is an equivalence, it suffices to show that \eqref{Dia large Yoneda} is a pullback square in $\widehatt{\Cat}_\infty$. Suppose $F \colon \mathscr X_{/U}\op \to \PSh(\mathscr E,\widehat{\mathcal S}\,)$ is a limit preserving functor whose restriction to $\mathscr X_{/U_i}\op$ lands in $\mathscr E$ for all $i$. The colimit diagram $D^{\triangleright} \colon \mathcal I^{\triangleright} \to \mathscr X$ taking $i \in \mathcal I$ to $U_i$ and the cone point to $U$ naturally lifts to a colimit diagram in $\mathscr X_{/U}$ \cite[Prop.~1.2.13.8]{LurieHTT}. If $f \colon V \to U$ is any object of $\mathscr X_{/U}$, the pullback diagram $f^*D^{\triangleright}$ is a colimit diagram in $\mathscr X_{/V}$ since colimits in $\mathscr X$ are universal \cite[Prop.~6.1.3.10(1)]{LurieHTT}. The forgetful functor $\mathscr X_{/V} \to \mathscr X_{/U}$ takes it to a colimit diagram in $\mathscr X_{/U}$ by another application of \cite[Prop.~1.2.13.8]{LurieHTT}. Writing $V_i = f^*U_i$, we conclude that $V = \colim_{i \in \mathcal I} V_i$ in $\mathscr X_{/U}$. The maps $V_i \to U_i$ witness $V_i$ as an object of $\mathscr X_{/U_i}$, showing that $F(V_i)$ is representable for all $i \in \mathcal I$. Thus, $F(V)$ is representable for all $V \in \mathscr X_{/U}$ since $F$ preserves small limits and $\mathscr E \subseteq \PSh(\mathscr E,\widehat{\mathcal S}\,)$ is closed under small limits. Thus, \eqref{Dia large Yoneda} is a pullback square, which finishes the proof.
\end{proof}

\begin{Cor}\label{Cor hypersheaf of categories}
Let $\mathscr C$ be an $\infty$-site, and let $\mathscr E$ be an $\infty$-category with small limits. Then the functor\vspace{-.5em}
\begin{align*}
\mathscr C\op &\to \widehat{\Cat}_\infty \\
U &\mapsto \Sh^\hyp(\mathscr C_{/U},\mathscr E)\\[-1.7em]
\end{align*}
is a hypersheaf.
\end{Cor}

\begin{proof}
Since $\Sh^\hyp(\mathscr C_{/U}) \simeq \Sh^\hyp(\mathscr C)_{/h_U^\hyp}$, this follows from \autoref{Prop sheaf of categories} applied to $\mathscr X = \Sh^\hyp(\mathscr C)$.
\end{proof}

\subsection{Internal higher category theory}\label{App internal}
We need some generalities about higher category theory internal to an $\infty$-topos $\mathscr X$. In our application, this will be the $\infty$-topos of $\kappa$-condensed spaces $\Cond_\kappa(\mathcal S)$, but we review the basics in larger generality.

\begin{Par}\label{Par internal cat}
If $\mathscr X$ is an $\infty$-topos, write $\Cat_\infty(\mathscr X) \subseteq \Fun(\Delta\op,\mathscr X)$ for the complete Segal objects; see for instance \cite[Def.~13.1.1]{BGH} or \cite[Def.~3.2.4]{Martini}. Note that the inclusion $\Cat_\infty(\mathscr X) \subseteq \Fun(\Delta\op,\mathscr X)$ is an accessible localisation; see for instance \cite[Prop.~3.2.9]{Martini}. If $f_* \colon \mathscr Y \to \mathscr X$ is a geometric morphism, then it induces an adjunction
\[
\begin{tikzcd}
\Fun(\Delta\op,\mathscr Y) \ar[yshift=-.4em]{r}[swap]{f_*}\ar[phantom,"\dashv" rotate=-90]{r} & \Fun(\Delta\op,\mathscr X)\ar[yshift=.4em]{l}[swap]{f^*}\punct{.}
\end{tikzcd}
\]
Both $f_*$ and $f^*$ preserve complete Segal objects since these are determined by finite limit conditions. Thus, they induce an adjunction
\begin{equation}
\begin{tikzcd}
\Cat_\infty(\mathscr Y) \ar[yshift=-.4em]{r}[swap]{f_*}\ar[phantom,"\dashv" rotate=-90]{r} & \Cat_\infty(\mathscr X)\ar[yshift=.4em]{l}[swap]{f^*}\punct{.}
\end{tikzcd}\label{Dia adjunction cat}
\end{equation}
\end{Par}

\begin{Par}\label{Par Cat}
If $\mathscr X$ is an $\infty$-topos, the $\infty$-category $\Cat_\infty(\mathscr X)$ is canonically equivalent to the $\infty$-category $\Sh(\mathscr X,\Cat_\infty)$ of \autoref{Def sheaf with values}. Indeed, for $\mathscr X = \mathcal S$, this is the complete Segal space model for $\Cat_\infty$ \cite[Thm.~4.11]{JoyalTierney}. In general, if $\mathscr X \leftrightarrows \PSh(\mathscr C)$ is an accessible left exact localisation at a strongly saturated class of morphisms $S$, then both $\Cat_\infty(\mathscr X)$ and $\Sh(\mathscr X,\Cat_\infty)$ are defined inside $\PSh(\Delta\times\mathscr C)$ as the objects that satisfy the completeness and Segal conditions in the first variable and are $S$-local in the second. Note also that $\Sh(\mathscr X,\Cat_\infty) \simeq \mathscr X \otimes \Cat_\infty$ by \autoref{Rmk presentable tensor} and \cite[Tag \href{https://kerodon.net/tag/06NJ}{06NJ}]{Kerodon}.
\end{Par}

\begin{Par}
Let $\mathscr X$ be an $\infty$-topos. A map $f \colon \mathscr C \to \mathscr D$ in $\Cat_\infty(\mathscr X)$ is a \emph{left fibration} if the square
\begin{equation}
\begin{tikzcd}
\mathcal Fun([1],\mathscr C) \ar{r}{f \circ -}\ar{d}[swap]{\ev_0} & \mathcal Fun([1],\mathscr D) \ar{d}{\ev_0} \\
\mathscr C \ar{r}[swap]{f} & \mathscr D
\end{tikzcd}\label{Dia left fibration}
\end{equation}
is a pullback in $\Cat_\infty(\mathscr X)$ \cite[Def.~4.1.1]{Martini}, where $\mathcal Fun$ denotes the internal functor category (denoted by $[-,-]$ in \cite[\S3.4]{Martini}) and $[1]$ denotes the constant category object (the image of $[1]$ under the functor $\Gamma^* \colon \Cat_\infty = \Cat_\infty(\mathcal S) \to \Cat_\infty(\mathscr X)$ induced by the terminal morphism of $\infty$-topoi $\Gamma_* \colon \mathscr X \to \mathcal S$). Given $\mathscr C \in \Cat_\infty(\mathscr X)$, we write $\LFib_{/\mathscr C} \subseteq \Cat_\infty(\mathscr X)_{/\mathscr C}$ for the full $\infty$-subcategory of left fibrations. (This is denoted $\LFib(\mathscr C)$ in \cite[\S4]{Martini}, where $\LFib_{\mathscr C}$ is reserved for an internal variant \cite[\S4.5]{Martini}.)
\end{Par}

\begin{Par}\label{Par preserve left fibration}
Let $f_* \colon \mathscr Y \to \mathscr X$ be a geometric morphism of $\infty$-topoi. Then $f_*$ and $f^*$ preserve left fibrations. Indeed, we have $\mathcal Fun([1],f_*\mathscr C) \simeq f_*\mathcal Fun([1],\mathscr C)$ and likewise for $f^*$. The results again follows since the square \eqref{Dia left fibration} is a finite limit condition, and both $f_*$ and $f^*$ preserve finite limits. In particular, if $\iota \colon \mathscr X \leftrightarrows \PSh(\mathcal I) :\! L$ is an accessible left exact localisation, then a functor $f \colon \mathscr C \to \mathscr D$ in $\Cat_\infty(\mathscr X)$ is a left fibration if and only if $\iota(f)$ is a left fibration in $\Cat_\infty(\PSh(\mathcal I))$.
\end{Par}

\begin{Rmk}\label{Rmk internal hom colim}
If $\mathscr X$ is an $\infty$-topos and $\mathscr C \in \Cat_\infty(\mathscr X)$ is an object, then the functor $\mathscr C \times (-) \colon \Cat_\infty(\mathscr X) \to \Cat_\infty(\mathscr X)$ preserves colimits. Indeed, the analogous statement is true in $\Fun(\Delta\op,\mathscr X)$ since it is an $\infty$-topos \cite[Prop.~6.1.3.10(1)]{LurieHTT}. The result in $\Cat_\infty(\mathscr X)$ follows since the localisation functor $L \colon \Fun(\Delta\op,\mathscr X) \to \Cat_\infty(\mathscr X)$ preserves finite products \cite[Cor.~3.2.12]{Martini} and colimits in $\Cat_\infty(\mathscr X)$ are computed by taking the colimit in $\Fun(\Delta\op,\mathscr X)$ and applying $L$.

Using the equivalence $\Map(\mathscr C,\mathcal Fun(-,\mathscr E)) \simeq \Map(\mathscr C \times (-),\mathscr E)$ for $\mathscr E \in \Cat_\infty(\mathscr X)$, we conclude that $\mathcal Fun(-,\mathscr E) \colon \Cat_\infty(\mathscr X)\op \to \Cat_\infty(\mathscr X)$ preserves limits for any $\mathscr E \in \Cat_\infty(\mathscr X)$. In particular, taking global sections, we see that $\Fun(-,\mathscr E) \colon \Cat_\infty(\mathscr X)\op \to \Cat_\infty$ preserves limits.
\end{Rmk}

\begin{Lemma}\label{Lem colim left fibrations}
Let $\mathscr X$ be an $\infty$-topos and let $\mathscr C \in \Cat_\infty(\mathscr X)$. Then the inclusion functor $\LFib_{/\mathscr C} \to \Cat_\infty(\mathscr X)_{/\mathscr C}$ preserves small colimits.
\end{Lemma}

\begin{proof}
Recall that a colimit in $\Cat_\infty(\mathscr X)_{/\mathscr C}$ of a diagram $D \colon \mathscr D \to \Cat_\infty(\mathscr X)_{/\mathscr C}$ is the colimit $D(\infty)$ of the composition $\mathscr D \to \Cat_\infty(\mathscr X)_{/\mathscr C} \to \Cat_\infty(\mathscr X)$ together with its canonical map to $\mathscr C$ \cite[Prop.~1.2.13.8]{LurieHTT}. Since $\LFib_{/\mathscr C}$ is a full $\infty$-subcategory of $\Cat_\infty(\mathscr X)_{/\mathscr C}$, it suffices to show that, for any small diagram $D \colon \mathscr D \to \LFib_{/\mathscr C}$ with colimit $D(\infty)$ in $\Cat_\infty(\mathscr X)$, the functor $D(\infty) \to \mathscr C$ is a left fibration. For $\mathscr X = \mathcal S$, this is proved in \cite[Cor.~A.5]{Ramzi}. When $\mathscr X = \PSh(\mathcal I)$ is a presheaf $\infty$-topos for some small $\infty$-category $\mathcal I$, we have $\Cat_\infty(\mathscr X) \simeq \PSh(\mathcal I,\Cat_\infty)$, and the left fibration condition of \eqref{Dia left fibration} is checked pointwise for $i \in \mathcal I$. Since colimits are also computed pointwise \cite[Cor.~5.1.2.3]{LurieHTT}, we deduce the result for $\mathscr X = \PSh(\mathcal I)$ from the result for~$\mathcal S$. 

Finally, let $\mathscr X$ be an arbitrary $\infty$-topos $\mathscr X$, and choose an accessible left exact localisation $\iota \colon \mathscr X \leftrightarrows \PSh(\mathcal I) :\! L$ for some small $\infty$-category $\mathcal I$ \cite[Def.~6.1.0.4]{LurieHTT}. The adjunction \eqref{Dia adjunction cat} shows that $L \colon \Cat_\infty(\PSh(\mathcal I)) \to \Cat_\infty(\mathscr X)$ preserves colimits, so $D(\infty)$ is computed by taking the colimit in $\Cat_\infty(\PSh(\mathcal I))$ and applying $L$. Since each $D(d) \to \mathscr C$ for $d \in \mathscr D$ is a left fibration, the same goes for the colimit in $\Cat_\infty(\PSh(\mathcal I))_{/\mathscr C}$ by the case proved above and \ref{Par preserve left fibration}. The result for $D(\infty)$ follows since $L$ preserves left fibrations, again by \ref{Par preserve left fibration}.
\end{proof}

Cisinski pointed out that a fully synthetic proof of the lemma above can be found in \cite[Prop.~9.8.11]{CisinskiSynthetic} (phrased for right fibrations, which is equivalent by the auto-equivalence $(-)\op \colon \Cat_\infty(\mathscr X) \to \Cat_\infty(\mathscr X)$).

\begin{Par}\label{Par internal straightening}
Let $\mathscr X$ be an $\infty$-topos, and let $\mathscr C \in \Cat_\infty(\mathscr X)$. The internal straightening/ unstraightening theorem \cite[Thm.~4.5.1]{Martini} gives an equivalence
\[
\LFib_{/\mathscr C} \simeq \Fun_{\widehat{\Cat}_\infty(\mathscr X)}(\mathscr C,\mathscr X_{/-}),
\]
where $\mathscr X_{/-} \in \widehat{\Cat}_\infty(\mathscr X) = \Sh(\mathscr X,\widehat{\Cat}_\infty)$ denotes the large internal $\infty$-category given by $B \mapsto \mathscr X_{/B}$ (see \autoref{Ex large codomain fibration spaces}). In fact, the equivalence of [\emph{loc.~cit.}] is an internal variant of the one above, but we only need this external variant.
\end{Par}

\begin{Par}\label{Par Grothendieck construction}
Concretely, if $F \in \Fun_{\widehat{\Cat}_\infty(\mathscr X)}(\mathscr C,\mathscr X_{/-})$, the associated left fibration $\mathscr D \to \mathscr C$ is given by the fibre product
\[
\begin{tikzcd}[row sep=1em,column sep=1em]
\mathscr D \ar{r}\ar{d}\arrow[dr, phantom, "\lrcorner" , very near start, color=black] & (\mathscr X_{/-})_*\ar{d}\\
\mathscr C \ar{r} & \mathscr X_{/-}\punct{,}
\end{tikzcd}
\]
where $(\mathscr X_{/-})_* \in \widehat{\Cat}_\infty(\mathscr X) \simeq \Sh(\mathscr X,\widehat{\Cat}_\infty)$ is the sheaf of large $\infty$-categories $\Sh(\mathscr X_{/-},\mathcal S_*)$ of \autoref{Prop sheaf of categories}, whose value on $U \in \mathscr X$ is $\Sh(\mathscr X_{/U},\mathcal S_*) \simeq \mathscr X_{/U} \otimes \mathcal S_* \simeq (\mathscr X_{/U})_*$ by \autoref{Rmk presentable tensor} and \makebox{\cite[Ex.~4.8.1.21]{LurieHA}.} Thus, if $U \in \mathscr X$, then $\mathscr D(U) \simeq \mathscr C(U) \times_{\mathscr X_{/U}} (\mathscr X_{/U})_*$ consists of a pair of an object $x \in \mathscr C(U)$ and a global section $s$ of $F(x) \in \mathscr X_{/U}$.
\end{Par}

\section{Scheme-theoretic classification of w-strictly local schemes}\label{Sec appendix B}
If $X$ is a scheme, then the acyclic (resp.~w-strictly local) objects in $\AffProEt_{/X}$ are classified by geometric morphisms (resp.~locally coherent geometric morphisms) by \autoref{Thm equivalence} (resp.~\autoref{Cor loccoh}). In the w-strictly local case, we give a third perspective that is more scheme-theoretic in nature, and closely related to the construction of enough w-strictly local schemes in $\AffProEt_{/X}$ \cite[Lem.~2.4.9]{BhattScholze}. See \autoref{Prop Xcons} for the statement, and \autoref{Rmk construction wsl} for a comparison with the construction of \cite{BhattScholze}. The contents of this section do not play a role in the main body of the paper, and are merely added for completeness.

\begin{BDef}\label{Def Xcons}
Let $X = \Spec A$ be an affine scheme. 
Define $A\cons$ as the tensor product $\bigotimes_{f \in A} (A_f \times A/(f))$ in $A$-algebras, and write $X\cons = \Spec A\cons$ for the corresponding product $\prod_{f \in A} D(f) \amalg V(f)$ in $\Sch_{/X}$. Note that $A\cons$ agrees with the ring $A^Z/I_{A^Z}$ of \cite[\S2.2]{BhattScholze}; see \cite[Tag \href{https://stacks.math.columbia.edu/tag/0975}{0975}]{Stacks}. In particular, by \cite[Ex.~2.2.5]{BhattScholze}, the monomorphism $X\cons \hookrightarrow X$ is the terminal map from a zero-dimensional reduced scheme to $X$. Note that $X\cons \hookrightarrow X$ is a bijection on points, and identifies $\lvert X\cons \rvert$ with the constructible topology on $X$ (this justifies the notation $X\cons$).

If $X = \Spec A$ is an affine scheme and $U = D(f) \subseteq X$ is a standard affine open, then $U\cons \cong X\cons \times_X U$, for instance because they satisfy the same universal property for maps from zero-dimensional reduced schemes to $X$. Thus, we may extend the definition to arbitrary schemes: the association $\operatorname{AffOpen}(X)\op \to \CRing$ taking $U$ to $\mathscr O_X(U)\cons$ extends uniquely to a quasi-coherent sheaf of $\mathscr O_X$-algebras $\mathscr O_{X,\operatorname{cons}}$, and we define $X\cons$ as $\mathbf{Spec}_X \mathscr O_{X,\operatorname{cons}}$. The underlying topological space of $X\cons$ local constructible topology of \autoref{Ex lcons}. The scheme $X\cons$ is zero-dimensional and reduced (but not necessarily affine), so in particular, every qcqs open $V \subseteq X\cons$ is affine. The map $X\cons \to X$ is an affine monomorphism, and every map $Z \to X$ from a zero-dimensional scheme $Z$ factors uniquely via $X\cons$.
\end{BDef}

\begin{BPar}\label{Par Henselisation coincides}
Let $Y \to X$ be a morphism of schemes. Analogously to \cite[Def.~2.2.10]{BhattScholze}, we would like to define the Henselisation $\Hens_X(Y) \in \AffProEt_{/X}$ as the limit over all factorisations $Y \to U \to X$ with $U \in \AffEt_{/X}$. However, this only makes sense when the functor $\Hom_X(Y,-) \colon \AffEt_{/X} \to \Set$ is filtering (i.e., its category of elements is filtered \cite[Ch.~VII, \S6, Def.~2]{MacLaneMoerdijk}). This is true when $X$ is affine, since $\AffEt_{/X}$ has finite limits in that case, and $\Hom_X(Y,-)$ preserves them. On the other hand, this fails for the identity $X \to X$ if $X$ is a projective variety of positive dimension, for then the category of elements is empty.

On the other hand, if $Y$ is strictly profinite (see \autoref{Def strictly profinite}) then $\Sh(Y_\et) \simeq \Sh(S)$ where $S = \lvert Y \rvert$ (see \autoref{Lem strictly profinite}), and $\Hens_X(Y)$ is the topos-theoretic Henselisation $X_{(s)}$ of the geometric morphism $s_* \colon \Sh(S) \simeq \Sh(Y_\et) \to \Sh(X_\et)$ (see \autoref{Def Henselisation}). Indeed, as left exact functors $\AffEt_{/X} \to \Set$, both are given by the association taking $U$ to $\Hom_X(Y,U) \simeq \Hom_Y(Y,s^*U) = \Gamma_*s^*U$.

In particular, $\Hens_X(Y)$ is always representable in this case by \autoref{Lem Henselisation representable}. Moreover, we conclude from \autoref{Lem Henselisation acyclic} and \autoref{Lem coherent} that $\Hens_X(Y)$ is \makebox{w-strictly} local, since the geometric morphism $\Sh(Y_\et) \to \Sh(X_\et)$ induced by a morphism of schemes $Y \to X$ is always locally coherent. Moreover, for any map $Y' \to Y$ of strictly profinite schemes over $X$, the map $\Hens_X(Y') \to \Hens_X(Y) \times_{\lvert Y\rvert} \lvert Y'\rvert$ is an isomorphism by \autoref{Lem Henselisation base change}. Thus, the map $\Hens_X(Y') \to \Hens_X(Y)$ is w-local by \autoref{Lem acyc profinite}.
\end{BPar}

\begin{BDef}
Let $X$ be a scheme. Write $\AffProEt_{/X}^{\wsl,\wl}$ for the full subcategory of $\AffProEt_{/X}$ whose objects are w-strictly local schemes and whose morphisms are w-local morphisms. Write $\Sch_{/X}^{\operatorname{sp}}$ for the full subcategory of $\Sch_{/X}$ on strictly profinite schemes.
\end{BDef}

Thus, \ref{Par Henselisation coincides} constructs a functor $\Hens_X(-) \colon \Sch_{/X}^{\operatorname{sp}} \to \AffProEt_{/X}^{\wsl,\wl}$. By definition, there is a natural transformation $Y \to \Hens_X(Y)$ induced by the factorisations $Y \to U \to X$ for $U \in \AffEt_{/X}$.

\begin{BLemma}\label{Lem homeomorphism}
Let $X$ be a scheme, and let $Y \in \Sch_{/X}^{\operatorname{sp}}$. Then the map $Y \to W = \Hens_X(Y)$ induces a homeomorphism $f \colon Y \to W^\cl$. Moreover, the following are equivalent:
\begin{enumerate}
\item $f$ is an isomorphism of schemes;
\item $Y \to X$ induces separable algebraic maps on residue fields;
\item $Y \to X\cons$ is in $\AffProEt_{/X\cons}$.
\end{enumerate}
In particular, any strictly profinite scheme $Y$ that is weakly \'etale over $X\cons$ is in $\AffProEt_{/X\cons}$.
\end{BLemma}

\begin{proof}
By \autoref{Prop equivalence S} and \autoref{Rmk wsl coherent}, we can recover the geometric morphism $\Sh(Y_\et) \to \Sh(X_\et)$ from $W$ by taking the geometric morphism $\Sh(W_\et^\cl) \to \Sh(X_\et)$ induced by the morphism of schemes $W^\cl \to X$. In the factorisation $\Sh(Y_\et) \to \Sh(W_\et^\cl) \to \Sh(X_\et)$, we conclude that the first map is an equivalence, hence $f$ is a homeomorphism of strictly profinite schemes by \autoref{Lem strictly profinite} and \cite[Exp.~IV, Cor.~4.2.4(a)]{SGA4I}.

If $f$ is an isomorphism of schemes, then $Y \to X$ is the composition of the closed immersion $W^\cl \hookrightarrow W$ with the weakly \'etale map $W \to X$, hence induces separable algebraic maps on residue fields. Conversely, if $Y \to X$ induces separable algebraic maps on residue fields, then $Y \to W^\cl$ induces isomorphisms on residue fields: both are separable algebraic over the residue fields of $X$, and both are separably algebraically closed since $Y$ and $W^\cl$ are strictly profinite. Thus, $Y \to W^\cl$ induces isomorphisms on local rings at all points, hence is an isomorphism of ringed spaces since it is also a homeomorphism. This proves \makebox{(a) $\Leftrightarrow$ (b).}

Since $Y \to X$ factors over $X\cons$ and $X\cons \to X$ induces isomorphisms of residue fields, implication (c) $\Rightarrow$ (b) follows from \cite[Tag \href{https://stacks.math.columbia.edu/tag/092R}{092R}]{Stacks}. Conversely, assume $Y \to X$ induces separable algebraic maps on residue fields; then the same holds for $Y \to X\cons$. Thus, we may replace $X$ with $X\cons$ and assume $X$ is $0$-dimensional and reduced. By (b) $\Rightarrow$ (a), we see that $Y \to W^\cl$ is an isomorphism. But $W$ is in $\AffProEt_{/X}$, hence absolutely flat \cite[Tag \href{https://stacks.math.columbia.edu/tag/092I}{092I}(1)]{Stacks}. Thus, $W^\cl \to W$ is an isomorphism, which shows that $Y \cong W$ is in $\AffProEt_{/X}$. This proves (b) $\Rightarrow$ (c). The final statements follows since such a scheme $Y$ satisfies condition (b), again by \cite[Tag \href{https://stacks.math.columbia.edu/tag/092R}{092R}]{Stacks}.
\end{proof}

\begin{BProp}\label{Prop Xcons}
Let $X$ be a scheme. Then the functors
\begin{align*}
\AffProEt_{/X}^{\wsl,\wl} &\leftrightarrows \AffProEt_{/X\cons}^{\operatorname{sp}} \\
W &\mapsto W^\cl\\
\Hens_X(V) &\mapsfrom V
\end{align*}
are mutually inverse equivalences.
\end{BProp}

The right hand side is easier to understand, since $X\cons$ is reduced and $0$-dimensional.

\begin{proof}
The functor $V \mapsto \Hens_X(V)$ is defined in \ref{Par Henselisation coincides}. Conversely, if $W \in \AffProEt_{/X}^{\wsl}$, then the map $W^\cl \to X\cons$ is in $\AffProEt_{/X\cons}$ by \autoref{Lem homeomorphism}. By definition, any w-local morphism $W' \to W$ maps $(W')^\cl$ to $W^\cl$, which defines the opposite functor.

If $V \in \AffProEt_{/X\cons}^{\operatorname{sp}}$ and $W \in \AffProEt_{/X}^\wsl$, then every map $V \to W$ factors uniquely as $V \to \Hens_X(V) \to W$ by the definition of $\Hens_X(-)$. Since $V$ identifies with $\Hens_X(V)^\cl$ by \autoref{Lem homeomorphism}, the induced map $\Hens_X(V) \to W$ is w-local if and only if the image of $V \to W$ lands in $W^\cl$. This gives an adjunction
\[
\Hom_X^\wl(\Hens_X(V),W) \simeq \Hom_{X\cons}(V,W^\cl).
\]
The counit $\Hens_X(W^\cl) \to W$ is an isomorphism by \autoref{Cor geometric morphism} and \autoref{Rmk wsl coherent}, and the unit $V \to \Hens_X(V)^\cl$ is an isomorphism by \autoref{Lem homeomorphism}.
\end{proof}

\begin{BCor}\label{Cor core}
If $X$ is a scheme, the inclusion $X\cons \hookrightarrow X$ induces an equivalence
\[
\GAL(X\cons) \simeq \Core(\GAL(X)).
\]
\end{BCor}

\begin{proof}
By \autoref{Cor loccoh}, the condensed category $\GAL(X)$ is the straightening of the fibred category $\AffProEt_{/X}^\wsl$. The unstraightening of the core is then given by restricting to the cartesian morphisms in $\AffProEt_{/X}^\wsl$. By \ref{Par cartesian}, a morphism $V \to W$ in $\AffProEt_{/X}^\wsl$ is cartesian if and only if it is is w-local. Thus, the unstraightening of $\Core(\GAL(X))$ is given by $\AffProEt_{/X}^{\wsl,\wl}$, which by \autoref{Prop Xcons} is canonically identified with $\AffProEt_{/X\cons}^\wsl$. Another application of \autoref{Cor loccoh} gives the result.
\end{proof}

The corollary can also be deduced from the stratified methods in \cite{BGH}: a qcqs scheme $X$ is canonically stratified by its profinite Zariski poset $(\lvert X \rvert\cons,\leq)$, and the diagram
\[
\begin{tikzcd}[row sep=1.3em, column sep=1.2em]
\GAL(X\cons) \ar{r}\ar{d} & \GAL(X)\ar{d} \\
(\lvert X\rvert\cons,=) \ar[hook]{r} & (\lvert X \rvert\cons,\leq)
\end{tikzcd}
\]
is a pullback, since both vertical maps are conservative with fibre above $x \in X$ given by $\Gal(\overline{\kappa(x)}/\kappa(x))$. On the other hand, since the vertical maps are conservative, the fibre product identifies with $\Core(\GAL(X))$.

\begin{BRmk}\label{Rmk Xcons size}
If $\kappa > \size(X)$, then the equivalence of \autoref{Prop Xcons} restricts to an equivalence $\AffProEt_{/X}^{\kappa,\wsl,\wl} \simeq \AffProEt_{/X\cons}^{\kappa,\wsl}$. Indeed, if $\size(W) < \kappa$, then $\size(W^\cl) < \kappa$. Conversely, if $\size(V) < \kappa$, then $\size(\lvert V \rvert) < \kappa$ by \autoref{Lem pi0 kappa}. Since $\Hens_X(V)$ agrees with the Henselisation $X_{(s)}$ along the geometric morphism $s_* \colon \Sh(\lvert V \rvert) \to \Sh(X_\et)$, we see that $\size(\Hens_X(V)) < \kappa$ by \autoref{Lem size Hens}.

Likewise, \autoref{Cor core} restricts to an equivalence $\GAL_\kappa(X\cons) \simeq \Core(\GAL_\kappa(X))$ if $\kappa > \size(X)$ (which automatically implies $\kappa > \size(X\cons)$).
\end{BRmk}

\begin{BRmk}\label{Rmk construction wsl}
The above results give a slight simplification of the construction of enough w-strictly local schemes in $\AffProEt_{/X}$ \cite[Cor.~2.2.14]{BhattScholze}: start with a cover $\{U_i \to X\cons\}$ of $X\cons$ by affine opens (any qcqs open of $X$ gives an affine open in $X\cons$), choose $V_i \in \AffProEt_{/U_i}$ strictly profinite and surjecting onto $U_i$ as in \cite[Lem.~2.2.7]{BhattScholze}, and take $W_i = \Hens_X(V_i)$. The method of \cite[\S2.2]{BhattScholze} differs from this in that it first replaces $X\cons$ by its Zariski Henselisation $X^Z$, which is a little harder to construct than $X\cons$. This first step is needed in \cite{BhattScholze} to prove that the obtained Henselisations actually form Henselian pairs, which follows from \cite[Lem.~2.2.12]{BhattScholze}. By contrast, our method proves directly that the Henselisation is w-strictly local (\autoref{Lem Henselisation acyclic} and \autoref{Lem coherent}, which can also be done directly via scheme theory instead of topos theory), hence a Henselian pair by \autoref{Lem wsl Henselian}.

Moreover, \autoref{Prop Xcons} shows that all w-strictly local schemes in $\AffProEt_{/X}$ are constructed in this way.
\end{BRmk}

\bibliographystyle{alphaurledit}
{\phantomsection\footnotesize\bibliography{Exodromy.bib}}

@article {Aoki,
    AUTHOR = {Aoki, Ko},
     TITLE = {Tensor triangular geometry of filtered objects and sheaves},
   JOURNAL = {Math. Z.},
  FJOURNAL = {Mathematische Zeitschrift},
    VOLUME = {303},
      YEAR = {2023},
    NUMBER = {3 Paper No. 62},
       DOI = {10.1007/s00209-023-03210-z},
}

@article {Artin,
    AUTHOR = {Artin, M.},
     TITLE = {On the joins of {H}ensel rings},
   JOURNAL = {Adv. Math.},
  FJOURNAL = {Advances in Mathematics},
    VOLUME = {7},
      YEAR = {1971},
     PAGES = {282--296},
       DOI = {10.1016/S0001-8708(71)80007-5},
}

@book {ArtinMazur,
    AUTHOR = {Artin, M. and Mazur, B.},
     TITLE = {{\'E}tale homotopy},
    SERIES = {Lecture Notes in Mathematics},
SHORTSERIES = {LNM},
    NUMBER = {100},
 PUBLISHER = {Springer-Verlag},
      YEAR = {1986},
       DOI = {10.1007/BFb0080957},
}

@misc {BGH,
    AUTHOR = {Barwick, Clark and Glasman, Saul and Haine, Peter J.},
     TITLE = {Exodromy},
      YEAR = {2018},
HOWPUBLISHED = {Preprint},
    EPRINT = {1807.03281},
EPRINTTYPE = {arxiv},
}

@misc {Pyknotic,
    AUTHOR = {Barwick, Clark and Haine, Peter J.},
     TITLE = {Pyknotic objects, {I}. {B}asic notions},
      YEAR = {2019},
  PUBSTATE = {Preprint},
    EPRINT = {1904.09966},
EPRINTTYPE = {arxiv},
}

@article {BhattScholze,
    AUTHOR = {Bhatt, Bhargav and Scholze, Peter},
     TITLE = {The pro-\'{e}tale topology for schemes},
   JOURNAL = {Ast\'{e}risque},
  FJOURNAL = {Ast\'{e}risque},
    VOLUME = {369},
      YEAR = {2015},
     PAGES = {99--201},
       DOI = {10.24033/ast.960},
}

@misc {CisinskiSynthetic,
    AUTHOR = {Cisinski, Denis-Charles and Cnossen, Bastiaan and Nguyen, Kim and Walde, Tashi},
     TITLE = {Synthetic category theory},
HOWPUBLISHED = {Book in preparation},
      YEAR = {2026},
      NOTE = {Version of 11 May 2026},
       URL = {https://cisinski.app.uni-regensburg.de/publikationen.html},
}

@misc {ClausenScholze,
    AUTHOR = {Clausen, Dustin and Scholze, Peter},
     TITLE = {Lectures on Condensed Mathematics},
      YEAR = {2026},
HOWPUBLISHED = {Preprint/lecture notes},
    EPRINT = {2605.03658},
EPRINTTYPE = {arxiv},
}

@misc {vDdBExodromy,
      AUTHOR = {van Dobben {de} Bruyn, Remy},
       TITLE = {Grothendieck {G}alois theory and \'etale exodromy},
        YEAR = {2024},
HOWPUBLISHED = {Preprint},
      EPRINT = {2410.06278},
  EPRINTTYPE = {arxiv},
}

@book {Friedlander,
    AUTHOR = {Friedlander, Eric M.},
     TITLE = {\'{E}tale homotopy of simplicial schemes},
    SERIES = {Annals of Mathematics Studies},
    NUMBER = {104},
 PUBLISHER = {Princeton University Press},
      YEAR = {1982},
       DOI = {10.1515/9781400881499},
}

@article {GepnerHaugsengNikolaus,
    AUTHOR = {Gepner, David and Haugseng, Rune and Nikolaus, Thomas},
     TITLE = {Lax colimits and free fibrations in {$\infty$}-categories},
   JOURNAL = {Doc. Math.},
  FJOURNAL = {Documenta Mathematica},
    VOLUME = {22},
      YEAR = {2017},
     PAGES = {1225--1266},
       DOI = {10.25537/dm.2017v22.1225-1266},
}

@misc {HaineClassifying,
    AUTHOR = {Haine, Peter J.},
     TITLE = {Classifying anima of condensed $\infty$-categories of points},
      YEAR = {2026},
HOWPUBLISHED = {Preprint},
    EPRINT = {2602.21330},
EPRINTTYPE = {arxiv},
}

@misc{HaineDescent,
    AUTHOR = {Haine, Peter J.},
     TITLE = {Descent for sheaves on compact {H}ausdorff spaces},
      YEAR = {2022},
HOWPUBLISHED = {Preprint},
    EPRINT = {2210.00186},
EPRINTTYPE = {arxiv},
}

@article {HaineBaseChange,
    AUTHOR = {Haine, Peter J.},
     TITLE = {From nonabelian basechange to basechange with coefficients},
   JOURNAL = {J. Pure Appl. Algebra},
  FJOURNAL = {Journal of Pure and Applied Algebra},
    VOLUME = {229},
      YEAR = {2025},
    NUMBER = {7, Paper No. 107993},
       DOI = {10.1016/j.jpaa.2025.107993},
}

@misc{HHLMMW,
    AUTHOR = {Haine, Peter J. and Holzschuh, Tim and Lara, Marcin and Mair, Catrin and Martini, Louis and Wolf, Sebastian},
     TITLE = {The condensed homotopy type of a scheme},
      YEAR = {2025},
HOWPUBLISHED = {Preprint},
      NOTE = {With an appendix by Bogdan Zavyalov},
    EPRINT = {2510.07443},
EPRINTTYPE = {arxiv},
}

@article {HemoRicharzScholbach,
    AUTHOR = {Hemo, Tamir and Richarz, Timo and Scholbach, Jakob},
     TITLE = {Constructible sheaves on schemes},
   JOURNAL = {Adv. Math.},
  FJOURNAL = {Advances in Mathematics},
    VOLUME = {429},
      YEAR = {2023},
      NOTE = {Paper No. 109179},
       DOI = {10.1016/j.aim.2023.109179},
}

@article {HesselholtPstragowski,
    AUTHOR = {Hesselholt, Lars and Pstr\k{a}gowski, Piotr},
     TITLE = {Dirac geometry {I}: {C}ommutative algebra},
   JOURNAL = {Peking Math. J.},
  FJOURNAL = {Peking Mathematical Journal},
    VOLUME = {8},
      YEAR = {2025},
    NUMBER = {3},
     PAGES = {405--480},
       DOI = {10.1007/s42543-023-00072-6},
}

@article {Hoyois,
    AUTHOR = {Hoyois, Marc},
     TITLE = {Higher {G}alois theory},
   JOURNAL = {J. Pure Appl. Algebra},
  FJOURNAL = {Journal of Pure and Applied Algebra},
    VOLUME = {222},
      YEAR = {2018},
    NUMBER = {7},
     PAGES = {1859--1877},
       DOI = {10.1016/j.jpaa.2017.08.010},
}

@article {JohnstoneMoerdijk,
    AUTHOR = {Johnstone, P. T. and Moerdijk, I.},
     TITLE = {Local maps of toposes},
   JOURNAL = {Proc. London Math. Soc. (3)},
  FJOURNAL = {Proceedings of the London Mathematical Society. Third Series},
    VOLUME = {58},
      YEAR = {1989},
    NUMBER = {2},
     PAGES = {281--305},
       DOI = {10.1112/plms/s3-58.2.281},
}

@incollection {JoyalTierney,
    AUTHOR = {Joyal, Andr\'e and Tierney, Myles},
     TITLE = {Quasi-categories vs {S}egal spaces},
 BOOKTITLE = {Categories in algebra, geometry and mathematical physics},
    SERIES = {Contemp. Math.},
    NUMBER = {431},
     PAGES = {277--326},
 PUBLISHER = {American Mathematical Society},
      YEAR = {2007},
       DOI = {10.1090/conm/431/08278},
}

@book {KashiwaraSchapira,
    AUTHOR = {Kashiwara, Masaki and Schapira, Pierre},
     TITLE = {Categories and sheaves},
    SERIES = {Grundlehren der mathematischen Wissenschaften},
    NUMBER = {332},
 PUBLISHER = {Springer-Verlag},
      YEAR = {2006},
       DOI = {10.1007/3-540-27950-4},
}

@misc {Kerodon,
 SHORTHAND = {Kerodon},
  SORTNAME = {Kerodon},
    AUTHOR = {Lurie, Jacob},
     TITLE = {Kerodon},
       URL = {https://www.kerodon.net},
      YEAR = {2018--2026},
}

@misc {LurieGoodwillie,
    AUTHOR = {Lurie, Jacob},
     TITLE = {$(\infty,2)$-{C}ategories and the {G}oodwillie calculus {I}},
HOWPUBLISHED = {Preprint},
      YEAR = {2009},
EPRINTTYPE = {arxiv},
    EPRINT = {0905.0462},
}

@misc {LurieHA,
   SHORTHAND = {HA},
    SORTNAME = {HA},
      AUTHOR = {Lurie, Jacob},
       TITLE = {Higher algebra},
        YEAR = {2017},
HOWPUBLISHED = {Book in preparation},
        NOTE = {Version of 18 September 2017},
         URL = {https://www.math.ias.edu/~lurie/papers/HA.pdf},
}

@misc {LurieDAGXIII,
    AUTHOR = {Lurie, Jacob},
     TITLE = {Derived algebraic geometry {XIII}: Rational and $p$-adic homotopy theory},
HOWPUBLISHED = {Preprint},
      YEAR = {2011},
       URL = {https://www.math.ias.edu/~lurie/papers/DAG-XIII.pdf},
      NOTE = {Version of 15 December 2011},
}

@book {LurieHTT,
 SHORTHAND = {HTT},
  SORTNAME = {HTT},
    AUTHOR = {Lurie, Jacob},
     TITLE = {Higher topos theory},
    SERIES = {Annals of Mathematics Studies},
    NUMBER = {170},
 PUBLISHER = {Princeton University Press},
      YEAR = {2009},
       DOI = {10.1515/9781400830558},
}

@misc {LurieSAG,
 SHORTHAND = {SAG},
  SORTNAME = {SAG},
    AUTHOR = {Lurie, Jacob},
     TITLE = {Spectral algebraic geometry},
      YEAR = {2018},
       URL = {https://www.math.ias.edu/~lurie/papers/SAG-rootfile.pdf},
HOWPUBLISHED = {Book in preparation},
      NOTE = {Version of 3 February 2018},
}

@misc {LurieUltra,
 SHORTHAND = {Lur18},
    AUTHOR = {Lurie, Jacob},
     TITLE = {Ultracategories},
      YEAR = {around 2018},
HOWPUBLISHED = {Preprint},
       URL = {https://www.math.ias.edu/~lurie/papers/Conceptual.pdf},
}

@book {MacLaneMoerdijk,
    AUTHOR = {Mac Lane, Saunders and Moerdijk, Ieke},
     TITLE = {Sheaves in geometry and logic},
    SERIES = {Universitext},
 PUBLISHER = {Springer-Verlag},
      YEAR = {1994},
       DOI = {10.1007/978-1-4612-0927-0},
}

@misc {Mair,
    AUTHOR = {Mair, Catrin},
     TITLE = {On {G}alois categories and condensed contractible schemes},
      YEAR = {2026},
HOWPUBLISHED = {Preprint},
    EPRINT = {2605.10358},
EPRINTTYPE = {arxiv},
}

@phdthesis {MairThesis,
    AUTHOR = {Mair, Catrin},
     TITLE = {The Role of Condensed Mathematics in Homotopy Theory},
    SCHOOL = {Technische Universit\"at Darmstadt},
      YEAR = {2025},
       DOI = {10.26083/tuprints-00029744},
}

@book {MakkaiPare,
    AUTHOR = {Makkai, Michael and Par\'{e}, Robert},
     TITLE = {Accessible categories: the foundations of categorical model theory},
    SERIES = {Contemporary Mathematics},
    NUMBER = {104},
 PUBLISHER = {American Mathematical Society},
      YEAR = {1989},
       DOI = {10.1090/conm/104},
}

@phdthesis {Mann,
    AUTHOR = {Mann, Lucas},
     TITLE = {A $p$-adic 6-functor formalism in rigid analytic geometry},
    SCHOOL = {Rheinische Friedrich-Wilhelms-Universit\"at Bonn},
      YEAR = {2022},
       URL = {https://hdl.handle.net/20.500.11811/10125},
}

@misc {Martini,
    AUTHOR = {Martini, Louis},
     TITLE = {Yoneda's lemma for internal higher categories},
      YEAR = {2021},
HOWPUBLISHED = {Preprint},
    EPRINT = {2103.17141},
EPRINTTYPE = {arxiv},
}

@misc {MartiniWolf,
    AUTHOR = {Martini, Louis and Wolf, Sebastian},
     TITLE = {Presentability and topoi in internal higher category theory},
      YEAR = {2022},
HOWPUBLISHED = {Preprint},
EPRINTTYPE = {arxiv},
    EPRINT = {2209.05103},
}

@incollection {Olivier,
    AUTHOR = {Olivier, Jean Pierre},
     TITLE = {Fermeture int\'{e}grale et changements de base absolument plats},
    SERIES = {Colloque d'alg\`ebre commutative},
    NUMBER = {4},
 PUBLISHER = {Univ. Rennes},
      YEAR = {1972},
       URL = {https://www.numdam.org/item/?id=PSMIR_1972___4_A9_0},
}

@article {PortaYu,
    AUTHOR = {Porta, Mauro and Yu, Tony Yue},
     TITLE = {Higher analytic stacks and {GAGA} theorems},
   JOURNAL = {Adv. Math.},
  FJOURNAL = {Advances in Mathematics},
    VOLUME = {302},
      YEAR = {2016},
     PAGES = {351--409},
       DOI = {10.1016/j.aim.2016.07.017},
}

@article {Ramzi,
    AUTHOR = {Ramzi, Maxime},
     TITLE = {A monoidal {G}rothendieck construction for $\infty$-categories},
   JOURNAL = {Nagoya Math. J.},
  FJOURNAL = {Nagoya Mathematical Journal},
    VOLUME = {261},
      YEAR = {2026},
    NUMBER = {e8},
       DOI = {10.1017/nmj.2025.10086},
}

@misc {RezkNotes,
    AUTHOR = {Rezk, Charles},
     TITLE = {Introduction to quasicategories},
HOWPUBLISHED = {Lecture notes},
      YEAR = {2022},
       URL = {https://rezk.web.illinois.edu/quasicats.pdf},
      NOTE = {Version of 29 November 2022}
}

@article {Schroer,
    AUTHOR = {Schr\"{o}er, Stefan},
     TITLE = {Geometry on totally separably closed schemes},
   JOURNAL = {Algebra Number Theory},
  FJOURNAL = {Algebra \& Number Theory},
    VOLUME = {11},
      YEAR = {2017},
    NUMBER = {3},
     PAGES = {537--582},
       DOI = {10.2140/ant.2017.11.537},
}

@book {SGA1,
 SHORTHAND = {SGA1},
  SORTNAME = {SGA1},
    AUTHOR = {Grothendieck, Alexander},
     TITLE = {S\'eminaire de G\'eom\'etrie Alg\'ebrique du Bois-Marie 1960--1961 - 
              Rev\^etements \'etales et groupe fondamental (SGA 1)},
    SERIES = {Lecture Notes in Mathematics},
    NUMBER = {224},
 PUBLISHER = {Springer-Verlag},
      YEAR = {1971},
       DOI = {10.1007/BFb0058656},
}

@book {SGA3II,
 SHORTHAND = {SGA3$_{\text{II}}$},
  SORTNAME = {SGA3II},
LABELWIDTH = {SGA3II},
    AUTHOR = {Demazure, M. and Grothendieck, A.},
     TITLE = {S\'eminaire de g\'eom\'etrie alg\'ebrique du Bois-Marie 1962--1964 - 
              Sch\'emas en groupes (SGA 3). {T}ome 2},
    SERIES = {Lecture Notes in Mathematics},
    NUMBER = {152},
 PUBLISHER = {Springer-Verlag},
      YEAR = {1970},
       DOI = {10.1007/BFb0059005},
}

@book {SGA4I,
 SHORTHAND = {SGA4$_{\text{I}}$},
  SORTNAME = {SGA41},
LABELWIDTH = {SGA4I},
    AUTHOR = {Artin, M. and Grothendieck, A. and Verdier, J. L.},
     TITLE = {S\'eminaire de g\'eom\'etrie alg\'ebrique du Bois-Marie 1963--1964 - 
             Th\'eorie des topos et cohomologie \'etale des sch\'emas (SGA 4). {T}ome 1},
    SERIES = {Lecture Notes in Mathematics},
    NUMBER = {269},
 PUBLISHER = {Springer-Verlag},
      YEAR = {1972},
       DOI = {10.1007/BFb0081551},
}

@book {SGA4II,
 SHORTHAND = {SGA4$_{\text{II}}$},
  SORTNAME = {SGA42},
LABELWIDTH = {SGA4II},
    AUTHOR = {Artin, M. and Grothendieck, A. and Verdier, J. L.},
     TITLE = {S\'eminaire de G\'eom\'etrie Alg\'ebrique du Bois-Marie 1963--1964 -
              Th\'eorie des topos et cohomologie \'etale des sch\'emas (SGA 4). {T}ome 2},
    SERIES = {Lecture Notes in Mathematics},
    NUMBER = {270},
 PUBLISHER = {Springer-Verlag},
      YEAR = {1972},
       DOI = {10.1007/BFb0061319},
}

@book {SGA4III,
 SHORTHAND = {SGA4$_{\text{III}}$},
  SORTNAME = {SGA43},
LABELWIDTH = {SGA4III},
    AUTHOR = {Artin, M. and Grothendieck, A. and Verdier, J. L.},
     TITLE = {S\'eminaire de G\'eom\'etrie Alg\'ebrique du Bois-Marie 1963--1964 - 
              Th\'eorie des topos et cohomologie \'etale des sch\'emas (SGA 4). {T}ome 3},
    SERIES = {Lecture Notes in Mathematics},
SHORTSERIES = {LNM},
    NUMBER = {305},
 PUBLISHER = {Springer-Verlag},
      YEAR = {1973},
       DOI = {10.1007/BFb0070714},
}

@misc {Stacks,
   SHORTHAND = {Stacks},
    SORTNAME = {Stacks},
      AUTHOR = {{The Stacks Project Authors}},
       TITLE = {The Stacks Project},
         URL = {https://stacks.math.columbia.edu},
        YEAR = {2005--2026},
}

@article {Treumann,
    AUTHOR = {Treumann, David},
     TITLE = {Exit paths and constructible stacks},
   JOURNAL = {Compos. Math.},
  FJOURNAL = {Compositio Mathematica},
    VOLUME = {145},
      YEAR = {2009},
    NUMBER = {6},
     PAGES = {1504--1532},
       DOI = {10.1112/S0010437X09004229},
}

@phdthesis {TreumannThesis,
    AUTHOR = {Treumann, David},
     TITLE = {Exit paths and constructible stacks},
    SCHOOL = {Princeton University},
 PUBLISHER = {ProQuest LLC},
      YEAR = {2007},
     PAGES = {84},
       URL = {https://www.proquest.com/dissertations-theses/exit-paths-constructible-stacks/docview/304828996/se-2},
}

@article {Wall,
    AUTHOR = {Wall, C. T. C.},
     TITLE = {Finiteness conditions for {CW}-complexes},
   JOURNAL = {Ann. of Math. (2)},
  FJOURNAL = {Annals of Mathematics. Second Series},
    VOLUME = {81},
    NUMBER = {1},
      YEAR = {1965},
     PAGES = {56--69},
       DOI = {10.2307/1970382},
}

@article {Wolf,
    AUTHOR = {Wolf, Sebastian},
     TITLE = {The pro-\'{e}tale topos as a category of pyknotic presheaves},
   JOURNAL = {Doc. Math.},
  FJOURNAL = {Documenta Mathematica},
    VOLUME = {27},
      YEAR = {2022},
     PAGES = {2067--2106},
       DOI = {10.4171/dm/x26},
}

@phdthesis {WolfThesis,
    AUTHOR = {Wolf, Sebastian},
     TITLE = {Internal higher categories and applications},
    SCHOOL = {Universit\"at Regensburg},
      YEAR = {2025},
       URL = {https://epub.uni-regensburg.de/76465/},
}
\addcontentsline{toc}{section}{References}

\end{document}